\documentclass[a4paper,10pt,final ]{amsart}
\usepackage[utf8x]{inputenc}
\usepackage{amsmath, amsthm, amssymb, enumerate, scalerel, tikz,  color, accents,
  fancybox, rotating, comment, colortbl, stmaryrd, hyperref,  cleveref, setspace, tikz-3dplot, }
\title{Symmetries of {\texorpdfstring{$\mathfrak{\lowercase{gl}}_N$}{\lowercase{gl}(N)}}-foams}
\author[Y.{} Qi]{You Qi}
\address[Y.{} Q.]{Department of Mathematics, University of Virginia,
  Charlottesville, VA 22904, USA}
\email{\href{mailto:yq2dw@virginia.edu}{yq2dw@virginia.edu}}
\author[L.-H.{} Robert]{Louis-Hadrien Robert}
 \address[L.-H.{} R.]{Université Clermont Auvergne, LMBP, Campus des Cézeaux, 3
   place Vasarely, TSA 60026, CS 60026, 63178 Aubière Cedex, France}
  \email{\href{mailto:louis\_hadrien.robert@uca.fr}{louis\_hadrien.robert@uca.fr}}
\address[L.-H.{} R.]{Laboratoire Cogitamus, \href{https://www.cogitamus.fr/indexen.html}{https://www.cogitamus.fr/}}
 \author[J.{} Sussan]{Joshua Sussan}
 \address[J.{} S.]{Department of Mathematics, CUNY Medgar Evers, Brooklyn, NY,
   11225, USA}
  \email{\href{mailto:jsussan@mec.cuny.edu}{jsussan@mec.cuny.edu}}
   \address[J.S.]{Mathematics Program, The Graduate Center, New York, NY,
   10016, USA}
  \email{\href{mailto:jsussan@gc.cuny.edu}{jsussan@gc.cuny.edu}}
 \author[E.{} Wagner]{Emmanuel Wagner}
 \address[E.{} W.]{Univ Paris Cit\'e, IMJ-PRG, Univ Paris Sorbonne, UMR 7586 CNRS,
   F-75013, Paris, France}
  \email{\href{mailto:emmanuel.wagner@imj-prg.fr}{emmanuel.wagner@imj-prg.fr}}
 \address[E.{} W.]{Laboratoire Cogitamus, \href{https://www.cogitamus.fr/indexen.html}{https://www.cogitamus.fr/}}
 \hypersetup{
    pdftoolbar=true,        
    pdfmenubar=true,        
    pdffitwindow=false,     
    pdfstartview={FitH},    
    pdftitle={\shorttitle},    
    pdfauthor={\shortauthors},     
    pdfsubject={\shorttitle},   
    pdfcreator={\shortauthors},   
    pdfproducer={\shortauthors}, 
    pdfkeywords={}, 
    pdfnewwindow=true,      
    colorlinks=true,       
    linkcolor=darkgray,          
    citecolor=teal,        
    filecolor=magenta,      
    urlcolor=violet,          
    linkbordercolor=red,
    citebordercolor=teal,
    urlbordercolor=violet,  
    linktocpage=true
    }
\makeatletter
\@namedef{subjclassname@2020}{\textup{2020} Mathematics Subject Classification}
\makeatother
\usetikzlibrary{arrows, decorations.markings, decorations, patterns,
positioning, decorations.pathreplacing, decorations.pathmorphing, intersections}
\tikzset{->-/.style={decoration={markings, mark=at position .5 with {\arrow{>}}},postaction={decorate}}}
\tikzset{-<-/.style={decoration={markings, mark=at position .5 with {\arrow{<}}},postaction={decorate}}}
\let\oldtocsubsection\tocsubsection
\renewcommand\tocsubsection[3]{\hspace{0.5cm}\oldtocsubsection{#1}{#2}{#3}}
\let\oldtocsubsubsection\tocsubsubsection
\renewcommand\tocsubsubsection[3]{\hspace{1cm}\oldtocsubsubsection{#1}{#2}{#3}}

\newcounter{res}[section]
\numberwithin{res}{section}
\newtheorem{thm}[res]{Theorem}

\newtheorem{lem}[res]{Lemma}
\newtheorem{lem-dfn}[res]{Lemma-Definition}

\newtheorem{prop}[res]{Proposition}
\newtheorem{cor}[res]{Corollary}

\theoremstyle{definition}
\newtheorem{notation}[res]{Notation}
\newtheorem{dfn}[res]{Definition}
\newtheorem{rmk}[res]{Remark}
\newtheorem{exa}[res]{Example}

\newtheorem{conv}[res]{Convention}
\def\co{\colon\thinspace}

\newcommand{\NB}[1]{\ensuremath{\vcenter{\hbox{#1}}}}

\newcommand{\NN}{\ensuremath{\mathbb{N}}}
\newcommand{\ZZ}{\ensuremath{\mathbb{Z}}}

\newcommand{\RR}{\ensuremath{\mathbb{R}}}

\renewcommand{\SS}{\ensuremath{\mathbb{S}}}

  \newcommand{\id}{\mathrm{Id}}

\newcommand{\Ker}{\mathop{\mathrm{Ker}}\nolimits}
\newcommand{\Hom}{\mathop{\mathrm{Hom}}}

\newcommand{\gll}{\ensuremath{\mathfrak{gl}}}
\newcommand{\sll}{\ensuremath{\mathfrak{sl}}}

\renewcommand{\deg}[2][{}]{\ensuremath{\mathrm{deg}_{#1}(#2)}}

\newcommand\kup[1]{\left\langle #1 \right\rangle}

\newcommand{\qbinom}[2]{\ensuremath
\begin{bmatrix}
  #1 \\
  #2
\end{bmatrix}}

\newcommand{\qbinomsmall}[2]{\ensuremath \left[\begin{smallmatrix}
    #1 \\
    #2
  \end{smallmatrix} \right]
}

 \newcommand{\Sym}{\ensuremath{\mathrm{Sym}}}

\newcommand{\powerset}[1]{\ensuremath{\mathcal{P}\left(#1\right)}}
\newcommand{\scalars}{\ensuremath{\Bbbk}}
\newcommand{\de}{\ensuremath{\mathbf{e}}}
\newcommand{\df}{\ensuremath{\mathbf{f}}}
\renewcommand{\dh}{\ensuremath{\mathbf{h}}}
\newcommand{\Le}{\ensuremath{\mathsf{e}}}
\newcommand{\Lf}{\ensuremath{\mathsf{f}}}
\newcommand{\Lh}{\ensuremath{\mathsf{h}}}
\newcommand{\LLn}[1][n]{\ensuremath{\mathsf{L}_{#1}}}
\newcommand{\Ln}[1][n]{\ensuremath{\mathbf{L}_{#1}}}

\newcommand{\newtoni}[1][i]{\ensuremath{p_{#1}}}
\newcommand{\dotnewtoni}[1][i]{\ensuremath{\textcolor[rgb]{0,0,0.6}{\spadesuit_{#1}}}}
\newcommand{\wdotnewtoni}[1][i]{\ensuremath{\textcolor[rgb]{0,0,0.6}{\widehat{\spadesuit}_{#1}}}}
\newcommand{\witt}{\ensuremath{\mathfrak{W}}}
\newcommand{\ourwitt}{\ensuremath{\witt_{-1}^{\infty}}}
\newcommand{\spherical}{spherical{}}
\newcommand{\pigments}{\mathbb{P}}
\newcommand{\NNN}{\ensuremath{\mathbb{N}_{-1}}}
\newcommand{\bracketN}[1]{\left\langle #1 \right\rangle_{\myN}}

\newcommand{\mymovie}[3][]{
  \NB{
    \begin{tikzpicture}[#1]
\begin{scope}
  \draw[gray, thick] (0, -0.05) -- +(0, 3.1);
  \draw[gray, thick] (4, -0.05) -- +(0, 3.1);
  \draw[gray, thick] (8, -0.05) -- +(0, 3.1);
  \draw[gray, line width=1mm] (0,0) -- +(8,0);
  \draw[white, densely dotted, line width=0.6mm] (0,0) -- +(8,0);
  \draw[gray, line width=1mm] (0,3) -- +(8,0);
  \draw[white, densely dotted, line width=0.6mm] (0,3) -- +(8,0);
  \node (Frame1) at (2, 1.5) {#2};
  \node (Frame1) at (6, 1.5) {#3};
\end{scope}
    \end{tikzpicture}
    }
}

\newcommand{\mymoviefour}[5][]{
  \NB{
    \begin{tikzpicture}[#1]
\begin{scope}
  \draw[gray, thick] (0, -0.05) -- +(0, 3.1);
  \draw[gray, thick] (4, -0.05) -- +(0, 3.1);
  \draw[gray, thick] (8, -0.05) -- +(0, 3.1);
  \draw[gray, thick] (12, -0.05) -- +(0, 3.1);
  \draw[gray, thick] (16, -0.05) -- +(0, 3.1);
  \draw[gray, line width=1mm] (0,0) -- +(16,0);
  \draw[white, densely dotted, line width=0.6mm] (0,0) -- +(16,0);
  \draw[gray, line width=1mm] (0,3) -- +(16,0);
  \draw[white, densely dotted, line width=0.6mm] (0,3) -- +(16,0);
  \node (Frame1) at (2, 1.5) {#2};
  \node (Frame1) at (6, 1.5) {#3};
  \node (Frame1) at (10, 1.5) {#4};
  \node (Frame1) at (14, 1.5) {#5};
\end{scope}
    \end{tikzpicture}
    }
}

\newcommand{\tone}{{t_1}}
\newcommand{\ttwo}{{t_2}}
\newcommand{\tthree}{{t_3}}

\newcommand{\dif}{\ensuremath{\partial}}
\newcommand{\Fp}{\ensuremath{\mathbb{F}_p}}
\newcommand{\tqftfunc}[1][]{\ensuremath{\mathcal{F}_\myN^{#1}}}
\newcommand{\etqftfunc}[1][]{\ensuremath{{}^{\mathrm{E}\!}\mathcal{F}_\myN^{#1}}}
\newcommand{\ztqftfunc}[1][]{\ensuremath{{}^{\mathrm{0}\!}\mathcal{F}_\myN^{#1}}}
\newcommand{\statespaceN}[2][]{\ensuremath{\tqftfunc[#1]\left(#2\right)}}
\newcommand{\estatespaceN}[2][]{\ensuremath{\etqftfunc[#1]\left(#2\right)}}
\newcommand{\zstatespaceN}[2][]{\ensuremath{\ztqftfunc[#1]\left(#2\right)}}

\newcommand{\prototqftfunc}[1][]{\ensuremath{\mathcal{V}_\myN^{#1}}}
\newcommand{\protostatespaceN}[2][]{\ensuremath{\prototqftfunc[#1]\left(#2\right)}}

\newcommand{\RN}{\ensuremath{R_\myN}}
\newcommand{\KN}{\ensuremath{\Bbbk_\myN}}

\newcommand{\myN}{\ensuremath{N}}
\newcommand{\web}{\ensuremath{\Gamma}}
\newcommand{\foam}{\ensuremath{F}}
\newcommand{\degN}[1]{\ensuremath{\mathrm{deg}_\myN\left(#1\right)}}
\newcommand{\facet}{\ensuremath{f}}
\newcommand{\surface}{\ensuremath{\Sigma}}
\newcommand{\annulus}{\ensuremath{\mathcal{A}}}
\newcommand{\foamcat}[1][]{\ensuremath{\mathsf{Foam}_{#1}}}
\newcommand{\vfoamcat}[1][]{\ensuremath{\mathsf{vFoam}_{#1}}}
\newcommand{\vectweb}[1]{\ensuremath{V\left(#1\right)}}
\newcommand{\Czip}{\ensuremath{Z}}
\newcommand{\Cunzip}{\ensuremath{Y}}
\newcommand{\Cdigcup}{\ensuremath{V}}
\newcommand{\Cdigcap}{\ensuremath{\Lambda}}
\newcommand{\Ccup}{\ensuremath{U}}
\newcommand{\Ccap}{\ensuremath{A}}
\newcommand{\indexweb}{\ensuremath{k}}
\newcommand{\br}{\ensuremath{\mathrm{br}}}
\newcommand{\rb}{\ensuremath{\mathrm{rb}}}

\newcommand{\parone}[1]{\ensuremath{\lambda_{#1}}}
\newcommand{\partwo}[1]{\ensuremath{\mu_{#1}}}
\newcommand{\parthree}[1]{\ensuremath{\nu_{#1}}}
\newcommand{\bracketNs}[1]{\bracketN{#1}^\mathrm{s}}

\newcommand{\imagesfolder}{.}

\allowdisplaybreaks
\begin{document}
\begin{abstract}
  We give an action of a Lie subalgebra of the Witt algebra on
  foams. This action is compatible with the $\mathfrak{gl}_N$-foam
  evaluation formula. In particular, this endows states spaces
  associated with $\mathfrak{gl}_N$-webs with an
  $\mathfrak{sl}_2$-action. When working in positive characteristic,
  this can be used to define a $p$-DG structure on these state spaces.
\end{abstract}
\maketitle
\setcounter{tocdepth}{3}
\tableofcontents

\section{Introduction}
\label{sec:intro}
The quest to categorify Witten--Reshetikhin--Turaev (WRT) invariants
\cite{Witten-qft, RT-3mfd} of 3-manifolds has been active for many
years \cite{CF}. The first major and promising step was accomplished
by Khovanov~\cite{KhJones} where he defined a categorification of the
Jones polynomial, now known as Khovanov homology. This link homology
theory led to many applications in low-dimensional topology.  Since
the discovery of Khovanov homology, many other link homologies have
been constructed: Heegaard--Floer homology \cite{HFK-OS, HFK-R},
$\gll_N$-homology, triply-graded HOMFLY-PT homology~\cite{KR1, KR2,
  KR3, Roulink}, etc.

The WRT invariants of a $3$-manifold $M$ are typically defined as 
certain linear combinations of quantum link invariants
(Reshetikhin--Turaev invariants) of a link presenting the manifold $M$
as a surgery on $\SS^3$ evaluated at a root of $1$.

This raises the challenge of making sense of ``\emph{categorification
  at root of $1$}''. A strategy has been suggested by Khovanov
\cite{Hopforoots} where he introduced $p$-complexes for $p$ a prime
number. This was developed later in various works
\cite{KQ,QYHopf,QiSussan2,EQ2} leading eventually to a categorification
of the Jones polynomial \cite{QiSussanLink} and of the colored Jones
polynomial~\cite{QRSW1} at prime roots of unity. Both these
constructions are based on a new categorification of the Jones
polynomial introduced by \cite{Cautisremarks} whose definition is
closely related to that of triply-graded link homology and therefore
to Soergel bimodules.

A key feature in categorification at prime roots of $1$ is to show that these link
homologies carry additional algebraic structures ($p$-DG
structures or $H$-module structures where $H=\Bbbk[\dif]/(\dif^p)$ where the degree of $\dif$ is two). This allows one to work in the stable
category of graded $H$-modules whose Grothendieck group is
\[\ZZ[q, q^{-1}]/ (1+ q^2 + \dots + q^{2p-2}).\]
In other words, this category is a categorical incarnation of the
arithmetic of $\ZZ[e^{\pi i /p}]$. We refer to the introductions of
\cite{QiSussanLink,QRSW1} for a more detailed account.

Even more structure has been found on triply-graded homology
\cite{KRWitt} in the form of an action of the positive half of the
Witt algebra.  In a similar vein, an action of $\sll_2$ was
constructed on Soergel bimodules \cite{EQ3}.  Building upon \cite{KQ},
Khovanov extended the $p$-differential on nilHecke algebras (in
characteristic $p$) to a half-Witt action on nilHecke algebras, and
more general KLR algebras in characteristic $0$ \cite{KhovNote}.  He
suggested that these actions may extend to actions on link homologies
which could explain various observed symmetries.  In a related but
different direction, Beliakova and Cooper noticed that in
characteristic $p$ there is an action of the Steenrod algebra on
nilHecke algebras and that one could recover the $p$-DG structure on
these algebras from this point of view \cite{BCSteenrod}. Finally, the
operator $\nabla$ constructed by Wang in \cite{Wang} can be seen as a
(part of) the $\sll_2$ action described in the present paper.

The aim of this paper is to show that these algebraic structures also
show up in the TQFT functors used to define the original
categorification of the Jones polynomial (actually its equivariant
version, as defined by two of the authors of this paper \cite{RW1})
and in the $\gll_N$ generalizations. In other words, we put the
previous studies of $p$-DG structures on link homology
\cite{QiSussanLink,QRSW1} into the framework of investigating
(infinitesimal) symmetries of foam evaluations.
We tried to maintain as much
flexibility as possible to allow usage of the construction in this
paper in a large variety of contexts. This explains the presence of
parameters in the actions we define. The statement about the action of
a half of the Witt algebra is phrased in
Theorem~\ref{thm:wittaction-spherical} and
Theorem~\ref{thm:witt-acts-nons-spherical}. An $\sll_2$-action is
given in Propositions~\ref{prop:sl2action-spherical} and
\ref{prop:sl2-action-non-spherical}. The $p$-DG structure is presented
in Propositions~\ref{prop:pDG-structure} and
\ref{prop:pDG-structure-nons-perical}.

The construction is based on foams, which 
also makes sense in the context of Soergel bimodules\footnote{As in \cite{EliasQisl2}, we are also forced to utlize the $\gll_N$ realization instead of the $\sll_N$ one. This is simply due to the fact that, the Witt generators only act on the base ring of full symmetric functions, and not on the quotient ring by the ideal generated by the first elementary symmetric function.}.  Thus the
structure we exhibit here specializes to that in \cite{EQ3} and
\cite{KRWitt}. 

A forthcoming paper \cite{QRSW3} will
implement part of these structures at the level of link
homologies.

\subsection{Outline}
\label{sec:outline}
The remainder of the paper is organized as follows.
\begin{itemize}
\item Section~\ref{sec:top-prelim} gives a self-contained account
  of $\gll_N$-webs and foams, a recollection of the foam evaluation formula, a discussion of
  how to compute Euler characteristics of surfaces, and a definition of
  $\gll_N$-state spaces.
\item Section~\ref{sec:alg-prelim} presents the
  Witt Lie algebra $\witt$ as well as  $\ourwitt$, a half of $\witt$
  and describes its action on polynomial rings. It also
  explains how $\sll_2$ embeds in this algebra, and how this can be used
  to define $p$-DG-structures.
\item Section~\ref{sec:an-mathfr-acti} defines the action of
  $\ourwitt$ on foams, giving rise to $\sll_2$ and
  $p$-DG-structures. The section concludes with relations to related
  actions already occurring in the literature.
\end{itemize}

\subsection{Acknowledgments}
\label{sec:acknoledgment}
We would like to thank Lev Rozansky, Ben Elias and Joshua Wang for
interesting and enlightening conversations and to point to \cite{Wang}
and \cite{EliasTalk} as sources of inspiration.

We especially would like to thank Mikhail Khovanov who encouraged us
to search for Witt algebra actions on foams.

The anonymous referee helped us a lot to improve the exposition of the paper. We would like to thank them for their thoughtful comments.

The ideas of the present paper were already implicitly used in
\cite{QRSW1}.  The relevance of formalizing them became apparent
during the hybrid workshop ``Foam Evaluation'' held at ICERM organized
by one of the authors, Mikhail Khovanov, and Aaron Lauda.

Some figures are recycled from papers of various subsets of the authors
with or without other collaborators.

Y.Q.{} is partially supported by the Simons Foundation Collaboration Grants for Mathematicians. J.S.{} is
partially supported by the NSF grant DMS-1807161.
Y.Q.{} and J.S.{} thank the Universit\'e Paris Cit\'e  and its Programme d'invitations internationales scientifiques for their hospitality.
LH.R.{} was partially supported by the Luxembourg National Research Fund PRIDE17/1224660/GPS.
E.W.{} is partially supported by the ANR projects AlMaRe
(ANR-19-CE40-0001-01), AHA (JCJC ANR-18-CE40-0001) and CHARMS (ANR-19-CE40-0017).

\section*{Conventions}
\label{sec:conventions}
Pardon our French, $\NN$ stands for the set of non-negative
integers. Foams are read from bottom to top. We set
$\NNN= \{ k\in \ZZ, k\geq -1\}$. For a ring with unity $\scalars$, and for $x \in \scalars$, we set $\bar{x}=1-x$.
Note that this is not a ring automorphism.

The algebras $\RN=\ZZ[X_1, \dots,X_\myN]^{S_\myN}$ and $\KN = \scalars[X_1, \dots, X_\myN]^{S_\myN}$ 
will play central roles in this paper. They are non-negatively graded
by imposing that $\deg{X_i} =2$. The $i$th elementary, complete
homogeneous, and power sum symmetric polynomials in
$X_1,\dots, X_\myN$ are denoted by $E_i$, $H_i$ and $P_i$
respectively, so that
\[
  \RN=\ZZ[E_1, \dots, E_\myN] \qquad \text{and} \qquad \KN=
  \scalars[E_1, \dots, E_\myN].
\]

For $a \in \NN$, $\Sym_{a}$ denotes the ring of symmetric polynomials
in $a$ variables with $\ZZ$ coefficients, in particular
$\RN=\Sym_\myN$. When working in such a ring, we will
let
$e_i, h_i$ and $p_i$ be the $i$th elementary, complete homogeneous,
and power sum symmetric polynomials respectively without reference to
the variables.
The ring $\Sym_{a}$ is graded by imposing the $e_i$ is homogeneous of
degree $2i$. With this setting, we have:
\[ \deg{e_i} = \deg{h_i} = \deg {p_i} = 2i.\]

For $n \in \ZZ$, we let $[n]=\frac{q^n-q^{-n}}{q-q^{-1}}$, for $k \in
\NN$, we let $[k]!= \prod_{j=1}^k[j]$. Finally, for $m\in \ZZ$ and $a
\in \NN$,  define: 
\[
  \qbinom{m}{a}=\prod_{i=1}^a \frac{[m+1-i]}{[i]}.\]
Note that if $m$ is non-negative, one has $\qbinomsmall{m}{a}= \frac{[m]!}{[a]![m-a]!}$.

For a $\mathbb{Z}$-graded vector space $V$, let $V_i$ denote the subspace in degree $i$. Let $q^n V$ denote the $\mathbb{Z}$-graded vector space where $(q^n V)_i=V_{i-n}$.

\section{Topological preliminaries}
\label{sec:top-prelim}
\subsection{Webs and foams}
\label{sec:webs-foams}

\begin{dfn}
  Let $\surface$ be a surface.  A \emph{closed web} or simply a
  \emph{web}\footnote{Such a graph is sometimes referred to as a \emph{MOY graph}, but we will use the more commonly used term of \emph{web}.
} 
    is a finite, oriented, trivalent graph
  $\web = (V(\Gamma), E(\Gamma))$ embedded in the interior of
  $\surface$ and endowed with a \emph{thickness function}
  $\ell\co E(\Gamma) \to \NN$ satisfying a flow condition: vertices
  and thicknesses of their adjacent edges must be one of these two types:
  \[
    \NB
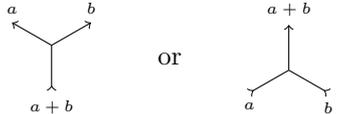
{\tikz[scale=0.6]{\begin{scope}[font=\tiny]
  \begin{scope}
    \draw [-<] (0,0) -- (-90:1) node[pos =1, below] {$a+b$};
    \draw [->] (0,0) -- (30:1) node[pos =1, above] {$b$};
    \draw [->] (0,0) -- (150:1) node[pos =1, above] {$a$};
  \end{scope}
\end{scope}
}} \qquad\text{or}\qquad \NB{\tikz[scale=0.6]{\begin{scope}[font=\tiny]
    \draw [->] (0,0) -- (90:1) node[pos =1, above] {$a+b$};
    \draw [-<] (0,0) -- (-30:1) node[pos =1, below] {$b$};
    \draw [-<] (0,0) -- (-150:1) node[pos =1, below] {$a$};
\end{scope}
}}.
  \]
  The first type is called a \emph{split} vertex, the second a \emph{merge}
  vertex. In each of these types, there is one \emph{thick} edge and
  two \emph{thin} edges. Oriented circles with non-negative thickness
  are regarded as edges without vertices and can be part of a web. The
  embedding of $\web$ in $\surface$ is smooth outside its vertices,
  and at the vertices should fit with the local models above.
\end{dfn}

\begin{figure}
  \centering
  \NB{\tikz[font=\tiny]{\input{\imagesfolder/pf_exa_web}}}
    \caption{Example of a web in $\RR^2$.}
    \label{fig:exa-web}
\end{figure}

  The surfaces we will be interested in are $\RR^2$ and $\annulus= \{(x,y)
  \in \RR^2, 1\leq |x|^2+ |y|^2\leq 2\} \simeq \SS^1\times [0,1]$. In
  the latter case, we require that the web is \emph{directed}, meaning
  that the projection map $\pi\co \web \to \SS^1$ preserves
  orientation locally. Such webs are called \emph{vinyl graphs}.

\begin{rmk}
  There are neither sources nor sinks in a web. A web is not
  necessarily connected. 
\end{rmk}

\begin{dfn}\label{def:foam}
  Let $M$ be an~oriented smooth 3-manifold with a~collared boundary.
  A~\emph{foam} $\foam \subset M$ is a~collection of \emph{facets},
  that are compact oriented surfaces labeled with non-negative integers
  and glued together along their boundary points, such that every
  point $p$ of $\foam$ has a~closed neighborhood homeomorphic to one
  of the~following:
  \begin{enumerate}
  \item a~disk, when $p$ belongs to a~unique facet,
  \item \label{it:Y}$Y \times [0,1]$, where $Y$ is the neighborhood of
    a~merge or split vertex of a web, when $p$ belongs to three facets, 
  \item the~cone over the~1-skeleton of a~tetrahedron with $p$ as
    the~vertex of the~cone (so that it belongs to six facets).
  \end{enumerate}
  See Figure~\ref{fig:foam-local-model} for a pictorial representation
  of these three cases. The set of points of the~second type is
  a~collection of curves called \emph{bindings} and the~points of
  the~third type are called \emph{singular vertices}.
  The~\emph{boundary} $\partial\foam$ of $\foam$ is the~closure of
  the~set of boundary points of facets that do not belong to
  a~binding. It is understood that $\foam$ coincides with
  $\partial\foam\times[0,1]$ on the~collar of $\partial M$. For each facet
  $\facet$ of $\foam$, we denote by $\ell(\facet)$ its~label, 
  called the \emph{thickness of $\facet$}. A~foam $\foam$ is
  \emph{decorated} if each facet $\facet$ of $\foam$ is assigned
  a~symmetric polynomial $P_f \in \Sym_{\ell(\facet)}$.  In the second local
  model \ref{it:Y}, it is implicitly understood that thicknesses of
  the three facets are given by that of the edges in $Y$. In
  particular it satisfies a flow condition and locally one has a thick
  facet and two thin ones. We also require that orientations of
  bindings are induced by that of the thin facets and by the opposite
  of the thick facet. Foams are regarded up to ambient isotopy relative to
  boundary. Foams without boundary are said to be \emph{closed}.
\end{dfn}

\begin{rmk}
  \begin{itemize}
    \item Diagrammatically, decorations on facets are depicted by dots
      placed on facets 
      adorned with symmetric polynomials in the correct number of
      variables (the thickness of the facet they sit on). The
      decoration of a given facet is the product of all adornments of
      dots sitting on that facet. 
  \item Decorations will be slightly generalized a bit later. See
    Section~\ref{sec:new-decorations} and
    Convention~\ref{conv:new-dec}.
  \end{itemize}
\end{rmk}

\begin{figure}[ht]
  \centering
  \NB{\tikz[]{\input{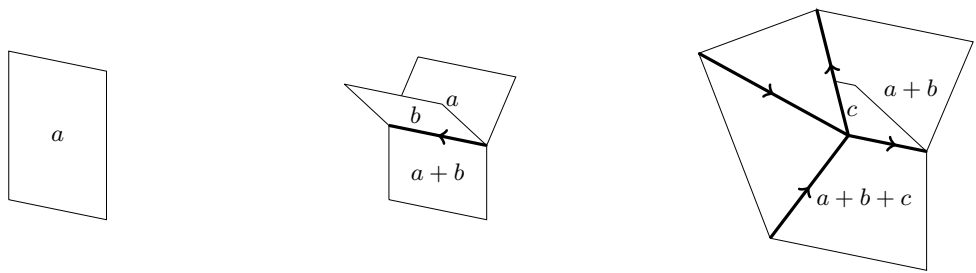}}}
  \caption{The three local models of a foam. Taking into account the
    thicknesses, the model in the middle is denoted $Y^{(a,b)}$,
    and the model on the right is denoted $T^{(a,b,c)}$.}
  \label{fig:foam-local-model}
\end{figure}

The $3$-manifolds in which we will consider foams are $\RR^3$, $\RR^2 \times
[0,1]$ and $\annulus \times [0,1]$.

\begin{notation}
  For a foam $\foam$, we write:
  \begin{itemize}
  \item $\foam^2$ for the~collection of facets of $\foam$,
  \item $\foam^1$ for the collection of bindings,
  \item $\foam^0$ for the set of singular vertices of $\foam$.
  \end{itemize}
  We partition $\foam^1$ as follows:
  $\foam^1= \foam^1_\circ\sqcup \foam^1_{-}$, where $\foam^1_\circ$ is
  the collection of circular bindings and $\foam^1_-$ is the
  collection of bindings diffeomorphic to intervals. If $s \in
  \foam^1_-$, any of its points has a neighborhood diffeomorphic to $Y^{(a,b)}$ for a given $a$ and $b$, and we set:
  \begin{equation}
    \degN{s} = ab + (a+b)(\myN-a-b).\label{eq:deg-binding}
  \end{equation}
  If $v \in
  \foam^0$, it has a neighborhood diffeomorphic to $T^{(a,b,c)}$ and we set:
  \begin{equation}
    \degN{v} = ab +bc + ac + (a+b+c)(\myN-a-b-c).\label{eq:deg-sing}
  \end{equation}
\end{notation}

\begin{dfn}\label{dfn:deg-foam}
  Let $\foam$ be a decorated foam and suppose that all decorations are
  homogeneous. For all $\myN$ in $\NN$, the \emph{$\myN$-degree of
    $\foam$} is the integer $\degN{\foam} \in \ZZ$ given by the
  following formula: \begin{equation}\label{eq:deg-foam}
    \degN{\foam}: = \sum_{\facet \in \foam^2}
    \Big(\deg{P_\facet} -\ell(\facet)(\myN-\ell(\facet))\chi(\facet)
    \Big) + \sum_{s \in \foam^1_-} \degN{s} - \sum_{v \in \foam^0} \degN{v}.
  \end{equation}
\end{dfn}

The reader may want to wait until Remark~\ref{rmk:on-foam-degree} to
see a better approach to calculating the $\myN$-degree in the case of foams
with trivial decorations.

The~boundary of a~foam $\foam\subset M$ is a~web in $\partial M$. In the
case $M = \surface\times[0,1]$ is a~thickened surface, a~generic
section $\foam_t := \foam \cap (\surface\times\{t\})$ is
a~web. The~bottom and top webs $\foam_0$ and $\foam_1$ are called the~\emph{input} and \emph{output} of $\foam$ respectively.

If $\surface$ is a surface, $\foamcat[\surface]$ is the category which
has webs in $\surface$ as objects and 
\[
  {\Hom}_{\foamcat[\surface]}\left( \web_0,\web_1\right)
  = \left\{\text{decorated foams $\foam$ in $\surface\times [0,1]$ with $\foam_i =\web_i$ for $i\in \{0,1\}$}\right\}.
\]
Composition is given by stacking foams on one another and
rescaling. Decorations behave multiplicatively. The identity of $\web$
is $\web\times [0,1]$ decorated by the constant polynomial $1$ on
every facet. The $\myN$-degree of foams is additive under composition (see for instance \cite[Lemma 3.4]{QueffelecRoseFoam}).
If $\web$ is a web in a surface $\surface$ and
$h:\surface\times[0,1] \to \surface$ is a smooth isotopy\footnote{For
  the sake of satisfying the collared condition, one should assume
  that $h_t=\id_\surface$ for $t\in [0,\epsilon[\cup]1-\epsilon, 1]$
  for an $\epsilon\in ]0,1]$.} of $\surface$, one can define the foam
$\foam(h)$ to be the trace of $h(\web)$ in $\surface\times[0,1]$: for
all $t \in [0,1]$, $\foam(h)_t=h_t(\web)$. Such foams are called
\emph{traces of isotopies}. They have degree $0$.

\begin{dfn}\label{dfn:basic-foams}
 A foam in a surface $\surface\times[0,1]$ is \emph{basic} if it is a trace of isotopy or if it
 is equal to $\web\times[0,1]$ outside a cylinder $B\times[0,1]$, and where
 it is given inside by one of the local models given in Figure~\ref{fig:basic-foam}. The \emph{non-trivial part} of a basic foam $F$ is the empty set if $F$ is the trace of an isotopy and is the part of the foam where the local model appears otherwise.  
 \begin{figure}[ht]
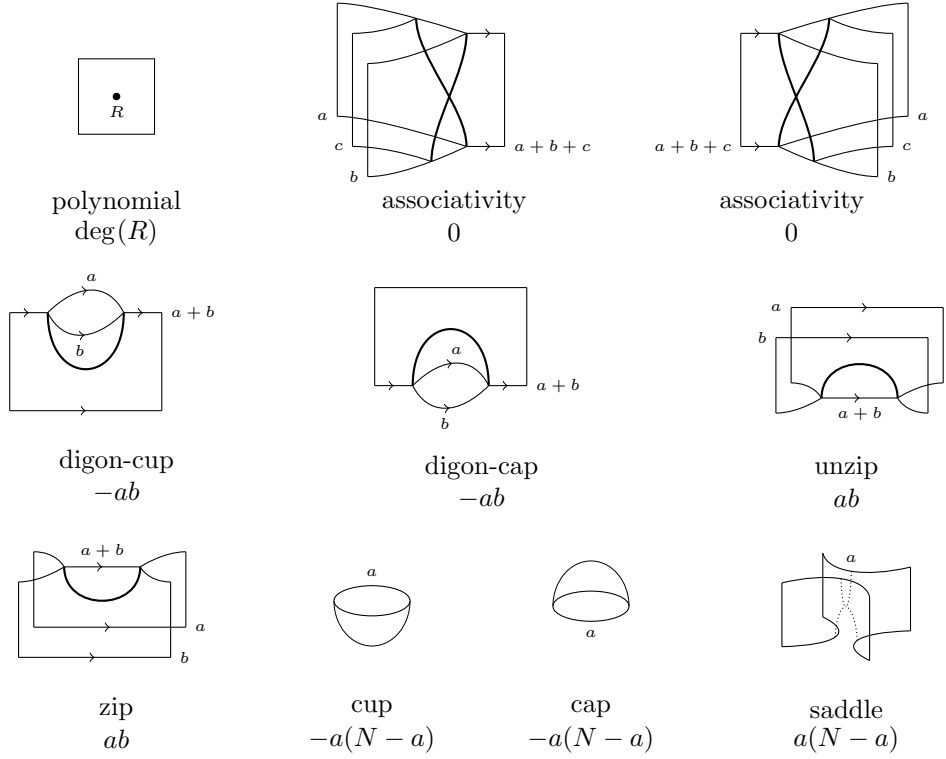

   \centering
   \begin{tikzpicture}[xscale=2.4, yscale =-3.2]
     \node (pol) at (0,0) {\NB{\tikz[font=\tiny]{\input{\imagesfolder/pf_foam-polR}}}};
     \node[yshift= -1.4cm] at (pol) {polynomial};
     \node[yshift = -1.8cm] at (pol) {$\deg{R}$};
     \node (asso) at(1.85,0) {\NB{\tikz[font=\tiny]{\input{\imagesfolder/pf_foam-assoc}}}};
     \node[yshift = -1.4cm] at (asso) {associativity};
     \node[yshift= -1.8cm] at (asso) {$0$};
     \node (coasso) at (3.7,0) {\NB{\tikz[font=\tiny]{\input{\imagesfolder/pf_foam-coassoc}}}};
     \node[yshift = -1.4cm] at (coasso) {associativity};
     \node[yshift= -1.8cm] at (coasso) {$0$};
     \node (digcup) at (0,1) {\NB{\tikz[font=\tiny]{\input{\imagesfolder/pf_foam-digon-cup}}}};
     \node[yshift = -1.64cm] at (digcup) {digon-cup};
     \node[yshift= -2.04cm] at (digcup) {$-ab$};
     \node (digcap) at (2,1.1) {\NB{\tikz[font=\tiny]{\input{\imagesfolder/pf_foam-digon-cap}}}};
     \node[yshift = -1.4cm] at (digcap) {digon-cap};
     \node[yshift= -1.8cm] at (digcap) {$-ab$};
     \node (zip) at (4,1.1) {\NB{\tikz[font=\tiny]{\input{\imagesfolder/pf_foam-unzip}}}};
     \node[yshift = -1.4cm] at (zip) {unzip};
     \node[yshift= -1.8cm] at (zip) {$ab$};
     \node (unzip) at (0,2.1) {\NB{\tikz[font=\tiny]{\input{\imagesfolder/pf_foam-zip}}}};
     \node[yshift = -1.4cm] at (unzip) {zip};
     \node[yshift= -1.8cm] at (unzip) {$ab$};
     \node (cup) at (1.4,2.1) {\NB{\tikz[font=\tiny]{\input{\imagesfolder/pf_foam-cup}}}};
     \node[yshift = -1.4cm] at (cup) {cup};
     \node[yshift= -1.8cm] at (cup) {$-a(\myN-a)$};
     \node (cap) at (2.6,2.1) {\NB{\tikz[font=\tiny]{\input{\imagesfolder/pf_foam-cap}}}};
     \node[yshift = -1.4cm] at (cap) {cap};
     \node[yshift= -1.8cm] at (cap) {$-a(\myN-a)$};
     \node (saddle) at (4,2.1) {\NB{\tikz[font=\tiny]{\input{\imagesfolder/pf_foam-saddle}}}};
     \node[yshift = -1.4cm] at (saddle) {saddle};
     \node[yshift= -1.8cm] at (saddle) {$a(\myN-a)$};
   \end{tikzpicture}
   \caption{The degree of a basic foam is given below the name of each
     of the local models.}\label{fig:basic-foam}
 \end{figure}

 A foam in $\surface \times [0,1]$ is \emph{in good position} if it is
 a composition of basic foams. A foam in $\surface \times [0,1]$ is
 \emph{\spherical} if it is isotopic to a foam in good position for
 which the saddle model is not used.

If $\web$ is a web in $\RR^2$, we denote by $\vectweb{\web}$ the free
$\scalars$-module generated by foams in good position in $\RR^2\times
[0,1]$ with $\emptyset$ as input and $\web$ as output.
\end{dfn}

\begin{rmk}
  Every foam in $\surface \times [0,1]$ is isotopic to a foam in good
  position, however, not every foam is \spherical. For instance, a torus (of arbitrary thickness) is not \spherical. 
\end{rmk}

\begin{conv}If a foam is both \spherical{} and in good position,
  we assume that no saddle appears in its decomposition as a composition
  of basic foams.
\end{conv}

\subsection{\texorpdfstring{$\gll_\myN$}{gl(N)}-foams evaluation}
\label{sec:foam evaluation}

In this subsection we briefly summarize the $\gll_\myN$-foam evaluation
introduced in \cite{RW1}.

For the rest of this section, we fix $\myN$ indeterminates $X_1,
\dots, X_\myN$. The elements of $\pigments:=\{1, \dots, \myN\}$ are
called \emph{pigments}.
A \emph{$\gll_\myN$-coloring} $c$ of a foam $\foam$ is a map $c \co \foam^2 \to
\powerset{\pigments}$, where $\mathcal{P}$ stands for powerset. It
should satisfy the following two conditions:
\begin{enumerate}[(a)]
\item for each facet $\facet \in \foam^2$, $\# c(\facet) = \ell(\facet)$,
\item around each binding, $c(\facet_{\mathrm{thick}})=
  c(\facet_{1}) \cup c(\facet_{2})$, where $\facet_{\mathrm{thick}}$
  denotes the thick facet  and $f_1$ and $f_2$ the
  thin facets at this binding.
\end{enumerate}

Given $\foam$, a decorated closed foam, and $c$ a $\gll_\myN$-coloring
of $\foam$, the \emph{colored $\gll_\myN$-evaluation of $(\foam,c)$} is
the rational function in variables $X_1, \dots, X_\myN$
defined by:
\begin{equation}\label{eq:colored-ev}
  \bracketN{\foam,c} := (-1)^{s(\foam,c)}\frac{P(\foam,c)}{Q(\foam,c)}
\end{equation}
with
\begin{align}
    P(\foam,c) &:= \prod_{\facet \in \foam^2}P_\facet(\underline{X}_{c(f)})  \qquad \textrm{and}\\
    Q(\foam,c) &:=\prod_{1\leq i < j \leq \myN}(X_i - X_j)^{\chi(\foam_{ij}(c))/2},
\end{align}
and where we have the following.
\begin{itemize}
\item $P_\facet(\underline{X}_{c(\facet)})$ is the evaluation of the polynomial
  $P_\facet$ in the indeterminates $\underline{X}_{c(\facet)}:= \{X_i | i \in c(\facet)\}$.
 \item $\foam_{ij}(c)$ is the surface formed by facets $\facet$ in $\foam^2$
   whose colors $c(\facet)$ contain either $i$ or $j$ but not both. This surface
   is called the \emph{bichrome surface of $(\foam,c)$ associated with
   $(i,j)$},
\item $s(\foam,c)$ is the integer given by the following formula:
  \begin{equation}
s(\foam,c) = \sum_{i=1}^\myN \frac{i\chi(\foam_i(c))}{2} + \sum_{1\leq i <j
      \leq \myN} \theta_{ij}^+(\foam,c),
  \end{equation}
  where we have the following. \begin{itemize}
  \item[$\circ$] $\theta_{ij}^+(\foam,c)$ counts the number of circles
    separating $i$-pigmented and $j$-pigmented regions in $\foam_{ij}(c)$ which
    are positive.  This means that the orientation of the circle and
    the relative position of the $i$-pigmented and $j$-pigmented regions are
    locally given by the local model:
    \[
       \NB
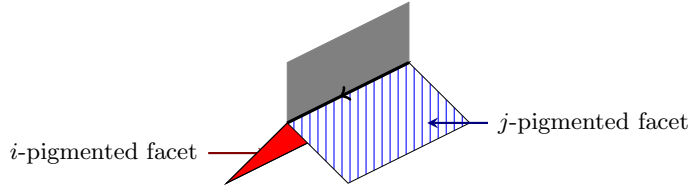
{\tikz[scale=0.8, font=\small]{\begin{scope}
  \coordinate (ci) at (-0.3,-0.5);
  \coordinate (cj) at (2.3, 0);
  \draw[stealth-, thick, red!50!black] (ci) -- +(-1,0) node[left, black] {$i$-pigmented facet};
  \filldraw[ very thin, gray, fill opacity=0.2] (0,0) -- (2,1) --
  (2,2) -- (0,1) -- cycle;
  \filldraw[very thin, fill= red, opacity = 0.3,] (0,0) -- (2,1) -- (1,0) -- (-1,-1) -- cycle;
  \filldraw[very thin, fill= red, opacity = 1, pattern=dots, pattern color=red] (0,0) -- (2,1) -- (1,0) -- (-1,-1) -- cycle;
  \fill[white, opacity = 0.5] (0,0) -- (2,1) -- (3,0) -- (1,-1) -- cycle;
  \filldraw[very thin, opacity = 0.7, pattern=vertical lines, pattern
  color=blue] (0,0) -- (2,1) -- (3,0) -- (1,-1) -- cycle;
  \draw[very thick, -<-] (0,0) -- (2,1); 
  \draw[stealth-, thick, blue!50!black] (cj) -- +(1,0) node[right, black] {$j$-pigmented facet};
\end{scope}

}} .
    \]
    \item[$\circ$] $\foam_i(c)$ is the surface formed by facets $\facet$ in
    $\foam^2$ whose colors $c(\facet)$ contain $i$. These surfaces are
    called \emph{monochrome surfaces of $(F,c)$ associated with $i$.}
  \end{itemize}
\end{itemize}

Note that the symmetric group $S_\myN$ on $\myN$ letters acts both on the set
of colorings of $\foam$ (by permuting the pigments) and on the ring
of rational functions in variables  $X_1, \dots, X_\myN$. The colored evaluation
intertwines these two actions as stated by next lemma.

\begin{lem}[{\cite[Lemma 2.17]{RW1}}] \label{lem:sym-col-ev}
    If $\sigma\in S_\myN$, then
    \begin{equation}
      \label{eq:4}
      \bracketN{F, \sigma\cdot c} =\sigma \cdot \bracketN{F,c}.
    \end{equation}
\end{lem}

 Finally define the \emph{$\gll_\myN$-evaluation} of a foam $\foam$ by:
 \begin{equation}
\bracketN{\foam}:= \sum_c\bracketN{\foam,c},\end{equation}
where the sum runs over all $\gll_\myN$-colorings of $\foam$.

\begin{prop}[{\cite[Proposition 2.19]{RW1}}] \label{prop:evaluation}
  For any decorated closed foam, $ \bracketN{\foam}$ is a symmetric polynomial
  in $X_1, \dots, X_N$. If decorations are homogeneous,
  \[\deg{\bracketN{\foam}} = \degN{\foam}. \]
\end{prop}

\begin{rmk}[{\cite[Lemma 2.15]{RW1}}] \label{rmk:on-foam-degree}
  \begin{enumerate}
  \item Let $\foam$ be a (not-necessarily closed) foam with trivial
    decoration and $c$ a coloring of $\foam$. The following identity
    holds
    \begin{equation}\label{eq:deg-col}
      \degN{\foam} = -\sum_{1\leq i<j \leq \myN} \chi(F_{ij}(c)).
    \end{equation}
    \item Note that contrary to \cite[Definition 6.2]{BN} or
      \cite[Definition 3.1]{QueffelecRoseFoam}, the formula
      \eqref{eq:deg-col} only involves Euler characteristic of surfaces
      and has no contribution from webs on the  boundary. This is
      because we chose only to work with closed webs, for which these
      contributions vanish.
  \end{enumerate}
\end{rmk}

\subsection{New decorations}
\label{sec:new-decorations}

For our purposes it will be convenient to have decorations of a new
type. In formula \eqref{eq:colored-ev}, the decoration of a facet is evaluated on
variables corresponding to pigments in the color of the
facet. This is equivalent to saying that one has the following local relation on the evaluation:
\begin{align}\label{eq:normal-dec-ev}
  \bracketN{\NB{\tikz[font=\small]{\begin{scope}
  \draw (0,0) rectangle (1,1) coordinate [midway] (A);
  \fill (A) circle (0.5mm) node[below] {$R$};
\end{scope}}},c} = R(\underline{X}_{c(\facet)}) \bracketN{\NB{\tikz[]{\begin{scope}
  \draw (0,0) rectangle (1,1) coordinate [midway] (A);
\end{scope}}},c}, 
\end{align}
where $\facet$ denotes the facet decorated by $R$. 

Following~\cite{ETW}, one can also introduce decorations which are
polynomials to be evaluated on variables corresponding to pigments
that are not in the color of that facet. Although we gain flexibility
with this new decoration, it does not affect the result from the
previous subsection, as we shall explain below. If $\facet$ is a facet
in a foam $\foam$ and $c$ is a coloring of $\foam$, let
$\widehat{c}(\facet)$ be the complement of $c(\facet)$ in $\pigments$.

\begin{lem}\label{lem:bubble}
  The colored evaluation satisfies the following local relation:
  \begin{equation}
  (-1)^{(\myN-a)(\myN-a+1)/2}\bracketN{\NB{\tikz[scale =0.8]{\begin{scope}[font = \tiny]
  \draw (0,0) -- (0,2) -- (1,2.5) -- (1,0.5) -- cycle;
  \draw[thick,-<-] (0.5, 1.25) circle (0.2 and 0.5);
  \filldraw[fill =white, fill opacity =0.8, very thin] (0.5, 0.75) .. controls +(1, -0.5) and +(1, 0.5)
  .. (0.5, 1.75);
  \fill[pattern=north west lines, opacity = 0.5, pattern color= gray]
  (0.5, 1.25) circle (0.2 and 0.5);
  \node[font= \normalsize, scale=0.5] at (0.2, 0.3) {$a$};
  \node[font= \normalsize, scale=0.5] at (0.75, 1.65) {$\myN-a$};
  \fill (0.85, 1.25) circle (0.5mm) node[below, font= \normalsize, scale=0.6] {$R$};
\end{scope}
}},c} =
    R(\underline{X}_{\widehat{c}(\facet)}) \bracketN{\NB{\tikz[scale
    =0.8]{\begin{scope}[font = \tiny]
  \draw (0,0) -- (0,2) -- (1,2.5) -- (1,0.5) -- cycle;
  \node[font= \normalsize, scale=0.5] at (0.2, 0.3) {$a$};
\end{scope}
}} \ ,c},\label{eq:newdec-lem}     
  \end{equation}
  where:
  \begin{itemize}
  \item we have abused notation and let $c$ be the coloring on both sides,
    (this is legitimate since there is a canonical $1$-to-$1$
    correspondence between the set of $\gll_N$-colorings of the foams on
    both sides: take the complementary pigments for the coloring of the extra bubble); 
  \item $\facet$ is the facet locally represented by the rectangle of thickness a on
both sides of the equation;
  \item the hashed facet on the left-hand side has thickness $\myN$.
  \end{itemize}
\end{lem}

\begin{proof} 
  This follows directly from 
  formula \eqref{eq:colored-ev}. Let us denote by $\foam$ and $G$ the foam on the left-hand
  side and the on the right-hand side of \eqref{eq:newdec-lem}
  respectively. Only the sign discussion is not trivial. By Lemma~\ref{lem:sym-col-ev}, one can suppose that the bubble is
  colored by $\{1, \dots, N-a\}$. In that case,  $\theta_{ij}^+(\foam)
  = \theta_{ij}^+(G)$ for all $i, j$ and
  \begin{equation}
  \chi(\foam_i(c)) =
  \begin{cases}
    \chi(G_i(c)) +2 & \text{if $1\leq i\leq \myN-a$,} \\
    \chi(G_i(c)) & \text{if $\myN - a+1 \leq i\leq \myN$}.
  \end{cases}  \qedhere
\end{equation}
\end{proof}
 
\begin{rmk}
In Lemma~\ref{lem:bubble}, if the ``bubble'' were glued on the
  other side, the formula would have an extra $(-1)^{a(\myN-a)}$
  factor. This is  because, keeping the same notations as in the proof, one would have 
\begin{equation}
  \theta_{ij}^+(\foam, c) =
  \begin{cases}
    \theta_{ij}^+(G, c) +1 & \text{if $1\leq i\leq \myN-a$ and $\myN - a+1 \leq j\leq \myN$ }, \\
    \theta_{ij}^+(G, c) & \text{otherwise.} \\
  \end{cases}
\end{equation}
This trick of using bubbles was already used in \cite{RW1} to write
  down the formula for the neck-cutting relation.
\end{rmk}

Lemma~\ref{lem:bubble} allows us to make sense of new kinds of
decorations on facets. A decoration of a facet $\facet$ of thickness
$a$ can now be a product of a symmetric polynomial in $a$ variables (as
before) with a symmetric polynomial in $\myN-a$ variables, or a sum of
such expressions. In other words, a decoration of a facet $\facet$ of thickness $a$
is an element $P_f$ of $\scalars[x_1, \dots, x_{\myN}]^{S_a\times
  S_{\myN-a}}\cong \scalars[x_1,\dots, x_a]^{S_a}\otimes
\scalars[x_1,\dots, x_{\myN-a}]^{S_{\myN-a}}$. The formula for
$P(\foam,c)$ becomes:
\begin{align}
P(\foam,c) &:= \prod_{\facet \in \foam^2}P_\facet(\underline{X}_{c(f)}, \underline{X}_{\widehat{c}(f)})  
\end{align}
where $P_\facet(\underline{X}_{c(f)}, \underline{X}_{\widehat{c}(f)})$ is the evaluation of
$P_f$ on $\underline{X}_{c(f)}$ for the first $a=\ell(\facet)$ variables and
$\underline{X}_{\widehat{c}(f)}$ for the last $\myN-a$ variables.

Proposition~\ref{prop:evaluation} remains valid with these more general
decorations, since one can see these decorations as shortcuts for foams
with extra decorated glued bubbles as explained by Lemma~\ref{lem:bubble}.

\begin{conv} \label{conv:new-dec}
  From now on, decorations of foams are of this more general form.
\end{conv}

\subsection{Euler characteristics of surfaces}
\label{sec:euler-char-surf}

A key ingredient in formula~\eqref{eq:colored-ev} giving the colored
$\gll_\myN$-evaluation is the Euler characteristic of bichrome
surfaces. In this subsection, we inspect how this quantity can be
computed for foams in good position. We will as well
consider monochrome surfaces. Let us fix a surface $\surface$ and consider a colored foam $(\foam,c)$ in
$\surface\times[0,1]$. Suppose furthermore that 
$\foam$ is in good position and as such is a
composition of $\foam^{(1)}, \dots, F^{(\ell)}$. 
We say that
$\foam^{(k)}$, a basic foam which is not the trace of an isotopy,
\emph{involves} a pigment $i$ if $i$ belongs to the color of at least one of the
facets in the non-trivial part of $F^{(k)}$ (see Definition~\ref{dfn:basic-foams}).

\begin{lem}
  \label{lem:cupcap-mono}
  Let $\foam$ be a \spherical{} closed foam in good position, and $c$
  a $\gll_\myN$-coloring of $\foam$.
  For any pigment $i$, the number of cups involving $i$ equals the
  number of caps involving $i$.
\end{lem}
\begin{proof}
  Since $\foam$ is in good position, the projection $\pi$ onto $[0,1]$
  provides a Morse function. Since
  $\foam$ is \spherical, the surface $\foam_i(c)$ contains no
  critical point for $\pi$ of index $1$, hence the number of maxima is
  equal to the number of minima.  This implies the number of cups equals the number of caps.
\end{proof}

Consider a closed foam in good position $\foam$, and $c$ a
$\gll_\myN$-coloring $c$ of $\foam$. Fix $i<j$ two pigments. Since
$\foam$ is in good position, we can present the bichrome surface
$\foam_{ij}(c)$ as a movie. Pieces of $\foam$ which are traces of isotopies correspond to isotopies in the movie of $\foam_{ij}(c)$, and so there are
many basic foams for which nothing really happens in
$\foam_{ij}(c)$.
In Table~\ref{tab:bicol-cob}, we gather the interesting
pieces of $\foam$ and translate them into (movie) Morse moves for
$\foam_{ij}(c)$. Additionally, we make explicit the local contributions of
these various pieces to $\chi(F_{ij}(c))$. Finally, the last column
of the table gives notation to the number of basic foams of each particular type in $(\foam,c)$:
$\Czip_{ij}$, $\Czip_{ji}$, $\Cunzip_{ij}$, $\Cunzip_{ji}$
$\Cdigcup_{ij}$, $\Cdigcup_{ji}$, $\Cdigcap_{ij}$, $\Cdigcap_{ji}$,
$\Ccup_{ij}$, $\Ccup_{ji}$, $\Ccap_{ij}$, $\Ccap_{ji}$.

\begin{table}\begin{tabular}{|c|c|c|c|} 
  \hline
  Basic foam
  &Movie
  &$\chi(F_{ij}(c))$
  & Symbol
  \\\hline
  \NB{\tikz[font=\tiny]{\begin{scope}
  \begin{scope}
    \coordinate (L) at (0,0);
    \coordinate (R) at (2,0);
    \coordinate (ML) at (0.5, 0);
    \coordinate (MR) at (1.5, 0);
 \begin{scope}[yshift = -1.3cm]
    \coordinate (LB) at (0,0);
    \coordinate (RB) at (2,0);
    \draw[->-] (LB) -- (RB);
  \end{scope}
  \fill[gray, opacity=0.2] (LB) -- (RB) -- (R) -- (MR)  .. controls
  +(0, -1) and +(0, -1) .. (ML) -- (L) -- (LB); 
    \fill[blue, opacity = 0.7, pattern=vertical lines, pattern color=blue] (ML).. controls + (0.4, 0.4) and
    +(-0.2, 0.4) .. (MR) .. controls +(0, -1) and +(0, -1) .. (ML);
    \fill[red, fill opacity = 0.3,] (ML).. controls + (0.2, -0.4) and
    +(-0.4, -0.4) .. (MR) .. controls +(0, -1) and +(0, -1) .. (ML);
    \fill[red, opacity = 1, pattern=dots, pattern color=red] (ML).. controls + (0.2, -0.4) and
    +(-0.4, -0.4) .. (MR) .. controls +(0, -1) and +(0, -1) .. (ML);
    \draw[->-] (L) -- (ML);
    \draw[->-] (MR) -- (R) node[right] {$a+b$};
    \draw[->-] (ML).. controls + (0.4, 0.4) and +(-0.2, 0.4) .. (MR)
    node[above, midway] {$a$};
    \draw[->-] (ML).. controls + (0.2, -0.4) and +(-0.4, -0.4) .. (MR)
    node[below, midway] {$b$};
  \end{scope}  
  \draw (R) -- (RB);
  \draw (L) -- (LB);
  \draw[thick] (ML) .. controls +(0, -1) and +(0, -1) .. (MR);
\end{scope}

}}
  &\mymovie[scale=0.5]{\NB{\tikz[yscale=1]{
\begin{scope}[scale=1]
  \draw[->-, gray, densely dotted] (-0.9, 0) -- (0.9, 0);
\end{scope}
}}}{\NB{\tikz[yscale=1]{\begin{scope}[scale=1]
  \draw[ blue, decorate, decoration={zigzag, segment length=1.5mm,  amplitude=0.3mm},
  thick] (-0.5, 0)  arc (180:0:0.5cm);
  \draw[ very thick, red, dotted, line cap= round](-0.5, 0) arc (-180:0:0.5cm);
  \draw[-<-, gray, densely dotted] (-0.5, 0) -- +(-0.4, 0);
  \draw[->-, gray,densely dotted] ( 0.5, 0) -- +( 0.4, 0);
\end{scope}
}}}
  &+1
  &$\Cdigcup_{ij}$
  \\[0.5cm] \hline
  \NB{\tikz[font=\tiny]{\begin{scope}
  \begin{scope}
    \coordinate (L) at (0,0);
    \coordinate (R) at (2,0);
    \coordinate (ML) at (0.5, 0);
    \coordinate (MR) at (1.5, 0);
 \begin{scope}[yshift = -1.3cm]
    \coordinate (LB) at (0,0);
    \coordinate (RB) at (2,0);
    \draw[->-] (LB) -- (RB);
  \end{scope}
  \fill[gray, opacity=0.2] (LB) -- (RB) -- (R) -- (MR)  .. controls
  +(0, -1) and +(0, -1) .. (ML) -- (L) -- (LB); 
    \fill[red, fill opacity = 0.3,] (ML).. controls + (0.4, 0.4) and
    +(-0.2, 0.4) .. (MR) .. controls +(0, -1) and +(0, -1) .. (ML);
    \fill[red, opacity = 1, pattern=dots, pattern color=red] (ML).. controls + (0.4, 0.4) and
    +(-0.2, 0.4) .. (MR) .. controls +(0, -1) and +(0, -1) .. (ML);
    \fill[white, opacity = 0.6] (ML).. controls + (0.2, -0.4) and
    +(-0.4, -0.4) .. (MR) .. controls +(0, -1) and +(0, -1) .. (ML);
    \fill[blue, opacity = 1, pattern=vertical lines, pattern color=blue] (ML).. controls + (0.2, -0.4) and
    +(-0.4, -0.4) .. (MR) .. controls +(0, -1) and +(0, -1) .. (ML);
    \draw[->-] (L) -- (ML);
    \draw[->-] (MR) -- (R) node[right] {$a+b$};
    \draw[->-] (ML).. controls + (0.4, 0.4) and +(-0.2, 0.4) .. (MR)
    node[above, midway] {$a$};
    \draw[->-] (ML).. controls + (0.2, -0.4) and +(-0.4, -0.4) .. (MR)
    node[below, midway] {$b$};
  \end{scope}  
  \draw (R) -- (RB);
  \draw (L) -- (LB);
  \draw[thick] (ML) .. controls +(0, -1) and +(0, -1) .. (MR);
\end{scope}

}}
  &\mymovie[scale =0.5]{\NB{\tikz[yscale=1]{}}}{\NB{\tikz[yscale=-1]{\begin{scope}[scale=1]
  \draw[ blue, decorate, decoration={zigzag, segment length=1.5mm,  amplitude=0.3mm},
  thick] (-0.5, 0)  arc (180:0:0.5cm);
  \draw[ very thick, red, dotted, line cap= round](-0.5, 0) arc (-180:0:0.5cm);
  \draw[-<-, gray, densely dotted] (-0.5, 0) -- +(-0.4, 0);
  \draw[->-, gray,densely dotted] ( 0.5, 0) -- +( 0.4, 0);
\end{scope}
}}}
  &+1
  &$\Cdigcup_{ji}$
  \\[0.5cm]\hline
  \NB{\tikz[font=\tiny]{\begin{scope}
  \begin{scope}
 \begin{scope}[yshift = 1.3cm]
    \coordinate (LB) at (0,0);
    \coordinate (RB) at (2,0);
    \draw (LB) -- (RB);
  \end{scope}  
    \coordinate (L) at (0,0);
    \coordinate (R) at (2,0);
    \coordinate (ML) at (0.5, 0);
    \coordinate (MR) at (1.5, 0);
    \draw[->-] (L) -- (ML);
    \draw[->-] (MR) -- (R) node[right] {$a+b$};
    \draw[->-] (ML).. controls + (0.4, 0.4) and +(-0.2, 0.4) .. (MR)
    node[above, midway] {$a$};
    \draw[->-] (ML).. controls + (0.2, -0.4) and +(-0.4, -0.4) .. (MR)
    node[below, midway] {$b$};
    \fill[gray, opacity=0.2] (LB) -- (RB) -- (R) -- (MR)  .. controls
  +(0, 1) and +(0, 1) .. (ML) -- (L) -- (LB); 
    \fill[white, opacity = 0.4] (ML).. controls + (0.4, 0.4) and
    +(-0.2, 0.4) .. (MR) .. controls +(0, 1) and +(0, 1) .. (ML);
    \fill[blue, opacity = 1, pattern=vertical lines, pattern color=blue] (ML).. controls + (0.4, 0.4) and
    +(-0.2, 0.4) .. (MR) .. controls +(0, 1) and +(0, 1) .. (ML);
    \fill[red, fill opacity = 0.3,] (ML).. controls + (0.2, -0.4) and
    +(-0.4, -0.4) .. (MR) .. controls +(0, 1) and +(0, 1) .. (ML);
    \fill[red, opacity = 1, pattern=dots, pattern color=red] (ML).. controls + (0.2, -0.4) and
    +(-0.4, -0.4) .. (MR) .. controls +(0, 1) and +(0, 1) .. (ML);
  \end{scope}  
  \draw (R) -- (RB);
  \draw (L) -- (LB);
  \draw[thick] (ML) .. controls +(0, 1) and +(0, 1) .. (MR);
\end{scope}

}}
  &\mymovie[scale
    =0.5]{\NB{\tikz[yscale=1]{\begin{scope}[scale=1]
  \draw[ blue, decorate, decoration={zigzag, segment length=1.5mm,  amplitude=0.3mm},
  thick] (-0.5, 0)  arc (180:0:0.5cm);
  \draw[ very thick, red, dotted, line cap= round](-0.5, 0) arc (-180:0:0.5cm);
  \draw[-<-, gray, densely dotted] (-0.5, 0) -- +(-0.4, 0);
  \draw[->-, gray,densely dotted] ( 0.5, 0) -- +( 0.4, 0);
\end{scope}
}}}{\NB{\tikz[yscale=1]{}}}
  &+1
  &$\Cdigcap_{ij}$
  \\[0.5cm]\hline
  \NB{\tikz[font=\tiny]{\begin{scope}
  \begin{scope}
 \begin{scope}[yshift = 1.3cm]
    \coordinate (LB) at (0,0);
    \coordinate (RB) at (2,0);
    \draw (LB) -- (RB);
  \end{scope}  
    \coordinate (L) at (0,0);
    \coordinate (R) at (2,0);
    \coordinate (ML) at (0.5, 0);
    \coordinate (MR) at (1.5, 0);
    \draw[->-] (L) -- (ML);
    \draw[->-] (MR) -- (R) node[right] {$a+b$};
    \draw[->-] (ML).. controls + (0.4, 0.4) and +(-0.2, 0.4) .. (MR)
    node[above, midway] {$a$};
    \draw[->-] (ML).. controls + (0.2, -0.4) and +(-0.4, -0.4) .. (MR)
    node[below, midway] {$b$};
    \fill[gray, opacity=0.2] (LB) -- (RB) -- (R) -- (MR)  .. controls
  +(0, 1) and +(0, 1) .. (ML) -- (L) -- (LB); 
    \fill[red, fill opacity = 0.3,] (ML).. controls + (0.4, 0.4) and
    +(-0.2, 0.4) .. (MR) .. controls +(0, 1) and +(0, 1) .. (ML);
    \fill[red, opacity = 1, pattern=dots, pattern color=red] (ML).. controls + (0.4, 0.4) and
    +(-0.2, 0.4) .. (MR) .. controls +(0, 1) and +(0, 1) .. (ML);
    \fill[white, opacity = 0.6] (ML).. controls + (0.2, -0.4) and
    +(-0.4, -0.4) .. (MR) .. controls +(0, 1) and +(0, 1) .. (ML);
    \fill[blue, opacity = 1, pattern=vertical lines, pattern color=blue] (ML).. controls + (0.2, -0.4) and
    +(-0.4, -0.4) .. (MR) .. controls +(0, 1) and +(0, 1) .. (ML);
  \end{scope}  
  \draw (R) -- (RB);
  \draw (L) -- (LB);
  \draw[thick] (ML) .. controls +(0, 1) and +(0, 1) .. (MR);
\end{scope}

}}
  &\mymovie[scale =0.5]{\NB{\tikz[yscale=-1]{\begin{scope}[scale=1]
  \draw[ blue, decorate, decoration={zigzag, segment length=1.5mm,  amplitude=0.3mm},
  thick] (-0.5, 0)  arc (180:0:0.5cm);
  \draw[ very thick, red, dotted, line cap= round](-0.5, 0) arc (-180:0:0.5cm);
  \draw[-<-, gray, densely dotted] (-0.5, 0) -- +(-0.4, 0);
  \draw[->-, gray,densely dotted] ( 0.5, 0) -- +( 0.4, 0);
\end{scope}
}}}{\NB{\tikz[yscale=1]{}}}
  &+1
  &$\Cdigcap_{ji}$
  \\[0.5cm]\hline
  \NB{\tikz[font=\tiny]{\begin{scope}
  \begin{scope}
    \coordinate (L1) at (0.2,0.4);
    \coordinate (L2) at (0,0);
    \coordinate (R1) at (2.2,0.4);
    \coordinate (R2) at (2,0);
    \coordinate (ML) at (0.6, 0.2);
    \coordinate (MR) at (1.6, 0.2);
  \begin{scope}[yshift = -1cm]
    \coordinate (L1B) at (0.2,0.4);
    \coordinate (L2B) at (0,0);
    \coordinate (R1B) at (2.2,0.4);
    \coordinate (R2B) at (2,0);
  \end{scope}
  \end{scope}  
    \draw[->-] (L1B) .. controls +( 0, 0) and +(0,0) .. (R1B) node [right, pos
    = 1] {$a$};
    \draw[->-] (L2B) .. controls +( 0, 0) and +(0,0) .. (R2B) node [right, pos
    = 1] {$b$};
  \draw (R1) -- (R1B);
  \draw (R2) -- (R2B);
  \draw (L1) -- (L1B);
  \draw (L2) -- (L2B);
  \draw[thick] (ML) .. controls +(0, -0.6) and +(0, -0.6) .. (MR);
    \fill[white, opacity = 0.6] (L1) .. controls +( 0.3, 0) and
    +(0,0) .. (ML) .. controls +(0, -0.6) and +(0, -0.6) .. (MR)
    .. controls +(0, 0) and +(-0.3,0) .. (R1)--(R1B) -- (L1B) -- (L1); 
    \fill[blue, opacity = 1, pattern=vertical lines, pattern
    color=blue] (L1) .. controls +( 0.3, 0) and
    +(0,0) .. (ML) .. controls +(0, -0.6) and +(0, -0.6) .. (MR)
    .. controls +(0, 0) and +(-0.3,0) .. (R1)--(R1B) -- (L1B) -- (L1); 
  \fill[gray, opacity=0.2] (ML) .. controls +(0, -0.6) and +(0, -0.6) .. (MR);
    \fill[red, fill opacity = 0.3] (L2) .. controls +( 0.3, 0) and
    +(0,0) .. (ML) .. controls +(0, -0.6) and +(0, -0.6) .. (MR)
    .. controls +(0, 0) and +(-0.3,0) .. (R2)--(R2B) -- (L2B) -- (L2);
    \fill[red, opacity = 1, pattern=dots, pattern color=red] (L2) .. controls +( 0.3, 0) and
    +(0,0) .. (ML) .. controls +(0, -0.6) and +(0, -0.6) .. (MR)
    .. controls +(0, 0) and +(-0.3,0) .. (R2)--(R2B) -- (L2B) -- (L2);
  \draw[->-] (ML) -- (MR) node[above, midway] {$a+b$};
    \draw (MR) .. controls +(0, 0) and +(-0.3,0) .. (R1) ;
    \draw (MR) .. controls +(0, 0) and +(-0.3,0) .. (R2);
    \draw (L1) .. controls +( 0.3, 0) and +(0,0) .. (ML);
    \draw (L2) .. controls +( 0.3, 0) and +(0,0) .. (ML);
\end{scope}

}}
  &\mymovie[scale =0.5]{\NB{\tikz[yscale=1]{\begin{scope}[scale=1]
  \draw[ blue, decorate, decoration={zigzag, segment length=1.5mm,  amplitude=0.3mm},
  thick] (-0.7, 0.3) -- +(1.4,0);
  \draw[ very thick, red, dotted, line cap= round]  (-0.7, -0.3) -- +(1.4,0);
\end{scope}
}}}{\NB{\tikz[yscale=1]{\begin{scope}[scale=1]
  \draw[ blue, decorate, decoration={zigzag, segment length=1.5mm,  amplitude=0.3mm},
  thick]
  (-0.7, 0.3)
  .. controls +(0.3, 0) and +(-0.1,0.1) ..
  (-0.2, 0);
  \draw[ blue, decorate, decoration={zigzag, segment length=1.5mm,  amplitude=0.3mm},
  thick] (0.7, 0.3)
   .. controls +(-0.3, 0) and +(0.1,0.1) ..
  ( 0.2, 0);
  \draw[ very thick, red, dotted, line cap= round]
  (-0.7, -0.3) .. controls +(0.3, 0) and +(-0,0) .. (-0.2, 0);
  \draw[ very thick, red, dotted, line cap= round] 
  (0.7, -0.3) .. controls +(-0.3, 0) and +(0 ,-0) .. ( 0.2, 0);
  \draw[->-, gray, densely dotted] (-0.2, 0) -- +(0.4, 0);
\end{scope}
}}}
  &-1
  &$\Czip_{ij}$
  \\ [0.5cm]\hline
  \NB{\tikz[font=\tiny]{\begin{scope}
  \begin{scope}
    \coordinate (L1) at (0.2,0.4);
    \coordinate (L2) at (0,0);
    \coordinate (R1) at (2.2,0.4);
    \coordinate (R2) at (2,0);
    \coordinate (ML) at (0.6, 0.2);
    \coordinate (MR) at (1.6, 0.2);
  \begin{scope}[yshift = -1cm]
    \coordinate (L1B) at (0.2,0.4);
    \coordinate (L2B) at (0,0);
    \coordinate (R1B) at (2.2,0.4);
    \coordinate (R2B) at (2,0);
  \end{scope}
  \end{scope}  
    \draw[->-] (L1B) .. controls +( 0, 0) and +(0,0) .. (R1B) node [right, pos
    = 1] {$a$};
    \draw[->-] (L2B) .. controls +( 0, 0) and +(0,0) .. (R2B) node [right, pos
    = 1] {$b$};
  \draw (R1) -- (R1B);
  \draw (R2) -- (R2B);
  \draw (L1) -- (L1B);
  \draw (L2) -- (L2B);
  \draw[thick] (ML) .. controls +(0, -0.6) and +(0, -0.6) .. (MR);
  \fill[gray, opacity=0.2] (ML) .. controls +(0, -0.6) and +(0, -0.6) .. (MR);
    \fill[red, fill opacity = 0.3] (L1) .. controls +( 0.3, 0) and
    +(0,0) .. (ML) .. controls +(0, -0.6) and +(0, -0.6) .. (MR)
    .. controls +(0, 0) and +(-0.3,0) .. (R1)--(R1B) -- (L1B) -- (L1);
    \fill[red, opacity = 1, pattern=dots, pattern color=red] (L1) .. controls +( 0.3, 0) and
    +(0,0) .. (ML) .. controls +(0, -0.6) and +(0, -0.6) .. (MR)
    .. controls +(0, 0) and +(-0.3,0) .. (R1)--(R1B) -- (L1B) -- (L1);
    \fill[white, opacity = 0.6] (L2) .. controls +( 0.3, 0) and
    +(0,0) .. (ML) .. controls +(0, -0.6) and +(0, -0.6) .. (MR)
    .. controls +(0, 0) and +(-0.3,0) .. (R2)--(R2B) -- (L2B) -- (L2); 
    \fill[blue, opacity = 1, pattern=vertical lines, pattern
    color=blue] (L2) .. controls +( 0.3, 0) and
    +(0,0) .. (ML) .. controls +(0, -0.6) and +(0, -0.6) .. (MR)
    .. controls +(0, 0) and +(-0.3,0) .. (R2)--(R2B) -- (L2B) -- (L2); 
  \draw[->-] (ML) -- (MR) node[above, midway] {$a+b$};
    \draw (MR) .. controls +(0, 0) and +(-0.3,0) .. (R1) ;
    \draw (MR) .. controls +(0, 0) and +(-0.3,0) .. (R2);
    \draw (L1) .. controls +( 0.3, 0) and +(0,0) .. (ML);
    \draw (L2) .. controls +( 0.3, 0) and +(0,0) .. (ML);
\end{scope}

}}
  &\mymovie[scale =0.5]{\NB{\tikz[yscale=-1]{\begin{scope}[scale=1]
  \draw[ blue, decorate, decoration={zigzag, segment length=1.5mm,  amplitude=0.3mm},
  thick] (-0.7, 0.3) -- +(1.4,0);
  \draw[ very thick, red, dotted, line cap= round]  (-0.7, -0.3) -- +(1.4,0);
\end{scope}
}}}{\NB{\tikz[yscale=-1]{}}}
  &-1
  &$\Czip_{ji}$
  \\ [0.5cm]\hline
  \NB{\tikz[font=\tiny]{\begin{scope}
  \begin{scope}
    \coordinate (L1) at (0.2,0.4);
    \coordinate (L2) at (0,0);
    \coordinate (R1) at (2.2,0.4);
    \coordinate (R2) at (2,0);
    \coordinate (ML) at (0.6, 0.2);
    \coordinate (MR) at (1.6, 0.2);
    \draw[->-] (ML) -- (MR) node[below, midway] {$a+b$};
    \draw (MR) .. controls +(0, 0) and +(-0.3,0) .. (R1) ;
    \draw (MR) .. controls +(0, 0) and +(-0.3,0) .. (R2);
    \draw (L1) .. controls +( 0.3, 0) and +(0,0) .. (ML);
    \draw (L2) .. controls +( 0.3, 0) and +(0,0) .. (ML);
  \end{scope}  
 \begin{scope}[yshift = 1cm]
    \coordinate (L1B) at (0.2,0.4);
    \coordinate (L2B) at (0,0);
    \coordinate (R1B) at (2.2,0.4);
    \coordinate (R2B) at (2,0);
    \draw[->-] (L1B) .. controls +( 0, 0) and +(0,0) .. (R1B) node [left, pos
    = 0] {$a$};
    \draw[->-] (L2B) .. controls +( 0, 0) and +(0,0) .. (R2B) node [left, pos
    = 0] {$b$};
 \end{scope}  
  \draw (R1) -- (R1B);
  \draw (R2) -- (R2B);
  \draw (L1) -- (L1B);
  \draw (L2) -- (L2B);
  \draw[thick] (ML) .. controls +(0, 0.6) and +(0, 0.6) .. (MR);
  \fill[gray, opacity=0.2] (ML) .. controls +(0, 0.6) and +(0, 0.6) .. (MR);
    \fill[white, opacity = 0.6] (L1) .. controls +( 0.3, 0) and
    +(0,0) .. (ML) .. controls +(0, 0.6) and +(0, 0.6) .. (MR)
    .. controls +(0, 0) and +(-0.3,0) .. (R1)--(R1B) -- (L1B) -- (L1); 
    \fill[blue, opacity = 1, pattern=vertical lines, pattern
    color=blue] (L1) .. controls +( 0.3, 0) and
    +(0,0) .. (ML) .. controls +(0, 0.6) and +(0, 0.6) .. (MR)
    .. controls +(0, 0) and +(-0.3,0) .. (R1)--(R1B) -- (L1B) -- (L1); 
    \fill[red, fill opacity = 0.3] (L2) .. controls +( 0.3, 0) and
    +(0,0) .. (ML) .. controls +(0, 0.6) and +(0, 0.6) .. (MR)
    .. controls +(0, 0) and +(-0.3,0) .. (R2)--(R2B) -- (L2B) -- (L2);
    \fill[red, opacity = 1, pattern=dots, pattern color=red] (L2) .. controls +( 0.3, 0) and
    +(0,0) .. (ML) .. controls +(0, 0.6) and +(0, 0.6) .. (MR)
    .. controls +(0, 0) and +(-0.3,0) .. (R2)--(R2B) -- (L2B) -- (L2);

\end{scope}

}}
  &\mymovie[scale =0.5]{\NB{\tikz[yscale=1]{}}}{\NB{\tikz[yscale=1]{\begin{scope}[scale=1]
  \draw[ blue, decorate, decoration={zigzag, segment length=1.5mm,  amplitude=0.3mm},
  thick] (-0.7, 0.3) -- +(1.4,0);
  \draw[ very thick, red, dotted, line cap= round]  (-0.7, -0.3) -- +(1.4,0);
\end{scope}
}}}
  &-1
  &$\Cunzip_{ij}$
  \\ [0.5cm]\hline
  \NB{\tikz[font=\tiny]{\begin{scope}
  \begin{scope}
    \coordinate (L1) at (0.2,0.4);
    \coordinate (L2) at (0,0);
    \coordinate (R1) at (2.2,0.4);
    \coordinate (R2) at (2,0);
    \coordinate (ML) at (0.6, 0.2);
    \coordinate (MR) at (1.6, 0.2);
    \draw[->-] (ML) -- (MR) node[below, midway] {$a+b$};
    \draw (MR) .. controls +(0, 0) and +(-0.3,0) .. (R1) ;
    \draw (MR) .. controls +(0, 0) and +(-0.3,0) .. (R2);
    \draw (L1) .. controls +( 0.3, 0) and +(0,0) .. (ML);
    \draw (L2) .. controls +( 0.3, 0) and +(0,0) .. (ML);
  \end{scope}  
 \begin{scope}[yshift = 1cm]
    \coordinate (L1B) at (0.2,0.4);
    \coordinate (L2B) at (0,0);
    \coordinate (R1B) at (2.2,0.4);
    \coordinate (R2B) at (2,0);
    \draw[->-] (L1B) .. controls +( 0, 0) and +(0,0) .. (R1B) node [left, pos
    = 0] {$a$};
    \draw[->-] (L2B) .. controls +( 0, 0) and +(0,0) .. (R2B) node [left, pos
    = 0] {$b$};
 \end{scope}  
  \draw (R1) -- (R1B);
  \draw (R2) -- (R2B);
  \draw (L1) -- (L1B);
  \draw (L2) -- (L2B);
  \draw[thick] (ML) .. controls +(0, 0.6) and +(0, 0.6) .. (MR);
  \fill[gray, opacity=0.2] (ML) .. controls +(0, 0.6) and +(0, 0.6) .. (MR);
    \fill[red, fill opacity = 0.3] (L1) .. controls +( 0.3, 0) and
    +(0,0) .. (ML) .. controls +(0, 0.6) and +(0, 0.6) .. (MR)
    .. controls +(0, 0) and +(-0.3,0) .. (R1)--(R1B) -- (L1B) -- (L1);
    \fill[red, opacity = 1, pattern=dots, pattern color=red] (L1) .. controls +( 0.3, 0) and
    +(0,0) .. (ML) .. controls +(0, 0.6) and +(0, 0.6) .. (MR)
    .. controls +(0, 0) and +(-0.3,0) .. (R1)--(R1B) -- (L1B) -- (L1);
    \fill[white, opacity = 0.6] (L2) .. controls +( 0.3, 0) and
    +(0,0) .. (ML) .. controls +(0, 0.6) and +(0, 0.6) .. (MR)
    .. controls +(0, 0) and +(-0.3,0) .. (R2)--(R2B) -- (L2B) -- (L2); 
    \fill[blue, opacity = 1, pattern=vertical lines, pattern
    color=blue] (L2) .. controls +( 0.3, 0) and
    +(0,0) .. (ML) .. controls +(0, 0.6) and +(0, 0.6) .. (MR)
    .. controls +(0, 0) and +(-0.3,0) .. (R2)--(R2B) -- (L2B) -- (L2);

\end{scope}

}}
  &\mymovie[scale =0.5]{\NB{\tikz[yscale=-1]{}}}{\NB{\tikz[yscale=-1]{\begin{scope}[scale=1]
  \draw[ blue, decorate, decoration={zigzag, segment length=1.5mm,  amplitude=0.3mm},
  thick] (-0.7, 0.3) -- +(1.4,0);
  \draw[ very thick, red, dotted, line cap= round]  (-0.7, -0.3) -- +(1.4,0);
\end{scope}
}}}
  &-1
  &$\Cunzip_{ji}$
  \\ [0.5cm]\hline
  \NB{\tikz[font=\tiny]{\begin{scope}
  \fill[red, fill opacity = 0.3] (0,0) arc (180:0: 0.5cm and 0.2cm)
  arc(0:180: 0.5cm and -0.6cm);
  \fill[red, opacity = 1, pattern=dots, pattern color=red] (0,0) arc (180:0: 0.5cm and 0.2cm)
  arc(0:180: 0.5cm and -0.6cm);
  \fill[red, fill opacity = 0.3] (0,0) arc (180:0: 0.5cm and -0.2cm)
  arc(0:180: 0.5cm and -0.6cm);
  \fill[red, opacity = 1, pattern=dots, pattern color=red] (0,0) arc (180:0: 0.5cm and -0.2cm)
  arc(0:180: 0.5cm and -0.6cm);
  \draw (0,0) arc (180 :0: 0.5cm and 0.2cm) node[above, pos =
  0.5] {$a$};
  \draw[very thin] (0,0) arc (180 :0: 0.5cm and -0.6cm);
  \draw (0,0) arc (180 :0: 0.5cm and -0.2cm);
\end{scope}

}}
  &\mymovie[scale =0.3]{$\emptyset$}{\NB{\tikz[scale=0.6]{
\begin{scope}[scale=1]
  \draw[ very thick, red, dotted, line cap= round] (-0.5, 0)  arc (180:-180:0.5cm);
\end{scope}
}}}
  &+1
  &$\Ccup_{ij}$
  \\ [0.5cm]\hline
  \NB{\tikz[font=\tiny]{\begin{scope}
  \fill[white, opacity = 0.6] (0,0) arc (180:0: 0.5cm and 0.2cm)
  arc(0:180: 0.5cm and -0.6cm);
  \fill[blue, opacity = 1, pattern=vertical lines, pattern
    color=blue] (0,0) arc (180:0: 0.5cm and 0.2cm)
  arc(0:180: 0.5cm and -0.6cm);
  \fill[white, opacity = 0.6] (0,0) arc (180:0: 0.5cm and -0.2cm)
  arc(0:180: 0.5cm and -0.6cm);
  \fill[blue, opacity = 1, pattern=horizontal lines, pattern
    color=blue] (0,0) arc (180:0: 0.5cm and -0.2cm)
  arc(0:180: 0.5cm and -0.6cm);
  \draw (0,0) arc (180 :0: 0.5cm and 0.2cm) node[above, pos =
  0.5] {$a$};
  \draw[very thin] (0,0) arc (180 :0: 0.5cm and -0.6cm);
  \draw (0,0) arc (180 :0: 0.5cm and -0.2cm);
\end{scope}

}}
  &\mymovie[scale =0.3]{$\emptyset$}{\NB{\tikz[scale=0.6]{
\begin{scope}[scale=1]
   \draw[ blue, decorate, decoration={zigzag, segment length=1.5mm,  amplitude=0.3mm},
   thick] (0, 0) circle (0.5cm);
\end{scope}
}}}
  &+1
  &$\Ccup_{ji}$
  \\ [0.5cm]\hline
  \NB{\tikz[font=\tiny]{\begin{scope}
  \fill[red, fill opacity = 0.3] (0,0) arc (180:0: 0.5cm and 0.2cm)
  arc(0:180: 0.5cm and  0.6cm);
  \fill[red, opacity = 1, pattern=dots, pattern color=red] (0,0) arc (180:0: 0.5cm and 0.2cm)
  arc(0:180: 0.5cm and  0.6cm);
  \fill[red, fill opacity = 0.3] (0,0) arc (180:0: 0.5cm and -0.2cm)
  arc(0:180: 0.5cm and  0.6cm);
  \fill[red, opacity = 1, pattern=dots, pattern color=red] (0,0) arc (180:0: 0.5cm and -0.2cm)
  arc(0:180: 0.5cm and  0.6cm);
  \draw (0,0) arc (180 :0: 0.5cm and 0.2cm);
  \draw[very thin] (0,0) arc (180 :0: 0.5cm and 0.6cm);
  \draw (0,0) arc (180 :0: 0.5cm and -0.2cm) node[below, pos =
  0.5] {$a$};
\end{scope}

}}
  &\mymovie[scale =0.3]{\NB{\tikz[scale=0.6]{}}}{$\emptyset$}
  & +1
  &$\Ccap_{ij}$
  \\ [0.5cm]\hline
  \NB{\tikz[font=\tiny]{\begin{scope}
  \fill[white, opacity = 0.6] (0,0) arc (180:0: 0.5cm and 0.2cm)
  arc(0:180: 0.5cm and 0.6cm);
  \fill[blue, opacity = 1, pattern=vertical lines, pattern
    color=blue] (0,0) arc (180:0: 0.5cm and 0.2cm)
  arc(0:180: 0.5cm and 0.6cm);
  \fill[white, opacity = 0.6] (0,0) arc (180:0: 0.5cm and -0.2cm)
  arc(0:180: 0.5cm and  0.6cm);
  \fill[blue, opacity = 1, pattern=horizontal lines, pattern
    color=blue] (0,0) arc (180:0: 0.5cm and -0.2cm)
  arc(0:180: 0.5cm and  0.6cm);
  \draw (0,0) arc (180 :0: 0.5cm and 0.2cm);
  \draw[very thin] (0,0) arc (180 :0: 0.5cm and 0.6cm);
  \draw (0,0) arc (180 :0: 0.5cm and -0.2cm)node[below, pos =
  0.5] {$a$};
\end{scope}

}}
  &\mymovie[scale =0.3]{\NB{\tikz[scale=0.6]{}}}{$\emptyset$}
  & +1
  &$\Ccap_{ji}$
  \\ [0.5cm]\hline
\end{tabular}
\caption{In this table, the red dotted regions/lines represent facets 
  with $i$ in their colors in $F_{ij}(c)$; the blue hashed surface/zigzag lines
  represent facets with $j$ in their colors.}
\label{tab:bicol-cob}
\end{table}
Note that depending on the orientations and local configurations of $i$ and
$j$, each basic foam of interest comes in two flavors. 

With these notations, if $F$ is a spherical foam in good position, one has:
\begin{equation}
\label{eq:euler-basic}
  \chi(\foam_{ij}(c))= \Ccap_{ij} + \Ccap_{ji} + \Ccup_{ij}+\Ccup_{ji}
   + \Cdigcap_{ij}  + \Cdigcap_{ji} - \Czip_{ij} - \Czip_{ji}\\
   + \Cdigcup_{ij}  + \Cdigcup_{ji}-\Cunzip_{ij} - \Cunzip_{ji} .
\end{equation}

\begin{lem}\label{lem:cupcap-bi}
  For any \spherical{} foam $\foam$ in good position, any coloring $c$
  and any two pigments $i$ and $j$, the following identities hold:
  \begin{align}
    \Ccup_{ij} + \Ccap_{ji}= \Ccup_{ji} + \Ccap_{ij}. \end{align}
\end{lem}

\begin{proof}
  Let us denote momentarily $\Ccup$ and $\Ccap$ for the number of caps and
  cups involving both $i$ and $j$. 
  Lemma~\ref{lem:cupcap-mono}, gives:
  \begin{align*}
    \Ccup + \Ccup_{ij} &= \Ccap + \Ccap_{ij} \qquad \text{and} \\
    \Ccup + \Ccup_{ji} &= \Ccap + \Ccap_{ji}.
  \end{align*}
  The identity of the lemma follows from the difference of these two identities.
\end{proof}

\begin{lem}\label{lem:zip-unzip-merge-split}
  For any (not-necessarily \spherical{}) foam $\foam$ in good position, any coloring $c$
  and any two pigments $i$ and $j$, the following identities hold:
  \label{lem:merge-zip-in-bichrome}
  \begin{align}
    \Czip_{ij} + \Cdigcup_{ij} &=  \Cunzip_{ij} +
                                 \Cdigcap_{ij}, \label{eq:Euler1} \\
    \Czip_{ji} + \Cdigcup_{ji} &=  \Cunzip_{ji} + \Cdigcap_{ji}. \label{eq:Euler2}
  \end{align}
\end{lem}

\begin{proof}
  In the movie describing the surface $\foam_{ij}(c)$, we can track
  the split vertices at which pigment $i$ and pigment $j$
  meet. Keeping track of orientations, they come in two flavors:
  \[
    \NB
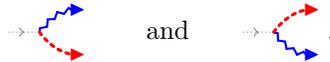
{\tikz[]{\begin{scope}[scale=1]
  \draw[ blue, decorate, decoration={zigzag, segment length=1.5mm,  amplitude=0.3mm},
  thick] (0.7, 0.3)
   .. controls +(-0.3, 0) and +(0.1,0.1) ..
  ( 0.2, 0)node[pos=0, blue, scale=0.8] {$\blacktriangleright$};
  \draw[ very thick, red, dotted, line cap= round, ] 
  (0.7, -0.3) .. controls +(-0.3, 0) and +(0 ,-0) .. ( 0.2, 0)node[pos=0, red, scale=0.8] {$\blacktriangleright$};
  \draw[->-, gray, densely dotted] (-0.2, 0) -- +(0.4, 0);
\end{scope}
}}\qquad \text{and} \qquad
    \NB{\tikz[yscale=-1]{}}.
  \]
  Since
  $\foam_{ij}(c)$ is a closed surface, at the beginning and at the end
  of the movie, there are no such vertices. Hence the number of births of
  these vertices is equal to the number of deaths of them. The
  identities \eqref{eq:Euler1} and \eqref{eq:Euler2} reflect that fact for each of the two flavors.
\end{proof}

Recall from Section \ref{sec:conventions} the convention that for $s$ in $\scalars$, $\bar{s} = 1 -s$. 

\begin{cor} 
For any \spherical{} foam $\foam$ in good position, any coloring $c$, any two pigments $i$ and $j$ and any $s$ in $\scalars$, the following identities hold:
\begin{equation}
    \begin{split}
    \frac{\chi(\foam_{ij}(c))}{2}=
 & \frac12\left(\Ccap_{ij}  + \Ccup_{ij} + \Ccap_{ji}
   +\Ccup_{ji}\right) \\ &+s\left(\Cdigcap_{ij}+\Cdigcap_{ji}-\Czip_{ij}-\Czip_{ji}\right)
 +\bar{s}\left(\Cdigcup_{ij}+\Cdigcup_{ji}-\Cunzip_{ij}-\Cunzip_{ji}\right) .\end{split}
\end{equation}
\end{cor}

\subsection{\texorpdfstring{$\gll_\myN$}{gl(N)}-state spaces}
\label{sec:state-spaces}

In this section, we briefly recall the universal construction of TQFTs
\cite{BHMV-tqft} in our context and construct $\gll_\myN$-state spaces
associated with webs as well as functors from foamy categories to
algebraic categories. We follow \cite{RW1}.
From now on we only consider webs in $\RR^2$ and $\foamcat$ refers to $\foamcat[\RR^2]$.

Let $\web$ be a web and denote by $\protostatespaceN{\web}$ the free
$\RN$-module generated by $\Hom_{\foamcat}\left(\emptyset, \web
\right)$. It is graded by the degree of foams \eqref{eq:deg-foam}.
Consider the $\RN$-bilinear form $\bracketN{\cdot;\cdot}$ on
$\protostatespaceN{\web}$ defined on
foams by:
\begin{equation}
  \label{eq:6}
  \bracketN{\foam;G} := \bracketN{\overline{G}\circ \foam},
\end{equation}
where $\overline{G}$ is the foam $G\co \web \to \emptyset$
obtained by mirroring $G$ along
$\RR^2\times\left\{\frac{1}{2}\right\}$, so that
$\overline{G}\circ
\foam$  is a closed foam and $\bracketN{\cdot; \cdot}$ is well-defined.

For any web in $\RR^2$, define $\statespaceN{\web}$ to be the quotient:
\begin{equation}
  \label{eq:2}
  \statespaceN{\web} := \protostatespaceN{\web}\Big/\Ker \bracketN{\cdot;\cdot}.
\end{equation}

As part of the universal construction, this extends for free to a
functor $\tqftfunc\co \foamcat \to \RN\mathsf{-Mod}_{\mathrm{gr}}$. The categories $\foamcat$ and $\RN\mathsf{-Mod}_{\mathrm{gr}}$ are both endowed with a monoidal structure. In $\foamcat$, the tensor product is given by disjoint union of webs and foams. The tensor product on $\RN\mathsf{-Mod}_{\mathrm{gr}}$ is taking tensor product over $\RN$ (using the commutativity of $\RN$, one can view any $\RN$-module as an $(\RN,\RN)$-bimodule).

\begin{exa}[Dot migration] \label{exa:dot-migration}
  Let $R \in \ZZ[x_1, \dots, x_a, y_1, \dots, y_b]^{S_{a+b}}$ be a
  polynomial in $a+b$ variables. Since
  \[
    \ZZ[x_1, \dots, x_a, y_1, \dots, y_b]^{S_{a+b}}
\subseteq \ZZ[x_1, \dots, x_a, y_1, \dots, y_b]^{S_{a}\times S_{b}},
\]
the latter space being isomorphic to $\ZZ[x_1, \dots, x_a]^{S_a} \otimes_\ZZ \ZZ[y_1, \dots, y_b]^{S_b}$,
one can write:
\[
  R = \sum_{j=1}^k R^{(1)}_j \otimes R^{(2)}_j
\]
with $R_j^{(1)}$ in  $\ZZ[x_1, \dots, x_a]^{S_a}$ and $R_j^{(2)}$ in
$\ZZ[y_1, \dots, y_b]^{S_b}$ for all $1\leq j \leq k$.

The following local relations\footnote{\label{ft:loc-rel}By \emph{local relation}, we mean
  that given a collection of foams which make sense in the given
  context (here they should be closed) and  are identical except in a
  ball where they are given by the model in the relation, then the
  relation holds for these foams.} hold for the colored evaluation:
\begin{equation}
  \label{eq:17}
  \bracketN{\ \NB{\tikz[]{\begin{scope}[font =\tiny, scale =1.5]
  \coordinate (OT)  at ( 0,   0);
  \coordinate (ABT) at (-1,   0);
  \coordinate (BT)  at ( 1, .5);
  \coordinate (AT)  at ( 1,-.5);
  \coordinate (OB)  at ( 0,   -.7);
  \coordinate (ABB) at (-1,   -.7);
  \coordinate (BB)  at ( 1, -0.2);
  \coordinate (AB)  at ( 1,-1.2);
  \draw[->-] (ABT) -- (OT) node[above, pos=0.3] {$a+b$};
  \draw[->-] (OT) -- (BT) node[above, pos=0.3] {$b$};
  \draw[thin] (BB) -- (BT);
  \draw[->-] (OB) -- (BB);
  \fill[opacity=0.7, white] (OT) -- (OB) -- (AB) -- (AT) -- cycle;
  \draw[->-] (OT) -- (AT);
  \draw[->-] (ABB) -- (OB);
  \draw[->-] (OB) -- (AB)  node[below, pos=0.3] {$a$};
  \draw[thick] (OB) -- (OT);
  \draw[thin] (ABB) -- (ABT);
  \draw[thin] (AB) -- (AT);
  \fill (-0.5, -0.2) circle (0.3mm) node [below] {$R$};
\end{scope}

}}\ , c} = \sum_{j=1}^k
  \bracketN{\ \NB{\tikz[]{\begin{scope}[font =\tiny, scale =1.5]
  \coordinate (OT)  at ( 0,   0);
  \coordinate (ABT) at (-1,   0);
  \coordinate (BT)  at ( 1, .5);
  \coordinate (AT)  at ( 1,-.5);
  \coordinate (OB)  at ( 0,   -.7);
  \coordinate (ABB) at (-1,   -.7);
  \coordinate (BB)  at ( 1, -0.2);
  \coordinate (AB)  at ( 1,-1.2);
  \draw[->-] (ABT) -- (OT) node[above, pos=0.3] {$a+b$};
  \draw[->-] (OT) -- (BT) node[above, pos=0.3] {$b$};
  \draw[thin] (BB) -- (BT);
  \draw[->-] (OB) -- (BB);
  \fill[opacity=0.7, white] (OT) -- (OB) -- (AB) -- (AT) -- cycle;
  \draw[->-] (OT) -- (AT);
  \draw[->-] (ABB) -- (OB);
  \draw[->-] (OB) -- (AB)  node[below, pos=0.3] {$a$};
  \draw[thick] (OB) -- (OT);
  \draw[thin] (ABB) -- (ABT);
  \draw[thin] (AB) -- (AT);
  \fill (0.7, 0.2) circle (0.3mm) node [below] {$R_j^{(2)}$};
  \fill (0.7, -0.6) circle (0.3mm) node [below] {$R_j^{(1)}$};
\end{scope}

}}\ , c}.
\end{equation}
This follows from the very definition of the colored
evaluation. Therefore, the same relation holds for the evaluation:
  \begin{equation}
\bracketN{\ \NB{\tikz[]{}}\ } = \sum_{j=1}^k \bracketN{\
        \NB{\tikz[]{}}\ }.
  \end{equation}
This implies that for any web $\web$, the following
local\footnote{Following footnote \ref{ft:loc-rel}, here to \emph{make
    sense} means that foams have boundary equal to $\web$. } relation holds in $\statespaceN{\web}$:
  \begin{equation}
\NB{\tikz[]{}}\  = \sum_{j=1}^k \ \NB{\tikz[]{}}\ .
  \end{equation}
  Above, foams represent their equivalence classes in
  $\statespaceN{\web}$. Many other local relations can be found in
  \cite[Section 3]{RW1}.
\end{exa}

We can do the same construction restricting to \spherical{} foams. In
that case, the functor obtained is denoted $\tqftfunc[s]$. It is not
obvious \emph{a priori} how to compare these two functors. Indeed
$\protostatespaceN[s]{\web}$ is contained in
$\protostatespaceN{\web}$, but it is then modded out by a smaller space
than that for constructing $\statespaceN{\web}$. It is not clear
whether or not the functor $\tqftfunc[s]$ is monoidal.

\begin{prop}[{\cite[Proof of Theorem 3.30]{RW1}}] \label{prop:cat-MOY-calc}
  The functor $\tqftfunc$ is monoidal and satisfies the following
  local relations (and their mirror images):
  \begin{align}
    \statespaceN{\emptyset} \cong& \RN \label{eq:cat-empty};\\
    \statespaceN{\NB{\tikz[font=\tiny]{\begin{scope}
  \draw[->-] (0,0) arc (0:360:0.3) node[midway, left] {$a$};
\end{scope}
}}}\cong & \qbinom{N}{a} \statespaceN{\emptyset}
                                                  \label{eq:cat-circle};\\
		\statespaceN{\NB{\tikz[font=\tiny]{\begin{scope}
  \coordinate (a) at (-1, 0.5);
  \coordinate (b) at (-1, 0 );
  \coordinate (c) at (-1,-0.5);
  \coordinate (abc) at (0.4, 0);
  \coordinate (r) at (1, 0);
  \coordinate (ab) at (-0.3, 0.25);
  \draw[>-] (c) -- (abc) node[pos =0, left] {$c$};
  \draw[>-] (a) -- (ab) node[pos =0, left] {$a$};
  \draw[>-] (b) -- (ab) node[pos =0, left] {$b$};
  \draw[->-] (ab) --(abc) node[pos=0.5, above] {$a+b$};
  \draw[->] (abc) -- (r) node[pos =1, right] {$a+b+c$};
\end{scope}}}}
			\cong&
		\statespaceN{\NB{\tikz[font=\tiny, xscale= 0.9]{\begin{scope}
  \coordinate (a) at (-1, 0.5);
  \coordinate (b) at (-1, 0 );
  \coordinate (c) at (-1,-0.5);
  \coordinate (abc) at (0.4, 0);
  \coordinate (r) at (1, 0);
  \coordinate (bc) at (-0.3, -0.25);
  \draw[>-] (c) -- (bc) node[pos =0, left] {$c$};
  \draw[>-] (a) -- (abc) node[pos =0, left] {$a$};
  \draw[>-] (b) -- (bc) node[pos =0, left] {$b$};
  \draw[->-] (bc) --(abc) node[pos=0.5, below] {$b+c$};
  \draw[->] (abc) -- (r) node[pos =1, right] {$a+b+c$};
\end{scope}}}} \label{eq:cat-assoc};
	\\
		\statespaceN{\NB{\tikz[font=\tiny,xscale=1.2]{\begin{scope}
\draw[>-] (0,0) -- +(0.2,0) node[pos=0, left]
      {$a+b$};
      \draw[->] (0.8,0) -- +(0.2, 0) node[pos=1, right]
      {$a+b$};
      \draw[->-] (0.2, 0) .. controls +(0.3,0.3) and +(-0.3, 0.3)
      .. (0.8,0) node[pos =0.5, above] {$a$};
      \draw[->-] (0.2, 0) .. controls +(0.3,-0.3) and +(-0.3, -0.3)
      .. (0.8,0) node[pos =0.5, below] {$b$};
\end{scope}}}} 
			\cong&
		\qbinom{a+b}{a}
                   \statespaceN{\NB{\tikz[font=\tiny,scale=1.2]{\begin{scope}
  \draw[->-] (0,0) -- +(1,0) node[pos= 0.5, above] {$a+b$} node[pos=
  0.5, below, white] {$a+b$};
\end{scope}}}} \label{eq:cat-digon};
	\\
		\statespaceN{\NB{\tikz[font=\tiny,xscale=1.2]{\begin{scope}
\draw[>-] (0,0) -- +(0.2,0) node[pos=0, left]
      {$a$};
      \draw[->] (0.8,0) -- +(0.2, 0) node[pos=1, right]
      {$a$};
      \draw[-<-] (0.2, 0) .. controls +(-0.1,0.4) and +( 0.1, 0.4)
      .. (0.8,0) node[pos =0.5, above] {$b$};
      \draw[->-] (0.2, 0) .. controls +(0.3,-0.1) and +(-0.3, -0.1)
      .. (0.8,0) node[pos =0.5, below] {$a+b$};
\end{scope}
}}} 
			\cong&
		\qbinom{N-a}{b}
                   \statespaceN{\NB{\tikz[font=\tiny,scale=1.2]{\begin{scope}
  \draw[->-] (0,0) -- +(1,0) node[pos= 0.5, above] {$a$} node[pos=
  0.5, below, white] {$a$};
\end{scope}}}} \label{eq:cat-bad-digon};
	\\
    \begin{split}
      \statespaceN{\NB{\tikz[font=\tiny,xscale=1]{\begin{scope}
  \coordinate (t1) at (-1, 0.5);
  \coordinate (t2) at (-0.3, 0.5);
  \coordinate (t3) at ( 0.3, 0.5);
  \coordinate (t4) at ( 1, 0.5);
  \coordinate (b1) at (-1,   -0.5);
  \coordinate (b2) at (-0.5, -0.5);
  \coordinate (b3) at ( 0.5, -0.5);
  \coordinate (b4) at ( 1,   -0.5);
  \draw[<-]  (t1) -- (t2) node[pos =0, left] {$1$};
  \draw[->-] (t2) -- (t3) node[pos =0.5, above] {$m$};
  \draw[-<]  (t3) -- (t4) node[pos =1, right] {$1$};
  \draw[>-]  (b1) -- (b2) node[pos =0, left] {$m$};
  \draw[-<-] (b2) -- (b3) node[pos =0.5, below] {$1$};
  \draw[->]  (b3) -- (b4) node[pos =1, right] {$m$};
  \draw[->-] (b2) -- (t2) node[pos = 0.5, left] {$m+1$};
  \draw[->-] (t3) -- (b3) node[pos = 0.5, right] {$m+1$};
\end{scope}
}}} \cong&
      \statespaceN{\NB{\tikz[font=\tiny,scale=1]{\begin{scope}
  \coordinate (t1) at (-1, 0.5);
  \coordinate (t2) at (-0.3, 0.5);
  \coordinate (t3) at ( 0.3, 0.5);
  \coordinate (t4) at ( 1, 0.5);
  \coordinate (b1) at (-1,   -0.5);
  \coordinate (b2) at (-0.5, -0.5);
  \coordinate (b3) at ( 0.5, -0.5);
  \coordinate (b4) at ( 1,   -0.5);
  \draw[-<-]  (t1) -- (t4) node[pos =0.5, above] {$1$};
  \draw[->-]  (b1) -- (b4) node[pos =0.5, below] {$m$};
\end{scope}
}}} \\ \oplus
      [N-m-1]& \statespaceN{\NB{\tikz[font=\tiny,scale=1]{\begin{scope}
  \coordinate (t1) at (-1, 0.5);
  \coordinate (t2) at (-0.3, 0.5);
  \coordinate (t3) at ( 0.3, 0.5);
  \coordinate (t4) at ( 1, 0.5);
  \coordinate (m2) at ( -0.4, 0);
  \coordinate (m3) at ( 0.4, 0);
  \coordinate (b1) at (-1,   -0.5);
  \coordinate (b2) at (-0.5, -0.5);
  \coordinate (b3) at ( 0.5, -0.5);
  \coordinate (b4) at ( 1,   -0.5);
  \draw[<-]  (t1) .. controls +(0.3, 0)  and +(0,0) .. (m2) node[pos =0, left] {$1$};
  \draw[-<]  (m3)  .. controls +(0.0, 0)  and +(-0.3,0) ..  (t4) node[pos =1, right] {$1$};
  \draw[>-]  (b1)  .. controls +(0.3, 0)  and +(0,0) .. (m2) node[pos =0, left] {$m$};
  \draw[->]  (m3)  .. controls +(0, 0)  and +(-0.3,0) ..  (b4) node[pos =1, right] {$m$};
  \draw[->-] (m2) -- (m3) node[pos = 0.5, above] {$m-1$};
\end{scope}
}}};
      \label{eq:cat-bad-square}
    \end{split}
    \\
		\statespaceN{\NB{\tikz[font=\tiny,xscale=0.8]{\begin{scope}
  \coordinate (t1) at (-1, 0.5);
  \coordinate (t2) at (-0.3, 0.5);
  \coordinate (t3) at ( 0.3, 0.5);
  \coordinate (t4) at ( 1, 0.5);
  \coordinate (b1) at (-1,   -0.5);
  \coordinate (b2) at (-0.5, -0.5);
  \coordinate (b3) at ( 0.5, -0.5);
  \coordinate (b4) at ( 1,   -0.5);
  \draw[>-]  (t1) -- (t2) node[pos =0, left] {$a$};
  \draw[->-] (t2) -- (t3) node[pos =0.5, above] {$a+d$};
  \draw[->]  (t3) -- (t4) node[pos =1, right] {$b$};
  \draw[>-]  (b1) -- (b2) node[pos =0, left] {$b+c$};
  \draw[->-] (b2) -- (b3) node[pos =0.5, below] {$b+c-d$};
  \draw[->]  (b3) -- (b4) node[pos =1, right] {$a+c$};
  \draw[->-] (b2) -- (t2) node[pos = 0.5, left] {$d$};
  \draw[->-] (t3) -- (b3) node[pos = 0.5, right] {$a+d-b$};
\end{scope}}}}
			\cong
		\bigoplus_{j = \max(0, b-a)}^b
			&\qbinom{c}{d-j}
                   \statespaceN{\NB{\tikz[font=\tiny,xscale=0.8]{\begin{scope}
  \coordinate (t1) at (-1, 0.5);
  \coordinate (t2) at (-0.5, 0.5);
  \coordinate (t3) at ( 0.5, 0.5);
  \coordinate (t4) at ( 1, 0.5);
  \coordinate (b1) at (-1,   -0.5);
  \coordinate (b2) at (-0.3, -0.5);
  \coordinate (b3) at ( 0.3, -0.5);
  \coordinate (b4) at ( 1,   -0.5);
  \draw[>-]  (t1) -- (t2) node[pos =0, left] {$a$};
  \draw[->-] (t2) -- (t3) node[pos =0.5, above] {$b-j$};
  \draw[->]  (t3) -- (t4) node[pos =1, right] {$b$};
  \draw[>-]  (b1) -- (b2) node[pos =0, left] {$b+c$};
  \draw[->-] (b2) -- (b3) node[pos =0.5, below] {$a+c+j$};
  \draw[->]  (b3) -- (b4) node[pos =1, right] {$a+c$};
  \draw[-<-] (b2) -- (t2) node[pos = 0.5, left] {$a+j-b$};
  \draw[-<-] (t3) -- (b3) node[pos = 0.5, right] {$j$};
\end{scope}}}} \label{eq:cat-square}.
	\end{align}
        These isomorphisms (except the first one) are realized
        as images of (linear combinations of) foams under $\tqftfunc$.
\end{prop}

  Wu \cite[Theorem 2.4]{Wu} proves that the relations given in
  Proposition~\ref{prop:cat-MOY-calc} are enough to reduce any web
  to the empty web $\emptyset$. One has $\statespaceN{\emptyset}
  \cong \RN$ (see \cite[Claim 3.32]{RW1}) which is a finitely generated projective $\RN$-module.
  Since being
  projective and finitely-generated is preserved under
  finite direct sums and finite direct summands, the $\gll_N$-state space
  of any web is a finitely generated projective $\RN$-module.

\begin{cor}[{\cite[Corollary 3.31]{RW1}}]
  The functor $\tqftfunc$ takes value in
  $\RN\mathsf{\textsf{-}proj}_{\mathrm{gr}}$, the category of finitely
  generated, graded, projective (and therefore free) $\RN$-modules.
\end{cor}

\begin{prop} \label{prop:cat-MOY-spherical}
  Relations \eqref{eq:cat-empty}, \eqref{eq:cat-circle},
  \eqref{eq:cat-assoc}, \eqref{eq:cat-digon},
  \eqref{eq:cat-bad-digon} and \eqref{eq:cat-square} are
    satisfied by $\tqftfunc[s]$.  
\end{prop}

\begin{proof}[Sketch of proof]
  The foams used in \cite{RW1} to define these isomorphisms are
  \spherical{}. Therefore the proof applies {\sl mutatis
    mutandis} to $\tqftfunc[s]$. 
\end{proof}

One of the foams used for the categorification of relation \eqref{eq:cat-bad-square} is not \spherical{}. This is why it is excluded from the statement. 

Vinyl graphs are a special kind of webs for which $\tqftfunc$ and
$\tqftfunc[s]$ coincide. We will not use this coincidence later but
we think this might be of independent interest.

Recall that vinyl graphs are directed webs in
$\annulus$. Note that any vinyl graph $\web$ has a well-defined index
$\indexweb$: that is the sum of thicknesses of edges intersected by a
generic ray, see Figure~\ref{fig:exa-vinyl-graphs}. Collections of
concentric positively oriented circles of various thicknesses provide
examples of vinyl graphs. They are denoted $\SS_{\underline{k}}$ where
$\underline{k}$ is the list of thicknesses, read from the center of the annulus going outwards, (see Figure~\ref{fig:exa-vinyl-graphs}).

\begin{figure}[ht]
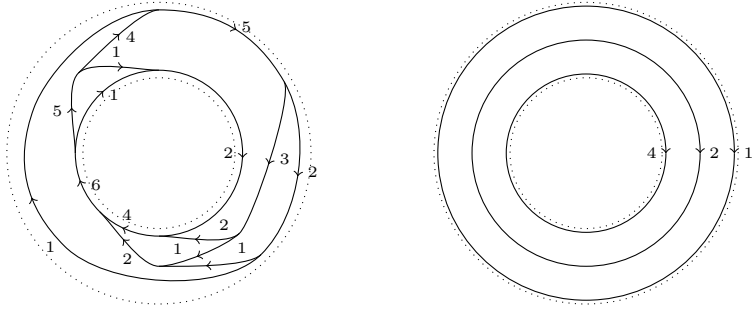

  \centering
\NB{\tikz[]{\begin{scope}[font = \tiny, scale= 0.5]
  \coordinate (O) at (0, 0);
  \draw[dotted] (O) circle (2cm);
  \draw[dotted] (O) circle (4cm);
  \coordinate (F1) at (30:3.8);
  \coordinate (F2) at (90:3.8);
  \coordinate (F3) at (-45:3.8);
  \coordinate (M1) at (135:3);
  \coordinate (M2) at (-90:3);
  \coordinate (M3) at (-45:3);
  \coordinate (I1) at (90:2.2);
  \coordinate (I2) at (180:2.2);
  \coordinate (I3) at (225:2.2);
  \coordinate (I4) at (-90:2.2);
  \draw[-<-] (I1) arc (90:180:2.2)  node[midway, right] {$1$};
  \draw[-<-] (I2) arc (180:225:2.2)  node[midway, right] {$6$};
  \draw[-<-] (I3) arc (-135:-90:2.2)  node[midway, above] {$4$};;
  \draw[-<-] (I4) arc (-90:90:2.2) node[midway, left] {$2$};
  \draw[-<-] (M3) .. controls +(45:0.5) and +(-60:0.5) .. (F1) node[midway, right] {$3$};
  \draw[-<-] (F2) .. controls +(180:0.5) and +(45:0.5) .. (M1)  node[midway, right] {$4$};
  \draw[-<-] (I1) .. controls +(180:0.5) and +(45:0.5) .. (M1) node[midway, above] {$1$};
  \draw[-<-] (M1) .. controls +(225:0.5) and +(90:0.5) .. (I2)  node[midway, left] {$5$};
  \draw[-<-] (I3) .. controls +(-45:0.5) and +(180:0.5) .. (M2)  node[midway, below] {$2$};
  \draw[-<-] (I4) .. controls +(0:0.5) and +(225:0.5) .. (M3)  node[near end, above] {$2$};
  \draw[-<-] (M2) .. controls +(0:0.5) and +(225:0.5) .. (M3) node[near start, above] {$1$};
  \draw[-<-] (M2) .. controls +(0:0.5) and +(225:0.5) .. (F3)  node[near end, above] {$1$};
  \draw[-<-] (F1) .. controls +(120:1.5) and +(0:1.5) .. (F2)  node[midway, right] {$5$};
  \draw[-<-] (F2) .. controls +(180:1.5) and  +(70:1) .. (160:3.5) .. controls +(250:1.5) and +(135:1.5) .. (225:3.5) node[left] {$1$} .. controls +(-45:1.5) and +(225:1.5) .. (F3);
  \draw[-<-] (F3) .. controls +(45:1.5) and +(-60:1.5) .. (F1)  node[midway, right] {$2$};
\end{scope}}}\qquad\qquad \NB{\tikz[]{\begin{scope}[font = \tiny, scale= 0.5]
  \draw[dotted] (0,0) circle (2cm);
  \draw[dotted] (0,0) circle (4cm);
  \draw[->] (0:2.1cm) arc (0:-360:2.1cm) node[left] {$4$};
  \draw[->] (0:3cm) arc (0:-360:3cm) node[right] {$2$};
  \draw[->] (0:3.9cm) arc (0:-360:3.9cm) node[right] {$1$};
\end{scope}}}
  \caption{Two vinyl graphs of index $7$, the one on the
    right-hand side is $\SS_{(4,2,1)}$.}
  \label{fig:exa-vinyl-graphs}
\end{figure}

\begin{dfn}\label{dfn-vinyl-foam}
  Let $\web_0$ and $\web_1$ be two vinyl graph
  A \emph{vinyl foam} $\foam\co \web_0 \to \web_1$ is a foam in
  $\annulus \times[0,1]$ such that the projection  onto $\SS^1\times
  [0,1]$ has no critical points. As usual, these foams are regarded up
  to ambient isotopy. Note that in such a situation, $\web_0$ and
  $\web_1$ have necessarily the same index $\indexweb$. We then say that
  $\foam$ has \emph{index} $\indexweb$.
\end{dfn}

A vinyl foam can be decomposed into basic foams without any cups, caps or
saddles. Note that for a foam, being vinyl is even more restrictive
than being \spherical{}.

Vinyl graphs and vinyl foams of a given index $\indexweb$ fit into a
category that we denote by $\vfoamcat[k]$. Taking the disjoint unions of
these categories for $k \in \NN$, we can form a category $\vfoamcat[]$
which is endowed with a monoidal structure given by taking concentric
disjoint union. Note however that this category is not symmetric.

\begin{prop}
  When restricted to the category $\vfoamcat$, the functors
  $\tqftfunc$ and $\tqftfunc[s]$ are isomorphic.
\end{prop}

\begin{proof}[Sketch of proof.]
  This is a direct consequence of the Queffelec--Rose--Sartori algorithm
  \cite{QRS} which rephrases Wu's results in the vinyl setting. It
  says that one can reduce any vinyl graph to a collection of circles using
  relations~\eqref{eq:cat-assoc}, \eqref{eq:cat-digon}
  and \eqref{eq:cat-square}.  Once dealing with circles one can use
  relation \eqref{eq:cat-circle}. All these relations are valid for
  $\tqftfunc$ and $\tqftfunc[s]$ which concludes the
  argument. For a more detailed proof of a similar argument, we refer
  to \cite[Propositions 2.25 and 4.11]{BPRW}.
\end{proof}

The restriction to vinyl graphs and vinyl foams is common and
relates to the presentations of links as braid closures. For instance,
triply-graded homology~\cite{Roulink, KR3}, symmetric
$\gll_N$-homologies \cite{Cautisremarks, QRS, RW2} and $\gll_0$-homology \cite{RW3} are defined using this framework. On the one hand, this is
topologically quite restrictive since in this context, link cobordisms
are not very interesting. On the other hand the relation with Soergel
bimodules and the power of representation theory are very insightful
there. Soergel bimodules can be understood via foams (see
\cite[Proposition 3.4]{WedrichExponential},
\cite[Proposition 4.15]{RW2}, \cite[Proposition 2.29]{BPRW}), so that
part of what we say here can be applied to that context. See
Section~\ref{sec:comp-lit} for an actual comparison.

\subsection{Base change}
\label{sec:base-change}
The discussion in this subsection is classical and holds for both $\tqftfunc$ and $\tqftfunc[s]$. For
brevity, we only discuss $\tqftfunc$.

The construction of Section~\ref{sec:state-spaces} can be deformed
using a unital ring $S$ and a unital ring morphism
$\phi\co \RN \to S$. One can construct a functor $\tqftfunc[\phi]$
from the category $\foamcat$ to $S$-$\mathsf{mod}$ by considering
$\protostatespaceN[\phi]{\web}$ to be the free $S$-module generated by
foams from the empty foam to $\web$ and modding out this space by the
$S$ bilinear induced by $\phi \circ \bracketN{\cdot;\cdot}$. If $S$ is
non-negatively graded and $\phi$ respects gradings, then
$\tqftfunc[\phi]$ takes values in $S$-$\mathsf{mod}_{\mathrm{gr}}$. We
say that the functor $\tqftfunc[\phi]$ is obtained from
$\tqftfunc[\phi]$ by \emph{base change}.

The isomorphisms stated in Proposition~\ref{prop:cat-MOY-calc} still
hold since their proofs are based on identities of evaluations of
closed foams and these identities are preserved under ring
morphisms $\phi$.

Two special base changes will be of interest to us. Suppose $\scalars$
is a unital ring and consider it as graded and concentrated in degree
$0$. Consider $\varphi\co \ZZ \to \scalars$ the unique unital ring
morphism. This induces a ring morphism $\varphi_{\mathrm{e}}$ from
$\RN$ to
$\KN =\scalars[X_1, \dots, X_\myN]^{S_\myN}=\scalars[E_1, \dots
E_\myN]$ by mapping $E_i$ to $E_i$ for all $1\leq i\leq \myN$.  The
functor induced by this base change is denoted $\etqftfunc[\scalars]$
where the letter `$\mathrm{E}$' stands for {\bf
  e}quivariant. 

The same morphism $\varphi$ induces a morphism $\varphi_0$ from $\RN$ to
$\scalars$ by mapping $E_i$ to $0$ for all $1\leq i\leq \myN$. As
before, this morphism preserves respects the grading.
This base change is denoted by $\ztqftfunc[\scalars]$.

\begin{rmk}
We would like to make a couple of comments about notation.
\begin{enumerate}
    \item The superscripts of the functors and modules discussed in this notation contains information about ring morphisms $\phi$ and also whether or not spherical versions of foams are being used.  Sometimes only one these superscripts appears and we hope that the reader will understand what it is referring to from context.
    \item On several occasions, the superscript referring to the ring morphism is replaced by the target of the morphism.
\end{enumerate}
\end{rmk}

\section{Algebraic preliminaries}
\label{sec:alg-prelim}
\subsection{Symmetries of foam evaluations}
We start by investigating automorphisms of foam evaluations.  We intentionally avoid the technical details to keep this part handwavy and concise for motivations. For now,
let us temporarily forget the gradings involved. By an
\emph{automorphism} of $\gll_\myN$-foams, we mean an invertible
operation that respects the 2-categorical structure of
$\foamcat$, that is, such an operation respects forming disjoint
unions (monoidal structure) and gluing of foams along
boundaries. 
For this discussion,  foams are only considered up to ambient isotopies that do not create extra singular curves or points.
Foams will be considered in their isotopy classes in this sense, and an \emph{automorphism} of foams will be considered up to this isotopy.

As an example, the involution of turning an arbitrary foam upside down
is an automorphism of $\gll_\myN$-foams. Similarly, fix $a\in \Bbbk$,
and changing any polynomial $R(X_1,\dots, X_k)$ occurring as foam
decoration into $R(X_1-a,\dots, X_k-a)$ is an (ungraded) automorphism.

Let $G$ be a subgroup of automorphisms of $\foamcat$. Suppose that $G$
also acts as $\ZZ$-linear automorphisms on the ground ring $R_N$. We
say that $G$ is \emph{compatible} with foam evaluations if, for any
$g\in G$ and closed foam $\Gamma$, we have
\begin{subequations}
\begin{equation} \label{eqn:groupversion}
    \bracketN{g \cdot \foam}= g\cdot \bracketN{\foam}.
\end{equation}
Clearly, such compatible automorphisms preserve the kernel of \eqref{eq:2}, and descend to automorphisms of state spaces of webs.

In particular, if $\foam_1$, $\foam_2$ are disjoint closed foams, preserving
the 2-categorical structure/monoidal leads to 
\begin{equation} \label{eqn:groupversion1} \bracketN{g \cdot
    (\foam_1\sqcup \foam_2)}= (g\cdot \bracketN{\foam_1}) (g\cdot
  \bracketN{ \foam_2}) .
\end{equation}
Furthermore, if $\foam$ and $G$ are foams that share a common
boundary web $\Gamma$, we have
\begin{equation}\label{eqn:groupversion2}
  \bracketN{(g\cdot \foam);(g\cdot G)} = g\cdot \bracketN{\foam;G}.
\end{equation}
\end{subequations}

Infinitesimally, if $g=e^{t\ell}$ is a one-parameter family of
compatible automorphisms, where $t$ is a formal parameter, then taking
derivatives of equations \eqref{eqn:groupversion},
\eqref{eqn:groupversion1}, \eqref{eqn:groupversion2} and evaluating at
$t=0$ results in the condition of a Lie algebra element acting on foam
evaluations. For instance, fix $a\in \Bbbk$, and consider the action
by changing any polynomial decoration $R(X_1,\dots, X_k)$ occurring on
foams into $R(X_1-a,\dots, X_k-a)$. Linearizing this action gives us
the differential action by the vector field
$\sum \partial/\partial X_i$.

If $\ell$ is an \emph{infinitesimal symmetry} of foam evaluations, then it must satisfy
\begin{subequations}
 \begin{equation}
    \bracketN{\ell\cdot \foam} = \ell\cdot \bracketN{\foam}.
\end{equation}
Furthermore, with respect to the monoidal structures, there are
Leibniz rules that are linearized from \eqref{eqn:groupversion1} and
\eqref{eqn:groupversion2}: 
\begin{equation} \label{eqn:LAversion1}
        \bracketN{\ell \cdot (\foam_1\sqcup \foam_2)}= (\ell \cdot \bracketN{\foam_1})    \bracketN{\foam_2} + 
         \bracketN{\foam_1} (\ell \cdot     \bracketN{\foam_2}).
\end{equation}
\begin{equation}\label{eqn:LAversion2}
\ell \cdot \bracketN{\foam;G}  = \bracketN{\ell \cdot \foam; G} + \bracketN{ \foam; \ell\cdot G}  .
\end{equation}
\end{subequations}
Here $\ell$ is treated as a $\Bbbk$-linear endormorphism on the $\Bbbk$-span of foams up to ambient isotopy. We will axiomatize the Lie algebra action in what follows.
 
\subsection{Smash products}
\label{sec:smash}
\newcommand{\lie}{\ensuremath{L}}
\newcommand{\alg}{\ensuremath{A}}
\newcommand{\Ue}{\ensuremath{\mathrm{U}}}
\newcommand{\Der}{\ensuremath{\mathrm{Der}}}
\newcommand{\wittp}[1][p]{\witt_p}

In what follows, $\lie$ is a Lie algebra and $\alg$ is an associative
algebra both over a common commutative ring $\scalars$. Unadorned
scalar products are over $\scalars$. The category of $\lie$-modules is
naturally endowed with a monoidal structure: for any $\lie$-modules
$M$ and  $N$, any $m \in M$, $n \in N$ and $\ell \in \lie$, define:
\begin{equation}
  \label{eq:11}
  \ell \cdot (m \otimes n) := (\ell \cdot m)\otimes n + m\otimes (\ell
  \cdot n),
\end{equation}
and declare that $\lie$ acts by $0$ on $\scalars$ (which is the
monoidal unit).
This is equivalent to the statement that the universal enveloping algebra $\Ue(\lie)$ is endowed with a Hopf
algebra structure by defining for all $\ell$ in $\lie$:
\begin{equation}
  \label{eq:Delta-Lie}
  \Delta(\ell) = 1 \otimes \ell + \ell \otimes 1 \qquad \epsilon(\ell)=0 \qquad S(\ell)=-\ell .
\end{equation}

The algebra $\alg = (\alg, \mu, \eta)$ is an \emph{$\lie$-module
  algebra} if $\alg$ is an $\lie$-module
and the maps $\mu \co \alg \otimes \alg \to \alg$ and $\eta\co
\scalars \to \alg$ are morphisms of $\lie$-modules. In other words,
$\alg$ is an algebra object in $\lie\textrm{-}\mathsf{mod}$. 

When $\alg$ is an $\lie$-module algebra, any element $\ell \in \lie$
defines a $\scalars$-linear operator $\widehat{\ell}$ on $A$, which because of
\eqref{eq:Delta-Lie} satisfies the Leibniz rule, that is, for any
$a_1, a_2 \in \alg$, one has:
\begin{equation}
  \label{eq:13}
  \widehat{\ell}(a_1 a_2) = \widehat{\ell}(a_1)a_2 + a_1\widehat{\ell}(a_2).
\end{equation}

\begin{exa}
  For any associative algebra $\alg$, the Lie algebra of derivations $\Der(\alg)$
  acts on $\alg$ naturally and $\alg$ is a $\Der(\alg)$-module
  algebra. Actually any $\lie$-module algebra structure on $\alg$
  arises from a Lie algebra morphism $\varphi\co \lie  \to \Der(\alg)$.
\end{exa}

Fix $\alg$, an $\lie$-module algebra. Define the associative algebra
$\alg\#\Ue(\lie)$ as follows. As a $\scalars$-module,
$\alg\#\Ue(\lie)$ is equal to $\alg \otimes\Ue(\lie)$. The multiplication
on $\alg\#\Ue(\lie)$ is defined by:
\begin{equation}
  (a \otimes \ell) \cdot   (b \otimes h) := \sum
  a(\ell_{(1)}\cdot b) \otimes \ell_{(2)}h,  \label{eq:mult-smash}
\end{equation}
where we used Sweedler's notation for the coproduct on $\Ue(\lie)$.
The algebra $\alg\#\Ue(\lie)$ is called the \emph{smash
  product algebra of $\alg$ and $\lie$}.
Note that both $\alg$ and $\Ue(\lie)$ lie in $\alg\#\Ue(\lie)$ as
subalgebras (as $\alg\otimes 1_{\Ue(\lie)}$ and $1_\alg \otimes
\Ue(\lie)$ respectively).

An $\alg\#\Ue(\lie)$-module is called an \emph{$\lie$-equivariant
  $\alg$-module}. An $\lie$-equivariant $\alg$-module $M$ is an
$A$-module with an $\lie$-action compatible with the action of $\alg$
(remember that $\lie$ acts on $\alg$ by derivations) in the following
sense. For any $a\in \alg, \ell \in \lie$ and $m\in M$:
\begin{equation}
  \ell\cdot (a\cdot m) = (\ell\cdot a)\cdot m + a\cdot (\ell \cdot
  m).\label{eq:15}
\end{equation}

If $H$ is a Hopf algebra over $\scalars$ acting on $\alg$, we adopt the same
terminology and let $\alg\# H$ be the associative algebra which as
a $\scalars$-module is equal to $A\otimes H$ and whose multiplication is
given by formula \eqref{eq:mult-smash}, (replacing $\ell_1, \ell_2$
by $h_1, h_2 \in H$). For a more systematic account on smash products,
we refer the reader to \cite[Chapter VII]{SweedlerBookHopfAlgebras}.

\subsection{$p$-DG structure}
\label{sec:p-dg-structure}
Here we recall some definitions from \cite[Section 2.1]{QRSW1}.
Let $H^\prime = \ZZ[\dif]$  be the graded polynomial algebra generated
by a degree $2$ generator $\dif$.  Define on $H^\prime$ a \emph{comultiplication} $\Delta: H^\prime \rightarrow H^\prime \otimes H^\prime$ by
setting
\begin{subequations}\label{eqn-Hopf-algebra-H-prime}
  \begin{equation}
    \Delta (\dif) = \dif \otimes 1 + 1\otimes \dif.
  \end{equation}
  Also set the {\sl counit} $\epsilon: H^\prime \rightarrow \scalars $ to be
  \begin{equation}
    \epsilon(\dif)=0,
  \end{equation}
  and {\sl antipode} $S:H^\prime \rightarrow H^\prime$ to be
  \begin{equation}
    S(\dif)=-\dif.
  \end{equation}
\end{subequations}
Then $H^\prime$ is a graded Hopf algebra.

The ideal $(\dif^p, p)\subset H^\prime$ is a Hopf ideal, in the sense
that, on the top of being an ideal, it is closed under $\Delta$,
$\epsilon$ and $S$. The graded quotient $H^\prime/(\dif^p,p)$ inherits
a graded Hopf algebra structure over $\Fp$ and is denoted by
$H$. The structure maps of $H$ are
still denoted $\Delta$, $\epsilon$ and $S$. An $H$-module algebra is also called \emph{$p$-DG algebra}.

The element $\dif \in H^\prime$ acts on
algebras as derivations, in other words, any derivation on an algebra $A$
induces an $H^\prime$-module structure on $A$. More generally, if $M$ is an
$A$-module with $A$ endowed with a derivation, any derivation on $M$ (compatible with that on $A$)
gives rise to an $H^\prime$-equivariant $A$-module structure on $M$. 

\begin{lem}\label{lem:action-H-polynomial}
  Mapping $\dif$ to
  $
    \sum_{k=1}^a x_k^2 \frac{\partial}{\partial x_k}
  $
  induces an action of the Hopf algebra $H$ on
  $\Fp[x_1, \dots x_a]$.
\end{lem}

\begin{proof}
  The only slightly non-trivial thing to show is that
  $\dif^p=0$. Because we are in characteristic $p$, $\dif^p0$ is a derivation. Hence it is enough to show that $\dif^px_k =0$ for all
  $1\leq k \leq a$. This follows from the formula:
  \[
    \dif^jx_k = j!x_k^{j+1}
  \]
  which is easily proved by induction, and the fact that we are working over a field of characteristic $p$.
\end{proof}

For this section, we let $A$ be the $p$-DG algebra $\Fp[x_1, \dots,
x_a]$. If $a_1 + a_2+\dots + a_\ell =a$ is a decomposition of $a$ into positive
integers, $G:= S_{a_1}\times S_{a_2} \times \cdots \times S_{a_\ell}$ acts
naturally on $A$ by permuting variables. Since the definition of $\dif$
is symmetric in the variables, one has the following corollary.

\begin{cor}
  The $H$-action on $A$ induces an $H$-action on $A^G$.
\end{cor}

Let $t$ be a homogeneous polynomial in $x_1, \dots, x_a$ of
degree\footnote{Remember that variables have degree $2$, so that $t$ is a linear polynomial.} $2$. Then define $A^t$ to be equal to $A$ as an $A$-module, but
endowed with an $H'$-action defined for any $P \in A^t$ by:
\begin{equation} \label{eq:deftwistact}
  \dif_{A^t}(P) = \dif_A(P) + tP. 
\end{equation}
We say that the $H'$-module structure is \emph{twisted} by $t$. Note that $t =\dif_{A^t}(1)$.
\begin{lem} \label{lem:twist-pDG-alg}
  For any homogeneous polynomial $t$ of degree $2$,
  $\dif_{A^t}^p = 0$. 
\end{lem}

\begin{proof}
  On has for all $P \in A$,
  \begin{equation}
    \dif_{A^{t}}^k(P) = \sum_{j=0}^k\begin{pmatrix} k \\j \end{pmatrix}\dif_A^j(P)\dif^{k-j}_{A^t}(1).
  \end{equation}
  We already know that $\dif^p_A(P) =0$, so it is enough to
  show that $\dif^p_{A^t}(1) =0$ since $\begin{pmatrix} p \\ j \end{pmatrix} =0$ for $j=1,\ldots, p-1$.
  Moreover, if $t_1$ and $t_2$ are two homogeneous polynomials of degree
  $2$, one has:
  \begin{equation}
    \label{eq:3}
     \dif_{A^{t_1 + t_2}}^k(1) = \sum_{j=0}^k\begin{pmatrix} k \\j \end{pmatrix}\dif_{A^{t_1}}^j(1)\dif^{k-j}_{A^{t_2}}(1),
  \end{equation}
  so that it is enough to prove the statement in the case $t=\mu_i
  x_i$ for $1\leq i \leq a$.
  In this case, an easy induction shows:
  \begin{equation}
    \dif_{A^{t}}^k(1) = \left(\prod_{j=0}^{k-1}(\mu_i+j)\right) x_i^k.
  \end{equation}
  The proof is now complete since the quantity $\prod_{j=0}^{p-1}(\mu+j)$ is $0$ in $\Fp$ for any $\mu
  \in \Fp$.
\end{proof}

\begin{rmk} \label{rmk:twist-works-on-invariants}
  If $t$ is invariant under the action of $G:= S_{a_1}\times S_{a_2} \times \cdots \times
S_{a_\ell}$, the  action defined by \eqref{eq:deftwistact} preserves $A^G$ and one has a
well-defined $H$-action on $(A^t)^G$.
\end{rmk}

\subsection{One-half of the Witt algebra}
\label{sec:Witt-alg-char-zero}

In what follows we will define actions of the Lie algebras $\sll_2$
and of a part of the Witt algebra $\ourwitt$. 
We briefly recall how these Lie algebras are
defined.

The Lie algebra $\sll_2$ over $\scalars$ is generated by symbols $\Le,
\Lf$ and $\Lh$ subject to the relations:
\begin{align}  \label{eq:sl2-relations}  
  [\Le, \Lf] = \Lh,
  \quad [\Lh, \Le] = 2 \Le,
  \quad [\Lh, \Lf ] = -2 \Lf.
\end{align}

As a $\scalars$-module, it is free of rank $3$ and can be graded by
declaring that $\deg \Le = -2, \deg \Lh =0$ and $\deg \Lf = 2$ (or by a
scaling of that grading).

The Lie algebra $\witt$ is generated by symbols $(\LLn)_{n \in \ZZ}$
subject to the relations\footnote{In \cite{KRWitt}, the presentation
  is different: $\LLn^{\text{\cite{KRWitt}}} \leftrightarrow -\LLn$.}
\begin{align}
  \label{eq:witt-relations}
 [\LLn, \LLn[m]] = (n-m) \LLn[m+n]
\end{align}
for all $n, m\in \ZZ$.

As a $\scalars$-module, it is free of countable rank and can be graded by
declaring that $\deg \LLn = 2n$ for all $n \in \ZZ$ (or by scaling of that). 

We will be interested in the Lie subalgebra $\ourwitt$ generated by
symbols $(\LLn)_{n\in \NNN}$, where $\NNN = \{n \in \ZZ | n \geq
-1\}$ since, as we shall see, it acts on polynomial rings, see \eqref{eq:witt-on-poly}. 

\begin{lem}
  \label{lem:sl2-to-witt}
  The map:
\[ 
   \iota\co \left\{
    \begin{array}{rcl}
     \Le & \mapsto & \LLn[-1] \\
      \Lh & \mapsto & 2\LLn[0] \\
      \Lf & \mapsto & -\LLn[1] \\
    \end{array}
    \right.
  \]
  induces a morphism of Lie algebras from $\sll_2$ to $\witt$ whose
  image is in $\ourwitt$. If $2$ is not a zero divisor in
  $\scalars$, the map is injective. 
\end{lem}

\begin{proof}
  It is straightforward to check that the relations given in
  \eqref{eq:sl2-relations} are satisfied by $\iota(\Le)$, $\iota(\Lh)$
  and $\iota(\Lf)$:
  \begin{align*}
    [\iota(\Le), \iota(\Lf)]&=[\LLn[-1], -\LLn[1]] = [\LLn[1], \LLn[-1]]
                              = 2\LLn[0] = \iota(\Lh) \\
    [\iota(\Lh), \iota(\Le)] &= [2\LLn[0], \LLn[-1]]= 2 [\LLn[0],
                               \LLn[-1]] = 2 \LLn[-1] = 2 \iota(\Le), \\
    [\iota(\Lh), \iota(\Lf)] &= [2\LLn[0], -\LLn[1]]= -2 [\LLn[0],
                               \LLn[1]] = 2 \LLn[1] = -2 \iota(\Lf). \qedhere
  \end{align*}
\end{proof}

\begin{rmk}
  Another possible embedding of $\sll_2$ in $\ourwitt$ is given by: 
   \[
    \begin{array}{rcl}
      \iota'\co\Le & \mapsto & -\LLn[1] \\
      \Lh & \mapsto & -2\LLn[0] \\
      \Lf & \mapsto & \LLn[-1]  \\
    \end{array}.
  \]
  The morphism $\iota$ and $\iota'$ are related by the
  $\sll_2$-automorphism
     \[
    \begin{array}{rcl}
      \Le & \mapsto & \Lf \\
      \Lh & \mapsto & -\Lh \\
      \Lf & \mapsto & \Le \\
    \end{array}.
  \]
\end{rmk}

The Lie algebra $\ourwitt$ acts on the polynomial ring $\scalars[z]$
by setting $\LLn \cdot Q(z) = -z^{n+1}Q'(z)$. For any
positive integer $k$, this action generalizes to $\scalars[z_1, \dots,
z_k]$ as follows. For any $Q \in \scalars[z_1, \dots, z_k]$, set
\begin{align}\label{eq:witt-on-poly}
  \LLn \cdot Q = -\sum_{i=1}^k
  z_i^{n+1}\frac{\partial Q}{\partial z_i}.
\end{align}
Note that
$A_k=\scalars[z_1, \dots, z_k]^{S_k}$ is a sub-$\ourwitt$-module for
this action. One can naturally extend this action on $\scalars(z_1,\dots
z_k)$ by imposing the Leibniz rule:
\[
\LLn \cdot \frac{Q_1}{Q_2} = \frac{(\LLn\cdot Q_1)Q_2 -
Q_1(\LLn \cdot Q_2)}{Q_2 ^2}.
\]

\begin{exa}
  If $x$ and $y$ are two variables, then for all $n\geq 0$,
\begin{align}
  \LLn \cdot (x-y) =& -(x^{n+1} - y^{n+1}) = -(x-y)\left(\sum_{i+j=n}
    x^iy^j \right)\\ =&- (x-y)\left(\sum_{i+j=n}
    \newtoni[i](x)\newtoni[j](y) \right) = -(x-y)h_n(x,y)
\end{align}
where $\newtoni[i](x)$ and $\newtoni[j](y)$ denote the $i$th and the
$j$th power sum
symmetric
polynomials in variables $x$ and $y$ respectively (that is to say $x^i$
and $y^j$), and $h_n$ is the $n$th complete homogeneous symmetric polynomial.

  Let $\underline{x}:=\{x_1, \dots, x_a\}$ be $a$ indeterminates and $\underline{y}:= \{y_1, \dots,
  y_b\}$ be $b$ other indeterminates. Set $\nabla(\underline{x},\underline{y})=
  \prod_{i=1}^a\prod_{j=1}^b (x_i - y_j).$ One has 
  \begin{align}
    \begin{split}
    \LLn \cdot \nabla(\underline{x},\underline{y})
    &=-\sum_{i=1}^a\sum_{j=1}^b\sum_{k+\ell=n}\newtoni[k](x_i)\newtoni[\ell](y_j)\nabla(\underline{x},\underline{y}) \\
    &=- \sum_{k+\ell=n}\newtoni[k](\underline{x})\newtoni[\ell](\underline{y})\nabla(\underline{x},\underline{y});
  \end{split}
  \end{align}
and more generally, for any non-negative integer $\alpha$:
    \begin{align}\label{eq:Ln-acts-on-Nabla}
    \LLn \cdot \nabla(\underline{x},\underline{y})^\alpha =- \alpha
    \sum_{k+\ell=n}\newtoni[k](\underline{x})\newtoni[\ell](\underline{y}) \nabla(\underline{x},\underline{y})^\alpha.
  \end{align}
\end{exa}

\subsection{Twists}
\label{sec:twists}

As we have seen in the previous subsection, $\ourwitt$ acts naturally
on polynomial rings. We will see that this action can be twisted. This machinery will be used in a sequel to this paper.

\begin{dfn}[\cite{KRWitt}]
  Let $A$ be an $\ourwitt$-module algebra.
 A family of elements $\tau:=(\tau_n)_{n \in \NNN} \in A^{\NNN}$ is
  \emph{$\ourwitt$-flat} if for all $m$ and $n$ in $\NNN$,
  \begin{equation}
    \label{eq:8}
    \LLn[n] \cdot \tau_m - \LLn[m] \cdot \tau_n = (n-m)\tau_{m+n}.
  \end{equation}
\end{dfn}

\begin{exa}
  \begin{enumerate}
  \item For any $i \in \{1, \dots, k\}$, the family
    $\left((n+1)x_i^n\right)_{n\in \NNN} \in \scalars[x_1, \dots, x_k]^{\NNN}$ is $\ourwitt$-flat.
   \item For any  $i\neq j\in \{1, \dots, k\}$, the family
      $\left(h_n(x_i, x_j)\right)_{n\in \NNN}\in \scalars[x_1, \dots, x_k]^{\NNN}$   is $\ourwitt$-flat.
    \end{enumerate}
\end{exa}

\begin{rmk}
  In general the defect of $\ourwitt$-flatness of a sequence $\tau$,
  encoded by \begin{equation}
    \label{eq:9}
        \LLn[n] \cdot \tau_m - \LLn[m] \cdot \tau_n - (n-m)\tau_{m+n}
      \end{equation}
      is called the \emph{$\ourwitt$-curvature} of $\tau$ and is
      denoted $\kappa(\tau)$.
\end{rmk}

The set of $\ourwitt$-flat sequences of an $\ourwitt$-module $A$ is a
$\scalars$-submodule of $A^{\NNN}$ since for arbitrary sequences $\tau$ and
$\tau'$ in $A^{\NNN}$, one has $\kappa(\tau + \tau') = \kappa(\tau) + \kappa(\tau')$.

Let $A$ be a commutative $\ourwitt$-module algebra and $\tau$ be an
$\ourwitt$-flat sequence. For any $n \in \NNN$, define the operator
$\Ln^{(\tau)}$ on $A$ by:
\begin{equation}
  \label{eq:10}
  \Ln^{(\tau)}(a) = \LLn\cdot a + \tau_na.
\end{equation}
for all $a\in A$.

\begin{lem}\label{lem:alg-twist}
  Mapping $\LLn$ to $\Ln^{(\tau)}$ endows $A$ with a (new)
  $\ourwitt$-module structure. 
\end{lem}

We write $A^{\tau}$ to encode this new $\ourwitt$-module structure, 
and we say that $A$ is twisted by $\tau$. Note that in general
$A^{\tau}$ is not anymore a $\ourwitt$-module algebra, but rather an
$\ourwitt$-equivariant $A$-module (isomorphic to $A$ as an $A$-module).

\begin{proof}
  This is a straightforward computation. For any $m,n \in \NNN$, one has:

  \begin{align}
    \begin{split}
    \Ln[n]^{(\tau)}(\Ln[m]^{(\tau)}(a)) =& \Ln[n]^{(\tau)}(\LLn[m]\cdot a + \tau_ma)\\
    =&\LLn[n] \cdot (\LLn[m] \cdot a) + \LLn[n]\cdot (\tau_m a)+ \tau_n
       (\LLn[m]\cdot a + \tau_ma) \\
    =& \LLn[n] \cdot (\LLn[m] \cdot a) + (\LLn[n]\cdot\tau_m)a +
       \tau_m\LLn[n]\cdot a +  \tau_n \LLn[m]\cdot a  + \tau_n\tau_m a.
\end{split}
     \end{align}
  Thus
  \begin{align}
    \begin{split}
    [\Ln[n]^{(\tau)},\Ln[m]^{(\tau)}] (a) 
    =& [\LLn[n], \LLn[m]] \cdot a + (\LLn[n](\tau_m) - \LLn[m](\tau_n)) a \\
    =& (n-m)\LLn[n+m](a) + (n-m)\tau_{m+n} a
    \end{split}
    \\
    =& (n-m) \Ln[n+m]^{(\tau)}(a).
    \qedhere
  \end{align}
\end{proof}
Suppose furthermore that $M$ is a $\ourwitt$-equivariant $A$-module,
then the $\ourwitt$-module $A^\tau\otimes_AM$ is denoted $M^{\tau}$.
  If $\tau$ and $\tau'$ are two $\ourwitt$-flat sequences, then
  \begin{equation}
  (M^{\tau'})^{\tau} = A^\tau \otimes_A M^{\tau'} = A^\tau \otimes_A
    (A^{\tau'} \otimes_A M) \cong A^{\tau + \tau'} \otimes M = M^{\tau +
    \tau'}.\label{eq:16}
\end{equation}

\section{Action on foams}
\label{sec:an-mathfr-acti}
\subsection{Action of one-half of the Witt algebra}
\label{sec:action-from-witt}

For simplicity we will
suppose that 2 is invertible in $\scalars$. However, this hypothesis
is not always necessary. See Remark~\ref{rmk:coefficients}.

Consider a set of indeterminates $\underline{x}=\{x_1,\ldots,x_a \}$.  

\begin{dfn}
  A \emph{Witt-sequence} $(\lambda_n)_{n \in \NNN} \in
  \scalars^{\NNN}$ is a sequence such that $\lambda_{-1}=0$ and for any  $m,n \in \NN$,
  \begin{equation}
    \label{eq:5}
n\lambda_{n} - m\lambda_{m}  = (n-m)\lambda_{m+n}.
\end{equation}
\end{dfn}

For any $\lambda \in \scalars$, the sequence given by
$\lambda_n = \lambda (n+1)$ is a Witt sequence.

Recall that for any $i$ in $\NN$, the power sum polynomial $p_i(\underline{x})$ is
defined as:
\begin{equation}
  \label{eq:7}
  \newtoni(\underline{x})= x_1^i+\cdots+x_a^i .  
\end{equation}
Decorations of foams that we will consider will often be power sums, so 
we use the following notation: 
\begin{equation}
  \NB{\tikz[scale=1.5, font=\small]{\begin{scope}
  \draw (0,0) rectangle (1,1) coordinate [midway] (A);
  \node at (A) {$\dotnewtoni$};
\end{scope}}} \ =\
  \NB{\tikz[scale=1.5, font=\small]{\begin{scope}
  \draw (0,0) rectangle (1,1) coordinate [midway] (A);
  \fill (A) circle (0.5mm) node[below] {$p_i$};
\end{scope}}}.
\end{equation}
Note that in particular, on a facet of thickness $a$,
$\dotnewtoni[0]=a$. Following Section~\ref{sec:new-decorations}, let
$\wdotnewtoni$ denote the $i$th power sum in the variables which are
not in the facet. In other words, \begin{equation}
  \label{eq:18}
  \NB{\tikz[scale=1.5, font=\small]{\begin{scope}
  \draw (0,0) rectangle (1,1) coordinate [midway] (A);
  \node at (A) {$\wdotnewtoni$};
\end{scope}}} \ =\
  P_i\cdot\ \NB{\tikz[scale=1.5, font=\small]{\begin{scope}
  \draw (0,0) rectangle (1,1) coordinate [midway] (A);
\end{scope}}}\   -\ 
  \NB{\tikz[scale=1.5, font=\small]{}},
\end{equation}
where, as stated in the conventions, $P_i$ denotes the $i$th power sum polynomial in $X_1,\dots, X_N$.

For the rest of this section, fix an element
$s \in \scalars$, and three Witt-sequences
$(\parone{n})_{n\in\NNN}, (\partwo{n})_{n\in\NNN}$ and
$(\parthree{n})_{n\in\NNN}$.

We now define a sequence of operators $(\Ln)_{n \in \NNN}$ acting on basic foams. For
 $n\in \NNN$ set: \begin{gather}\label{eq:Ln-on-pol}
  \Ln\left(~\NB{\tikz[scale=1.5, font=\small]{}}~\right)
  =\NB{\tikz[scale=1.5, font=\small]{\begin{scope}
  \draw (0,0) rectangle (1,1) coordinate [midway] (A);
  \fill (A) circle (0.5mm) node[below] {$\Ln(R)$};
\end{scope}}} \\
\label{eq:Ln-on-MP}
  \Ln\left(\NB{\tikz[scale=0.6, font=\tiny]{\begin{scope}[xscale = -1 ]
  \begin{scope}
    \coordinate (LL) at (0,0);
    \coordinate (L) at (0.5,0);
    \coordinate (R1) at (2.2,0.4);
    \coordinate (R2) at (2,0);
    \coordinate (R3) at (1.8, -0.4);
    \draw[-<-] (LL) -- (L);
    \draw (L) .. controls +(0,0) and +(-0.5, 0) .. (R1)
    coordinate[pos =0.5] (M);
    \draw (L) .. controls +(0,0) and +(-0.5, 0) .. (R3) ;
    \draw (M) .. controls +(0,0) and +(-0.5, 0) .. (R2) ;
  \end{scope}  
 \begin{scope}[yshift = -1.5cm]
    \coordinate (LLB) at (0,0);
    \coordinate (LB) at (0.5,0);
    \coordinate (R1B) at (2.2,0.4);
    \coordinate (R2B) at (2,0);
    \coordinate (R3B) at (1.8, -0.4);
    \draw[-<-] (LLB) -- (LB)  node[right, pos =0] {$a+b+c$};
    \draw (LB) .. controls +(0,0) and +(-0.5, 0) .. (R1B)    node[left] {$a$};
    \draw (LB) .. controls +(0,0) and +(-0.5, 0) .. (R3B)   node[left] {$b$}
    coordinate[pos =0.5] (MB);
    \draw (MB) .. controls +(0,0) and +(-0.5, 0) .. (R2B)    node[left] {$c$};
  \end{scope}  
  \draw (LL) -- (LLB);
  \draw (R1) -- (R1B);
  \draw (R2) -- (R2B);
  \draw (R3) -- (R3B);
  \draw[name path=path1, thick] (M) .. controls +(0,-0.5) and +(0,
  0.5) .. (LB);
  \draw[name path=path2, thick] (L) .. controls +(0,-0.5) and +(0,
  0.5) .. (MB);
  \path [name intersections={of=path1 and path2,by=O}];
  
\end{scope}}}\right)=
  \Ln\left(\NB{\tikz[scale=0.6,font=\tiny]{\begin{scope}
  \begin{scope}
    \coordinate (LL) at (0,0);
    \coordinate (L) at (0.5,0);
    \coordinate (R1) at (2.2,0.4);
    \coordinate (R2) at (2,0);
    \coordinate (R3) at (1.8, -0.4);
    \draw[->-] (LL) -- (L);
    \draw (L) .. controls +(0,0) and +(-0.5, 0) .. (R1)
    coordinate[pos =0.5] (M);
    \draw (L) .. controls +(0,0) and +(-0.5, 0) .. (R3) ;
    \draw (M) .. controls +(0,0) and +(-0.5, 0) .. (R2) ;
  \end{scope}  
 \begin{scope}[yshift = -1.5cm]
    \coordinate (LLB) at (0,0);
    \coordinate (LB) at (0.5,0);
    \coordinate (R1B) at (2.2,0.4);
    \coordinate (R2B) at (2,0);
    \coordinate (R3B) at (1.8, -0.4);
    \draw[->-] (LLB) -- (LB)  node[left, pos =0] {$a+b+c$};
    \draw (LB) .. controls +(0,0) and +(-0.5, 0) .. (R1B)    node[right] {$a$};
    \draw (LB) .. controls +(0,0) and +(-0.5, 0) .. (R3B)   node[right] {$b$}
    coordinate[pos =0.5] (MB);
    \draw (MB) .. controls +(0,0) and +(-0.5, 0) .. (R2B)    node[right] {$c$};
  \end{scope}  
  \draw (LL) -- (LLB);
  \draw (R1) -- (R1B);
  \draw (R2) -- (R2B);
  \draw (R3) -- (R3B);
  \draw[name path=path1, thick] (M) .. controls +(0,-0.5) and +(0,
  0.5) .. (LB);
  \draw[name path=path2, thick] (L) .. controls +(0,-0.5) and +(0,
  0.5) .. (MB);
  \path [name intersections={of=path1 and path2,by=O}];  
\end{scope}
}} \right) =0 \\
\begin{split}\label{eq:Ln-on-dig-cup}
  \Ln\left( \NB{\tikz[font=\tiny]{\begin{scope}
  \begin{scope}
    \coordinate (L) at (0,0);
    \coordinate (R) at (2,0);
    \coordinate (ML) at (0.5, 0);
    \coordinate (MR) at (1.5, 0);
    \draw[->-] (L) -- (ML);
    \draw[->-] (MR) -- (R) node[right] {$a+b$};
    \draw[->-] (ML).. controls + (0.4, 0.4) and +(-0.2, 0.4) .. (MR)
    node[above, midway] {$a$};
    \draw[->-] (ML).. controls + (0.2, -0.4) and +(-0.4, -0.4) .. (MR)
    node[below, midway] {$b$};
  \end{scope}  
 \begin{scope}[yshift = -1.3cm]
    \coordinate (LB) at (0,0);
    \coordinate (RB) at (2,0);
    \draw[->-] (LB) -- (RB);
  \end{scope}  
  \draw (R) -- (RB);
  \draw (L) -- (LB);
  \draw[thick] (ML) .. controls +(0, -1) and +(0, -1) .. (MR);
\end{scope}
}}\right)\ =\
  & \parone{n}\cdot\ \NB{\tikz[font=\tiny]{\begin{scope}[font=\tiny]
  \begin{scope}
    \coordinate (L) at (0,0);
    \coordinate (R) at (2,0);
    \coordinate (ML) at (0.5, 0);
    \coordinate (MR) at (1.5, 0);
    \draw[->-] (L) -- (ML);
    \draw[->-] (MR) -- (R) node[right] {$a+b$};
    \draw[->-] (ML).. controls + (0.4, 0.4) and +(-0.2, 0.4) .. (MR)
    node[above, pos=0.7 ] {$a$} node[below, pos =0.3] {$\dotnewtoni[n]$};
    \draw[->-] (ML).. controls +(0.2, -0.4) and +(-0.4, -0.4) .. (MR)
    node[below, pos =0.3] {$b$} node[below, pos =0.75] {$\dotnewtoni[0]$};
  \end{scope}  
 \begin{scope}[yshift = -1.3cm]
    \coordinate (LB) at (0,0);
    \coordinate (RB) at (2,0);
    \draw[->-] (LB) -- (RB);
  \end{scope}  
  \draw (R) -- (RB);
  \draw (L) -- (LB);
  \draw[thick] (ML) .. controls +(0, -1) and +(0, -1) .. (MR);
\end{scope}

}} 
  + \partwo{n}\cdot \ \NB{\tikz[font=\tiny]{\begin{scope}[font=\tiny]
  \begin{scope}
    \coordinate (L) at (0,0);
    \coordinate (R) at (2,0);
    \coordinate (ML) at (0.5, 0);
    \coordinate (MR) at (1.5, 0);
    \draw[->-] (L) -- (ML);
    \draw[->-] (MR) -- (R) node[right] {$a+b$};
    \draw[->-] (ML).. controls + (0.4, 0.4) and +(-0.2, 0.4) .. (MR)
    node[above, pos=0.7 ] {$a$} node[below, pos =0.3] {$\dotnewtoni[0]$};
    \draw[->-] (ML).. controls +(0.2, -0.4) and +(-0.4, -0.4) .. (MR)
    node[below, pos =0.3] {$b$} node[below, pos =0.75] {$\dotnewtoni[n]$};
  \end{scope}  
 \begin{scope}[yshift = -1.3cm]
    \coordinate (LB) at (0,0);
    \coordinate (RB) at (2,0);
    \draw[->-] (LB) -- (RB);
  \end{scope}  
  \draw (R) -- (RB);
  \draw (L) -- (LB);
  \draw[thick] (ML) .. controls +(0, -1) and +(0, -1) .. (MR);
\end{scope}

}}  \\
  &+{s} \sum_{k+\ell=n} \NB{\tikz[font=\tiny]{\begin{scope}[font=\tiny]
  \begin{scope}
    \coordinate (L) at (0,0);
    \coordinate (R) at (2,0);
    \coordinate (ML) at (0.5, 0);
    \coordinate (MR) at (1.5, 0);
    \draw[->-] (L) -- (ML);
    \draw[->-] (MR) -- (R) node[right] {$a+b$};
    \draw[->-] (ML).. controls + (0.4, 0.4) and +(-0.2, 0.4) .. (MR)
    node[above, pos=0.7 ] {$a$} node[below, pos =0.3] {$\dotnewtoni[k]$};
    \draw[->-] (ML).. controls +(0.2, -0.4) and +(-0.4, -0.4) .. (MR)
    node[below, pos =0.3] {$b$} node[below, pos =0.75] {$\dotnewtoni[\ell]$};
  \end{scope}  
 \begin{scope}[yshift = -1.3cm]
    \coordinate (LB) at (0,0);
    \coordinate (RB) at (2,0);
    \draw[->-] (LB) -- (RB);
  \end{scope}  
  \draw (R) -- (RB);
  \draw (L) -- (LB);
  \draw[thick] (ML) .. controls +(0, -1) and +(0, -1) .. (MR);
\end{scope}

}}
  \end{split}
  \\
\begin{split}\label{eq:Ln-on-digcap}
  \Ln\left( \NB{\tikz[font=\tiny]{\begin{scope}
  \begin{scope}
    \coordinate (L) at (0,0);
    \coordinate (R) at (2,0);
    \coordinate (ML) at (0.5, 0);
    \coordinate (MR) at (1.5, 0);
    \draw[->-] (L) -- (ML);
    \draw[->-] (MR) -- (R) node[right] {$a+b$};
    \draw[->-] (ML).. controls + (0.4, 0.4) and +(-0.2, 0.4) .. (MR)
    node[above, midway] {$a$};
    \draw[->-] (ML).. controls + (0.2, -0.4) and +(-0.4, -0.4) .. (MR)
    node[below, midway] {$b$};
  \end{scope}  
 \begin{scope}[yshift = 1.3cm]
    \coordinate (LB) at (0,0);
    \coordinate (RB) at (2,0);
    \draw (LB) -- (RB);
  \end{scope}  
  \draw (R) -- (RB);
  \draw (L) -- (LB);
  \draw[thick] (ML) .. controls +(0, 1) and +(0, 1) .. (MR);
\end{scope}
}}\right)\ =\
  &-  \parone{n}\cdot\ \NB{\tikz[font=\tiny]{\begin{scope}[font =\tiny]
  \begin{scope}
    \coordinate (L) at (0,0);
    \coordinate (R) at (2,0);
    \coordinate (ML) at (0.5, 0);
    \coordinate (MR) at (1.5, 0);
    \draw[->-] (L) -- (ML);
    \draw[->-] (MR) -- (R) node[right] {$a+b$};
    \draw[->-] (ML).. controls + (0.4, 0.4) and +(-0.2, 0.4) .. (MR)
    node[above, pos =0.7] {$a$}     node[above, pos =0.3] {$\dotnewtoni[n]$};
    \draw[->-] (ML).. controls + (0.2, -0.4) and +(-0.4, -0.4) .. (MR)
    node[left, pos =0.3] {$b$} node[above, pos =0.7, yshift = -0.5mm] {$\dotnewtoni[0]$};
  \end{scope}  
 \begin{scope}[yshift = 1.3cm]
    \coordinate (LB) at (0,0);
    \coordinate (RB) at (2,0);
    \draw (LB) -- (RB);
  \end{scope}  
  \draw (R) -- (RB);
  \draw (L) -- (LB);
  \draw[thick] (ML) .. controls +(0, 1) and +(0, 1) .. (MR);
\end{scope}

}}  
  \!\!\!\!- \partwo{n}\cdot \ \NB{\tikz[font=\tiny]{\begin{scope}[font =\tiny]
  \begin{scope}
    \coordinate (L) at (0,0);
    \coordinate (R) at (2,0);
    \coordinate (ML) at (0.5, 0);
    \coordinate (MR) at (1.5, 0);
    \draw[->-] (L) -- (ML);
    \draw[->-] (MR) -- (R) node[right] {$a+b$};
    \draw[->-] (ML).. controls + (0.4, 0.4) and +(-0.2, 0.4) .. (MR)
    node[above, pos =0.7] {$a$}     node[above, pos =0.3] {$\dotnewtoni[0]$};
    \draw[->-] (ML).. controls + (0.2, -0.4) and +(-0.4, -0.4) .. (MR)
    node[left, pos =0.3] {$b$} node[above, pos =0.7, yshift = -0.5mm] {$\dotnewtoni[n]$};
  \end{scope}  
 \begin{scope}[yshift = 1.3cm]
    \coordinate (LB) at (0,0);
    \coordinate (RB) at (2,0);
    \draw (LB) -- (RB);
  \end{scope}  
  \draw (R) -- (RB);
  \draw (L) -- (LB);
  \draw[thick] (ML) .. controls +(0, 1) and +(0, 1) .. (MR);
\end{scope}

}}  \\
  &+\bar{s}\sum_{k+\ell=n} \NB{\tikz[font=\tiny]{\begin{scope}[font =\tiny]
  \begin{scope}
    \coordinate (L) at (0,0);
    \coordinate (R) at (2,0);
    \coordinate (ML) at (0.5, 0);
    \coordinate (MR) at (1.5, 0);
    \draw[->-] (L) -- (ML);
    \draw[->-] (MR) -- (R) node[right] {$a+b$};
    \draw[->-] (ML).. controls + (0.4, 0.4) and +(-0.2, 0.4) .. (MR)
    node[above, pos =0.7] {$a$}     node[above, pos =0.3] {$\dotnewtoni[k]$};
    \draw[->-] (ML).. controls + (0.2, -0.4) and +(-0.4, -0.4) .. (MR)
    node[left, pos =0.3] {$b$} node[above, pos =0.7, yshift = -0.5mm] {$\dotnewtoni[\ell]$};
  \end{scope}  
 \begin{scope}[yshift = 1.3cm]
    \coordinate (LB) at (0,0);
    \coordinate (RB) at (2,0);
    \draw (LB) -- (RB);
  \end{scope}  
  \draw (R) -- (RB);
  \draw (L) -- (LB);
  \draw[thick] (ML) .. controls +(0, 1) and +(0, 1) .. (MR);
\end{scope}

}} 
  \end{split}
  \\
\begin{split} \label{eq:Ln-on-zip}
  \Ln\left( \NB{\tikz[font=\tiny]{\begin{scope}
  \begin{scope}
    \coordinate (L1) at (0.2,0.4);
    \coordinate (L2) at (0,0);
    \coordinate (R1) at (2.2,0.4);
    \coordinate (R2) at (2,0);
    \coordinate (ML) at (0.6, 0.2);
    \coordinate (MR) at (1.6, 0.2);
    \draw[->-] (ML) -- (MR) node[above, midway] {$a+b$};
    \draw (MR) .. controls +(0, 0) and +(-0.3,0) .. (R1) ;
    \draw (MR) .. controls +(0, 0) and +(-0.3,0) .. (R2);
    \draw (L1) .. controls +( 0.3, 0) and +(0,0) .. (ML);
    \draw (L2) .. controls +( 0.3, 0) and +(0,0) .. (ML);
  \end{scope}  
 \begin{scope}[yshift = -1cm]
    \coordinate (L1B) at (0.2,0.4);
    \coordinate (L2B) at (0,0);
    \coordinate (R1B) at (2.2,0.4);
    \coordinate (R2B) at (2,0);
    \draw[->-] (L1B) .. controls +( 0, 0) and +(0,0) .. (R1B) node [right, pos
    = 1] {$a$};
    \draw[->-] (L2B) .. controls +( 0, 0) and +(0,0) .. (R2B) node [right, pos
    = 1] {$b$};
 \end{scope}  
  \draw (R1) -- (R1B);
  \draw (R2) -- (R2B);
  \draw (L1) -- (L1B);
  \draw (L2) -- (L2B);
  \draw[thick] (ML) .. controls +(0, -0.6) and +(0, -0.6) .. (MR);
\end{scope}

}}\right)\ =\
  &   \parone{n}\cdot\ \NB{\tikz[font=\tiny]{\begin{scope}
  \begin{scope}
    \coordinate (L1) at (0.2,0.4);
    \coordinate (L2) at (0,0);
    \coordinate (R1) at (2.2,0.4);
    \coordinate (R2) at (2,0);
    \coordinate (ML) at (0.6, 0.2);
    \coordinate (MR) at (1.6, 0.2);
    \draw[->-] (ML) -- (MR) node[above, midway] {$a+b$};
    \draw (MR) .. controls +(0, 0) and +(-0.3,0) .. (R1) ;
    \draw (MR) .. controls +(0, 0) and +(-0.3,0) .. (R2);
    \draw (L1) .. controls +( 0.3, 0) and +(0,0) .. (ML);
    \draw (L2) .. controls +( 0.3, 0) and +(0,0) .. (ML);
  \end{scope}  
 \begin{scope}[yshift = -1cm]
    \coordinate (L1B) at (0.2,0.4);
    \coordinate (L2B) at (0,0);
    \coordinate (R1B) at (2.2,0.4);
    \coordinate (R2B) at (2,0);
    \draw[->-] (L1B) .. controls +( 0, 0) and +(0,0) .. (R1B) node [right, pos
    = 1] {$a$};
    \draw[->-] (L2B) .. controls +( 0, 0) and +(0,0) .. (R2B) node [right, pos
    = 1] {$b$}   node [pos = 0.2, above] {$\dotnewtoni[0]$};
 \end{scope}  
  \draw (R1) -- (R1B) node [pos = 0.2, left] {$\dotnewtoni[n]$};
  \draw (R2) -- (R2B);
  \draw (L1) -- (L1B);
  \draw (L2) -- (L2B);
  \draw[thick] (ML) .. controls +(0, -0.6) and +(0, -0.6) .. (MR);
\end{scope}

}} 
  +  \partwo{n}\cdot \ \NB{\tikz[font=\tiny]{\begin{scope}
  \begin{scope}
    \coordinate (L1) at (0.2,0.4);
    \coordinate (L2) at (0,0);
    \coordinate (R1) at (2.2,0.4);
    \coordinate (R2) at (2,0);
    \coordinate (ML) at (0.6, 0.2);
    \coordinate (MR) at (1.6, 0.2);
    \draw[->-] (ML) -- (MR) node[above, midway] {$a+b$};
    \draw (MR) .. controls +(0, 0) and +(-0.3,0) .. (R1) ;
    \draw (MR) .. controls +(0, 0) and +(-0.3,0) .. (R2);
    \draw (L1) .. controls +( 0.3, 0) and +(0,0) .. (ML);
    \draw (L2) .. controls +( 0.3, 0) and +(0,0) .. (ML);
  \end{scope}  
 \begin{scope}[yshift = -1cm]
    \coordinate (L1B) at (0.2,0.4);
    \coordinate (L2B) at (0,0);
    \coordinate (R1B) at (2.2,0.4);
    \coordinate (R2B) at (2,0);
    \draw[->-] (L1B) .. controls +( 0, 0) and +(0,0) .. (R1B) node [right, pos
    = 1] {$a$};
    \draw[->-] (L2B) .. controls +( 0, 0) and +(0,0) .. (R2B) node [right, pos
    = 1] {$b$}   node [pos = 0.2, above] {$\dotnewtoni[n]$};
 \end{scope}  
  \draw (R1) -- (R1B) node [pos = 0.2, left] {$\dotnewtoni[0]$};
  \draw (R2) -- (R2B);
  \draw (L1) -- (L1B);
  \draw (L2) -- (L2B);
  \draw[thick] (ML) .. controls +(0, -0.6) and +(0, -0.6) .. (MR);
\end{scope}

}}  \\
  &-\bar{s} \sum_{k+\ell=n} \NB{\tikz[font=\tiny]{\begin{scope}
  \begin{scope}
    \coordinate (L1) at (0.2,0.4);
    \coordinate (L2) at (0,0);
    \coordinate (R1) at (2.2,0.4);
    \coordinate (R2) at (2,0);
    \coordinate (ML) at (0.6, 0.2);
    \coordinate (MR) at (1.6, 0.2);
    \draw[->-] (ML) -- (MR) node[above, midway] {$a+b$};
    \draw (MR) .. controls +(0, 0) and +(-0.3,0) .. (R1) ;
    \draw (MR) .. controls +(0, 0) and +(-0.3,0) .. (R2);
    \draw (L1) .. controls +( 0.3, 0) and +(0,0) .. (ML);
    \draw (L2) .. controls +( 0.3, 0) and +(0,0) .. (ML);
  \end{scope}  
 \begin{scope}[yshift = -1cm]
    \coordinate (L1B) at (0.2,0.4);
    \coordinate (L2B) at (0,0);
    \coordinate (R1B) at (2.2,0.4);
    \coordinate (R2B) at (2,0);
    \draw[->-] (L1B) .. controls +( 0, 0) and +(0,0) .. (R1B) node [right, pos
    = 1] {$a$};
    \draw[->-] (L2B) .. controls +( 0, 0) and +(0,0) .. (R2B) node [right, pos
    = 1] {$b$}   node [pos = 0.2, above] {$\dotnewtoni[\ell]$};
 \end{scope}  
  \draw (R1) -- (R1B) node [pos = 0.2, left] {$\dotnewtoni[k]$};
  \draw (R2) -- (R2B);
  \draw (L1) -- (L1B);
  \draw (L2) -- (L2B);
  \draw[thick] (ML) .. controls +(0, -0.6) and +(0, -0.6) .. (MR);
\end{scope}

}}
  \end{split}
  \\
\begin{split}\label{eq:Ln-on-unzip}
  \Ln\left( \NB{\tikz[font=\tiny]{\begin{scope}
  \begin{scope}
    \coordinate (L1) at (0.2,0.4);
    \coordinate (L2) at (0,0);
    \coordinate (R1) at (2.2,0.4);
    \coordinate (R2) at (2,0);
    \coordinate (ML) at (0.6, 0.2);
    \coordinate (MR) at (1.6, 0.2);
    \draw[->-] (ML) -- (MR) node[below, midway] {$a+b$};
    \draw (MR) .. controls +(0, 0) and +(-0.3,0) .. (R1) ;
    \draw (MR) .. controls +(0, 0) and +(-0.3,0) .. (R2);
    \draw (L1) .. controls +( 0.3, 0) and +(0,0) .. (ML);
    \draw (L2) .. controls +( 0.3, 0) and +(0,0) .. (ML);
  \end{scope}  
 \begin{scope}[yshift = 1cm]
    \coordinate (L1B) at (0.2,0.4);
    \coordinate (L2B) at (0,0);
    \coordinate (R1B) at (2.2,0.4);
    \coordinate (R2B) at (2,0);
    \draw[->-] (L1B) .. controls +( 0, 0) and +(0,0) .. (R1B) node [left, pos
    = 0] {$a$};
    \draw[->-] (L2B) .. controls +( 0, 0) and +(0,0) .. (R2B) node [left, pos
    = 0] {$b$};
 \end{scope}  
  \draw (R1) -- (R1B);
  \draw (R2) -- (R2B);
  \draw (L1) -- (L1B);
  \draw (L2) -- (L2B);
  \draw[thick] (ML) .. controls +(0, 0.6) and +(0, 0.6) .. (MR);
\end{scope}
}}\right)\ =\
  & - \parone{n}\cdot\ \NB{\tikz[font=\tiny]{\begin{scope}
  \begin{scope}
    \coordinate (L1) at (0.2,0.4);
    \coordinate (L2) at (0,0);
    \coordinate (R1) at (2.2,0.4);
    \coordinate (R2) at (2,0);
    \coordinate (ML) at (0.6, 0.2);
    \coordinate (MR) at (1.6, 0.2);
    \draw[->-] (ML) -- (MR) node[below, midway] {$a+b$};
    \draw (MR) .. controls +(0, 0) and +(-0.3,0) .. (R1) ;
    \draw (MR) .. controls +(0, 0) and +(-0.3,0) .. (R2);
    \draw (L1) .. controls +( 0.3, 0) and +(0,0) .. (ML);
    \draw (L2) .. controls +( 0.3, 0) and +(0,0) .. (ML) node [above,
    pos=0.2] {$\dotnewtoni[0]$};
  \end{scope}  
 \begin{scope}[yshift = 1cm]
    \coordinate (L1B) at (0.2,0.4);
    \coordinate (L2B) at (0,0);
    \coordinate (R1B) at (2.2,0.4);
    \coordinate (R2B) at (2,0);
    \draw[->-] (L1B) .. controls +( 0, 0) and +(0,0) .. (R1B) node [left, pos
    = 0] {$a$} node
  [pos=0.8, below] {$\dotnewtoni[n]$};
    \draw[->-] (L2B) .. controls +( 0, 0) and +(0,0) .. (R2B) node [left, pos
    = 0] {$b$};
 \end{scope}  
  \draw (R1) -- (R1B);
  \draw (R2) -- (R2B);
  \draw (L1) -- (L1B);
  \draw (L2) -- (L2B);
  \draw[thick] (ML) .. controls +(0, 0.6) and +(0, 0.6) .. (MR);
\end{scope}

}} 
   -  \partwo{n}\cdot \ \NB{\tikz[font=\tiny]{\begin{scope}
  \begin{scope}
    \coordinate (L1) at (0.2,0.4);
    \coordinate (L2) at (0,0);
    \coordinate (R1) at (2.2,0.4);
    \coordinate (R2) at (2,0);
    \coordinate (ML) at (0.6, 0.2);
    \coordinate (MR) at (1.6, 0.2);
    \draw[->-] (ML) -- (MR) node[below, midway] {$a+b$};
    \draw (MR) .. controls +(0, 0) and +(-0.3,0) .. (R1) ;
    \draw (MR) .. controls +(0, 0) and +(-0.3,0) .. (R2);
    \draw (L1) .. controls +( 0.3, 0) and +(0,0) .. (ML);
    \draw (L2) .. controls +( 0.3, 0) and +(0,0) .. (ML) node [above,
    pos=0.2] {$\dotnewtoni[n]$};
  \end{scope}  
 \begin{scope}[yshift = 1cm]
    \coordinate (L1B) at (0.2,0.4);
    \coordinate (L2B) at (0,0);
    \coordinate (R1B) at (2.2,0.4);
    \coordinate (R2B) at (2,0);
    \draw[->-] (L1B) .. controls +( 0, 0) and +(0,0) .. (R1B) node [left, pos
    = 0] {$a$} node
  [pos=0.8, below] {$\dotnewtoni[0]$};
    \draw[->-] (L2B) .. controls +( 0, 0) and +(0,0) .. (R2B) node [left, pos
    = 0] {$b$};
 \end{scope}  
  \draw (R1) -- (R1B);
  \draw (R2) -- (R2B);
  \draw (L1) -- (L1B);
  \draw (L2) -- (L2B);
  \draw[thick] (ML) .. controls +(0, 0.6) and +(0, 0.6) .. (MR);
\end{scope}

}}  \\
  &-s \sum_{k+\ell=n} \NB{\tikz[font=\tiny]{\begin{scope}
  \begin{scope}
    \coordinate (L1) at (0.2,0.4);
    \coordinate (L2) at (0,0);
    \coordinate (R1) at (2.2,0.4);
    \coordinate (R2) at (2,0);
    \coordinate (ML) at (0.6, 0.2);
    \coordinate (MR) at (1.6, 0.2);
    \draw[->-] (ML) -- (MR) node[below, midway] {$a+b$};
    \draw (MR) .. controls +(0, 0) and +(-0.3,0) .. (R1) ;
    \draw (MR) .. controls +(0, 0) and +(-0.3,0) .. (R2);
    \draw (L1) .. controls +( 0.3, 0) and +(0,0) .. (ML);
    \draw (L2) .. controls +( 0.3, 0) and +(0,0) .. (ML) node [above,
    pos=0.2] {$\dotnewtoni[k]$};
  \end{scope}  
 \begin{scope}[yshift = 1cm]
    \coordinate (L1B) at (0.2,0.4);
    \coordinate (L2B) at (0,0);
    \coordinate (R1B) at (2.2,0.4);
    \coordinate (R2B) at (2,0);
    \draw[->-] (L1B) .. controls +( 0, 0) and +(0,0) .. (R1B) node [left, pos
    = 0] {$a$} node
  [pos=0.8, below] {$\dotnewtoni[\ell]$};
    \draw[->-] (L2B) .. controls +( 0, 0) and +(0,0) .. (R2B) node [left, pos
    = 0] {$b$};
 \end{scope}  
  \draw (R1) -- (R1B);
  \draw (R2) -- (R2B);
  \draw (L1) -- (L1B);
  \draw (L2) -- (L2B);
  \draw[thick] (ML) .. controls +(0, 0.6) and +(0, 0.6) .. (MR);
\end{scope}

}} 
  \end{split}
  \\
\begin{split}   \label{eq:Ln-on-cup}
  \Ln\left( \NB{\tikz[font=\tiny, scale=1.2]{\begin{scope}
  \draw (0,0) arc (180 :0: 0.5cm and 0.2cm) node[above, pos =
  0.5] {$a$};
  \draw[very thin] (0,0) arc (180 :0: 0.5cm and -0.6cm);
  \draw (0,0) arc (180 :0: 0.5cm and -0.2cm);
\end{scope}

}}\right)\ =\
  &  \parthree{n}\cdot\ \NB{\tikz[font=\tiny, scale=1.2]{\begin{scope}
  \draw (0,0) arc (180 :0: 0.5cm and 0.2cm) node[above, pos =
  0.5] {$a$};
  \draw[very thin] (0,0) arc (180 :0: 0.5cm and -0.6cm) node[pos=0.5,
  above] {$\dotnewtoni[0] \wdotnewtoni[n]$};
  \draw (0,0) arc (180 :0: 0.5cm and -0.2cm);
\end{scope}

}} 
  -  \parthree{n} \cdot\ \NB{\tikz[font=\tiny, scale=1.2]{\begin{scope}
  \draw (0,0) arc (180 :0: 0.5cm and 0.2cm) node[above, pos =
  0.5] {$a$};
  \draw[very thin] (0,0) arc (180 :0: 0.5cm and -0.6cm) node[pos=0.5,
  above] {$\dotnewtoni[n] \wdotnewtoni[0]$};
  \draw (0,0) arc (180 :0: 0.5cm and -0.2cm);
\end{scope}

}}
  +\frac{1}{2} \sum_{k+\ell=n} \NB{\tikz[font=\tiny, scale=1.2]{\begin{scope}
  \draw (0,0) arc (180 :0: 0.5cm and 0.2cm) node[above, pos =
  0.5] {$a$};
  \draw[very thin] (0,0) arc (180 :0: 0.5cm and -0.6cm) node[pos=0.5,
  above] {$\dotnewtoni[k] \wdotnewtoni[\ell]$};
  \draw (0,0) arc (180 :0: 0.5cm and -0.2cm);
\end{scope}

}}
  \end{split}
  \\[3pt]
\begin{split} \label{eq:Ln-on-cap}
  \Ln\left( \NB{\tikz[font=\tiny, scale=1.2]{\begin{scope}
  \draw (0,0) arc (180 :0: 0.5cm and 0.2cm);
  \draw[very thin] (0,0) arc (180 :0: 0.5cm and 0.6cm);
  \draw (0,0) arc (180 :0: 0.5cm and -0.2cm)node[ below, pos =
  0.5] {$a$};
\end{scope}

}}\right)\ =\
  &-   \parthree{n}\cdot\ \NB{\tikz[font=\tiny, scale=1.2]{\begin{scope}
  \draw (0,0) arc (180 :0: 0.5cm and 0.2cm);
  \draw[very thin] (0,0) arc (180 :0: 0.5cm and 0.6cm) node[pos=0.5,
  below] {$\dotnewtoni[0] \wdotnewtoni[n]$};
  \draw (0,0) arc (180 :0: 0.5cm and -0.2cm)node[ below, pos =
  0.5] {$a$};
\end{scope}

}} 
  +  \parthree{n}\cdot\ \NB{\tikz[font=\tiny, scale=1.2]{\begin{scope}
  \draw (0,0) arc (180 :0: 0.5cm and 0.2cm);
  \draw[very thin] (0,0) arc (180 :0: 0.5cm and 0.6cm) node[pos=0.5,
  below] {$\dotnewtoni[n] \wdotnewtoni[0]$};
  \draw (0,0) arc (180 :0: 0.5cm and -0.2cm)node[ below, pos =
  0.5] {$a$};
\end{scope}

}} 
 +\frac{1}{2} \sum_{k+\ell=n}\NB{\tikz[font=\tiny,
 scale=1.2]{\begin{scope}
  \draw (0,0) arc (180 :0: 0.5cm and 0.2cm);
  \draw[very thin] (0,0) arc (180 :0: 0.5cm and 0.6cm) node[pos=0.5,
  below] {$\dotnewtoni[k] \wdotnewtoni[\ell]$};
  \draw (0,0) arc (180 :0: 0.5cm and -0.2cm)node[ below, pos =
  0.5] {$a$};
\end{scope}

}}
 \end{split}
\end{gather}
Finally extend the action of $\Ln$ on all foams in good position by imposing the
Leibniz rule.

On the one hand, the fact that $(\parone{n})_{n\in \NNN}, (\partwo{n})_{n\in \NNN}$ and $(\parthree{n})_{n\in \NNN}$ are Witt sequences is key to proving that these definitions give indeed an action of the Witt algebra (see proof of Lemma~\ref{lem:witt-acts-good-position}). On the other hand, the relations between various parameters appearing in equations \eqref{eq:Ln-on-pol} -- \eqref{eq:Ln-on-cap} are there to ensure compatibility of this action with the topology of foams. In particular we want these definitions to be invariant under ambient isotopies of foams. For instance we
would like that: 
\[
\Ln \left( \NB{\tikz[font=\tiny, scale =0.8]{\begin{scope}
  \begin{scope}[yshift=1cm]
    \coordinate (L1T) at (0.2,0.4);
    \coordinate (L2T) at (0,0);
    \coordinate (R1T) at (2.2,0.4);
    \coordinate (R2T) at (2,0);
    \coordinate (MLT) at (0.6, 0.2);
    \coordinate (MRT) at (2.4, 0.2);
    \draw[->-] (MLT) -- (MRT) node[above, midway] {$a+b$};
    \draw[->-] (L1T) .. controls +( 0.3, 0) and +(0,0) .. (MLT)  node[pos =0, left] {$b$};
    \draw[->-] (L2T) .. controls +( 0.3, 0) and +(0,0) .. (MLT)
    node[pos =0, left] {$a$};
  \end{scope}

  \begin{scope}[yshift=0cm]
    \coordinate (L1) at (0.2,0.4);
    \coordinate (L2) at (0,0);
    \coordinate (R1) at (2.2,0.4);
    \coordinate (R2) at (2,0);
    \coordinate (ML) at (0.6, 0.2);
    \coordinate (MR1) at (1.3, 0.2);
    \coordinate (MR2) at (2, 0.2);
    \coordinate (MR) at (2.4, 0.2);
    \draw[densely dotted] (ML) -- (MR1);
    \draw[densely dotted] (MR2) -- (MR);
    \draw[densely dotted] (L1) .. controls +( 0.3, 0) and +(0,0) .. (ML);
    \draw[densely dotted] (L2) .. controls +( 0.3, 0) and +(0,0)
    .. (ML);
    \draw[densely dotted] (MR1) .. controls +( 0.3, -0.2) and +(-0.3,-0.2) .. (MR2);
    \draw[densely dotted] (MR1) .. controls +( 0.3, 0.2) and +(-0.3 ,0.2) .. (MR2);

  \end{scope}

  \begin{scope}[yshift = -1cm]
    \coordinate (L1B) at (0.2,0.4);
    \coordinate (L2B) at (0,0);
    \coordinate (R1B) at (2.2,0.4);
    \coordinate (R2B) at (2,0);
    \coordinate (MR2B) at (2, 0.2);
    \coordinate (MR3B) at (2.4, 0.2);
    \draw[->-] (L1B) .. controls +( 1.2, 0) and +(0,0) .. (MR2B);
    \draw[->-] (L2B) .. controls +( 1.2, 0) and +(0,0) .. (MR2B);
    \draw[->-] (MR2B) -- (MR3B);
 \end{scope}  
  \draw (MRT) -- (MR3B);
  \draw (L1T) -- (L1B);
  \draw (L2T) -- (L2B);
  \draw[thick] (MLT) -- (ML).. controls +(0, -0.6) and +(0, -0.6)
  .. (MR1) .. controls +(0, 0.6) and +(0, 0.6) .. (MR2) -- (MR2B);
\end{scope}

}}\right) = 0.
\]
The parameters are chosen so that such relations hold. Note however
that we do not have a full list of moves to be checked to ensure
consistency of these definitions. Instead we will use the power of the
universal construction (see Proposition~\ref{prop:witt-acts-good-pos}
and the proof of Theorem~\ref{thm:wittaction-spherical}).

\begin{lem}\label{lem:witt-acts-good-position}
  Mapping $\LLn$ to $\Ln$ for all $n\in \NNN$ defines an action of $\ourwitt$ on the
  $\scalars$-module generated by \spherical{} foams in good position.
\end{lem}

\begin{proof}
  We need to prove that $\Ln \circ \Ln[m] - \Ln[m]\circ \Ln[m]= (n-m) \Ln[n+m]$
  for any two $n,m$ in $\NNN$. Without loss of generality, we may fix
  $m$ and $n$ in $\NNN$ with $m\geq n$. Note that for any two \spherical{}
  foams $F$ and $G$ in good position,
  \begin{align}
    &[\Ln, \Ln[m]](F\circ G)
    = (\Ln\circ \Ln[m])(F\circ G)
      -  (\Ln[m] \circ \Ln[m])(F\circ G)\\
    &\qquad = \Ln(\Ln[m](F)\circ G + F\circ \Ln[m](G))
      - \Ln[m](\Ln(F)\circ G + F\circ \Ln(G)) \\
    &\qquad= (\Ln(\Ln[m](F))  - \Ln[m](\Ln(F)))\circ G + F\circ
      (\Ln(\Ln[m](G))  - \Ln[m](\Ln(G))) \\ & \qquad =
                                              ([\Ln, \Ln[m]](F))\circ
                                              G) + F\circ ([\Ln, \Ln[m]](G)) .
  \end{align}
  Thus $\Ln\circ \Ln[m] - \Ln[m] \circ \Ln[m]$ satisfies the Leibniz rule as well as
  $(n-m)\Ln[n+m]$. 
  
  Therefore, it is enough to check that the
  relations hold on basic webs. For traces of isotopies, this is
  trivial. For polynomials this follows from the fact that $\ourwitt$
  acts on the ring of symmetric polynomials defined in \eqref{eq:witt-on-poly}. The
  remaining basic foams to inspect are the zip, unzip, digon-cup,
  digon-cap, cap and cup foams. In all cases, this is a relatively
  straightforward computation (and all computations are similar). We
  treat the cup foam and leave the rest to the reader.

  \begin{align}
    \begin{split}
  &\Ln\circ\Ln[m]\left( \NB{\tikz[font=\tiny, scale=1.1]{}}\right)\ =\
 \Ln\left(    \parthree{m}  \cdot\ \NB{\tikz[font=\tiny, scale=1.1]{\begin{scope}
  \draw (0,0) arc (180 :0: 0.5cm and 0.2cm) node[above, pos =
  0.5] {$a$};
  \draw[very thin] (0,0) arc (180 :0: 0.5cm and -0.6cm) node[pos=0.5,
  above] {$\dotnewtoni[0] \wdotnewtoni[m]$};
  \draw (0,0) arc (180 :0: 0.5cm and -0.2cm);
\end{scope}

}}
  - \parthree{m} \cdot\ \NB{\tikz[font=\tiny, scale=1.1]{\begin{scope}
  \draw (0,0) arc (180 :0: 0.5cm and 0.2cm) node[above, pos =
  0.5] {$a$};
  \draw[very thin] (0,0) arc (180 :0: 0.5cm and -0.6cm) node[pos=0.5,
  above] {$\dotnewtoni[m] \wdotnewtoni[0]$};
  \draw (0,0) arc (180 :0: 0.5cm and -0.2cm);
\end{scope}

}} 
 +\frac{1}{2} \sum_{k+\ell=m}\NB{\tikz[font=\tiny,
    scale=1.1]{}}\right) \\
    &=
   -m\parthree{m} \cdot\ \NB{\tikz[font=\tiny, scale=1.3]{\begin{scope}
  \draw (0,0) arc (180 :0: 0.5cm and 0.2cm) node[above, pos =
  0.5] {$a$};
  \draw[very thin] (0,0) arc (180 :0: 0.5cm and -0.6cm) node[pos=0.5,
  above] {$\dotnewtoni[0] \wdotnewtoni[m+n]$};
  \draw (0,0) arc (180 :0: 0.5cm and -0.2cm);
\end{scope}

}}
  + m\parthree{m} \cdot\ \NB{\tikz[font=\tiny, scale=1.3]{\begin{scope}
  \draw (0,0) arc (180 :0: 0.5cm and 0.2cm) node[above, pos =
  0.5] {$a$};
  \draw[very thin] (0,0) arc (180 :0: 0.5cm and -0.6cm) node[pos=0.5,
  above] {$\dotnewtoni[m+n] \wdotnewtoni[0]$};
  \draw (0,0) arc (180 :0: 0.5cm and -0.2cm);
\end{scope}

}} 
 -\frac{1}{2} \sum_{k+\ell=m}k\cdot\ \NB{\tikz[font=\tiny,
      scale=1.3]{\begin{scope}
  \draw (0,0) arc (180 :0: 0.5cm and 0.2cm) node[above, pos =
  0.5] {$a$};
  \draw[very thin] (0,0) arc (180 :0: 0.5cm and -0.6cm) node[pos=0.5,
  above] {$\dotnewtoni[k+n] \wdotnewtoni[\ell]$};
  \draw (0,0) arc (180 :0: 0.5cm and -0.2cm);
\end{scope}

}}
+\ell\cdot\ \NB{\tikz[font=\tiny,
    scale=1.3]{\begin{scope}
  \draw (0,0) arc (180 :0: 0.5cm and 0.2cm) node[above, pos =
  0.5] {$a$};
  \draw[very thin] (0,0) arc (180 :0: 0.5cm and -0.6cm) node[pos=0.5,
  above] {$\dotnewtoni[k] \wdotnewtoni[\ell+n]$};
  \draw (0,0) arc (180 :0: 0.5cm and -0.2cm);
\end{scope}

}}
    \\ 
    &\quad  + \ \text{terms symmetric in $n$ and $m$.}
\end{split}
  \end{align}

The reader may have noticed that if $m=-1$, $\dotnewtoni[m] =0$  and
therefore is it is not true in this case that $L_n(\dotnewtoni[m]) =
-m\dotnewtoni[m+n]$. However since $\nu_{-1}=0$, the identity is
also correct in this case. Similar phenomena happens when $m=0$.
  
  Hence we obtain: \begin{align} \label{eqcomm}
    \begin{split}
    [\Ln, \Ln[m]] \left( \NB{\tikz[font=\tiny, scale=1.1]{}}\right)\ =\
     (n\parthree{n} - m\parthree{m})\cdot\   
    \NB{\tikz[font=\tiny, scale=1.3]{}}\    
    -(n\parthree{n} - m\parthree{m})\cdot\   
    \NB{\tikz[font=\tiny, scale=1.3]{}} \\ 
    \quad-\frac{1}{2} \sum_{k+\ell=m}k\cdot\ \NB{\tikz[font=\tiny,
      scale=1.3]{}}
      +\ell\cdot\ \NB{\tikz[font=\tiny,
      scale=1.3]{}} \ +
      \frac{1}{2} \sum_{k+\ell=n}k\cdot\ \NB{\tikz[font=\tiny,
      scale=1.3]{\begin{scope}
  \draw (0,0) arc (180 :0: 0.5cm and 0.2cm) node[above, pos =
  0.5] {$a$};
  \draw[very thin] (0,0) arc (180 :0: 0.5cm and -0.6cm) node[pos=0.5,
  above] {$\dotnewtoni[k+m] \wdotnewtoni[\ell]$};
  \draw (0,0) arc (180 :0: 0.5cm and -0.2cm);
\end{scope}

}}
      +\ell\cdot\ \NB{\tikz[font=\tiny,
      scale=1.3]{\begin{scope}
  \draw (0,0) arc (180 :0: 0.5cm and 0.2cm) node[above, pos =
  0.5] {$a$};
  \draw[very thin] (0,0) arc (180 :0: 0.5cm and -0.6cm) node[pos=0.5,
  above] {$\dotnewtoni[k] \wdotnewtoni[\ell+m]$};
  \draw (0,0) arc (180 :0: 0.5cm and -0.2cm);
\end{scope}

}}. 
  \end{split}
  \end{align}
  One has $(n\parthree{n} - m\parthree{m}) =
  (n-m)\parthree{n+m}$ because $(\parthree{n})_{n\in\NNN}$ is a Witt-sequence.
Let us now deal the terms in the second
  line of \eqref{eqcomm}. This
  is a linear combination of   
  \[
x_{ij}:=\NB{\tikz[font=\tiny,
      scale=1.3]{\begin{scope}
  \draw (0,0) arc (180 :0: 0.5cm and 0.2cm) node[above, pos =
  0.5] {$a$};
  \draw[very thin] (0,0) arc (180 :0: 0.5cm and -0.6cm) node[pos=0.5,
  above] {$\dotnewtoni[i] \wdotnewtoni[j]$};
  \draw (0,0) arc (180 :0: 0.5cm and -0.2cm);
\end{scope}

}} 
  \]
  where $i$ and $j$ are non-negative integers with $i+j=m+n$. We treat
  three cases: $0 \leq i \leq n$, $n < i < m$ and $m\leq i \leq m+n$. 
\begin{itemize}
  \item If $0 \leq i \leq n$, then $m\leq j \leq m+n$. In this case, the
  coefficient of $x_{ij}$ is
  \[
    \frac{1}2\left(-(j-n) + (j-m)\right) = \frac{n-m}{2}.
  \]

  \item If $n < i < m$, then $n< j < m$. In this case, the
  coefficient of $x_{ij}$ is
  \[
    \frac{1}2\left(-(i-n) - (j-n)\right) = \frac{2n -(i+j)}{2} =\frac{n-m}{2}.
  \]
  
  \item If $ m \leq i \leq m+n,$, then $0\leq j \leq n$. In this case, the
  coefficient of $x_{ij}$ is
  \[
    \frac{1}2\left(-(i-n) + (i-m)\right) = \frac{n-m}{2}.
  \]
\end{itemize}
  Hence we have:
  \begin{align}
    \begin{split}
    [\Ln, \Ln[m]] \left( \NB{\tikz[font=\tiny, scale=1.1]{}}\right)=\
    & (n-m)\parthree{m+n}\cdot\   
    \NB{\tikz[font=\tiny, scale=1.3]{}}\    
    -(n-m)\parthree{m+n}\cdot\   
    \NB{\tikz[font=\tiny, scale=1.3]{}} \\ 
    &\quad+ \frac{n-m}{2} \sum_{i+j=m+n}\cdot\ \NB{\tikz[font=\tiny,
      scale=1.3]{}}
    \end{split} \\
    =&(n-m)\Ln[m+n]\left( \NB{\tikz[font=\tiny,
         scale=1.1]{}}\right) . \qedhere      
  \end{align}
\end{proof}

\begin{prop}\label{prop:witt-acts-good-pos}
  Let $\foam$ be a closed \spherical{} foam in good position, then for all $n\in \NNN$
  \begin{equation} \bracketN{\Ln \cdot \foam}= \LLn\cdot \bracketN{\foam} .\end{equation}
\end{prop}

\begin{proof}
  Note that for all $n \in \NNN$, $\Ln\cdot \foam$ is a (linear
  combination of) foam(s) which is (are) equal to $\foam$ if we ignore
  the decorations. Hence the colorings of $\foam$ are in $1$-to-$1$
  correspondence with the colorings of $\Ln\cdot\foam$. 
  If $c$ is a   coloring of $\foam$, then the corresponding coloring of $\Ln\cdot
  \foam$ is still denoted $c$. For every coloring $c$ of $\foam$, one
  has: $Q(\Ln(\foam),c) = Q(\foam,c)$ and $s(\foam,c) = s(\Ln(\foam),
  c)$. 
  We will prove that:
  \begin{equation}\LLn\cdot \bracketN{\foam,c} = \bracketN{\Ln \cdot \foam,c},\end{equation} 
which then implies the proposition by summing over all colorings of
  $\foam$.

  We compute:
  \begin{align}
    \begin{split}
    \LLn(\bracketN{\foam,c})
    &= \LLn\cdot\left((-1)^{s(\foam,c)}\frac{P(\foam,c)}{Q(\foam,c)}\right)
    \\ 
    &= (-1)^{s(\foam,c)}\left( \frac{\LLn \cdot P(\foam,c)}{Q(\foam,c)} -
      \frac{P(\foam,c)(\Ln \cdot Q(\foam,c))}{Q(\foam,c)^2}\right) .
    \end{split}
\end{align}
We focus on $\LLn \cdot Q(\foam,c)$:
  \begin{align}
    \LLn \cdot Q(\foam,c)
    &= \LLn \cdot \left(
      \prod_{1\leq i <j \leq \myN}(X_i -X_j)^{\chi(F_{ij}(c))/2}
      \right)\\
    &= -\sum_{1\leq i<j\leq\myN}
      \frac{\chi(F_{ij}(c))}{2}\left(
      \sum_{k+\ell=n}\newtoni[k](X_i)\newtoni[\ell](X_j)Q(\foam,c)\right),
  \end{align}
  where the last step follows from \eqref{eq:Ln-acts-on-Nabla}. Finally, following the definition of complete symmetric polynomials in two variables, we get:
  \begin{align}
        \LLn \cdot \bracketN{\foam,c} =
   \frac{\displaystyle{\LLn \cdot P(\foam,c) +\sum_{1\leq i<j\leq\myN}
      \frac{\chi(F_{ij}(c))}{2} h_n(X_i, X_j)P(\foam,c) }}{(-1)^{s(\foam,c)} Q(\foam,c)}.
  \end{align}
We now look at $P(\Ln(\foam),c)$. Note that for each basic foam,
  $\Ln(\foam)$ is equal to $\foam$ with some additional
  decorations. Due to the Leibniz rule used to define the operator
  $\Ln$, one has that:
  \begin{align}P(\Ln(\foam),c) = \Ln(P(\foam,c)) + R(\foam,c)P(\foam
    ,c)\end{align}
  with $R(\foam,c)$ a sum of polynomials, one for
  each basic foam, evaluated using the colors associated by $c$ to the
  facets of these basic foams.

\begin{table}
\begin{tabular}{|c|c|c|c|}
  \hline
  Basic foam
  &Contribution to $R(\foam,c)$
  \\\hline
  \NB{\tikz[font=\tiny]{}}
  &$\displaystyle{s\sum_{i\in I}\sum_{j\in J} h_n(X_i,X_j)  +
    \parone{n} \newtoni[n](\underline{X}_I) + \partwo{n}\newtoni[n](\underline{X}_J)}$
  \\[0.5cm]\hline
  \NB{\tikz[font=\tiny]{}}
  &$\displaystyle{\overline{s}\sum_{i\in I}\sum_{j\in J} h_n(X_i,X_j) -
    \parone{n} \newtoni[n](\underline{X}_I) - \partwo{n}\newtoni[n](\underline{X}_J)}$
  \\[0.5cm]\hline
  \NB{\tikz[font=\tiny]{}}
  &$\displaystyle{-\overline{s}\sum_{i\in I}\sum_{j\in J} h_n(X_i,X_j) +
    \parone{n} \newtoni[n](\underline{X}_I) + \partwo{n}\newtoni[n](\underline{X}_J)}$
  \\ [0.5cm]\hline
  \NB{\tikz[font=\tiny]{}}
  &$\displaystyle{-s\sum_{i\in I}\sum_{j\in J} h_n(X_i,X_j) -
    \parone{n} \newtoni[n](\underline{X}_I) - \partwo{n}\newtoni[n](\underline{X}_J)}$
  \\ [0.5cm]\hline
  \NB{\tikz[font=\tiny]{\begin{scope}
  \fill[white, opacity = 0.6] (0,0) arc (180:0: 0.5cm and 0.2cm)
  arc(0:180: 0.5cm and -0.6cm);
  \fill[blue, opacity = 1, pattern=vertical lines, pattern
    color=blue] (0,0) arc (180:0: 0.5cm and 0.2cm)
  arc(0:180: 0.5cm and -0.6cm);
  \fill[white, opacity = 0.6] (0,0) arc (180:0: 0.5cm and -0.2cm)
  arc(0:180: 0.5cm and -0.6cm);
  \fill[blue, opacity = 1, pattern=horizontal lines, pattern
    color=blue] (0,0) arc (180:0: 0.5cm and -0.2cm)
  arc(0:180: 0.5cm and -0.6cm);
  \draw (0,0) arc (180 :0: 0.5cm and 0.2cm) node[above, pos =
  0.5] {$a$};
  \draw[very thin] (0,0) arc (180 :0: 0.5cm and -0.6cm);
  \draw (0,0) arc (180 :0: 0.5cm and -0.2cm);
\end{scope}

}}
  &$\displaystyle{\frac{1}{2}\sum_{i\in I}\sum_{j\notin I}
    h_n(X_i,X_j) - (N-a)\parthree{n} \newtoni[n](\underline{X}_I)} + a\parthree{n}\newtoni[n](\underline{X}_{\widehat{I}})$ 
  \\ [0.5cm]\hline
  \NB{\tikz[font=\tiny]{\begin{scope}
  \fill[white, opacity = 0.6] (0,0) arc (180:0: 0.5cm and 0.2cm)
  arc(0:180: 0.5cm and 0.6cm);
  \fill[blue, opacity = 1, pattern=vertical lines, pattern
    color=blue] (0,0) arc (180:0: 0.5cm and 0.2cm)
  arc(0:180: 0.5cm and 0.6cm);
  \fill[white, opacity = 0.6] (0,0) arc (180:0: 0.5cm and -0.2cm)
  arc(0:180: 0.5cm and  0.6cm);
  \fill[blue, opacity = 1, pattern=horizontal lines, pattern
    color=blue] (0,0) arc (180:0: 0.5cm and -0.2cm)
  arc(0:180: 0.5cm and  0.6cm);
  \draw (0,0) arc (180 :0: 0.5cm and 0.2cm);
  \draw[very thin] (0,0) arc (180 :0: 0.5cm and 0.6cm);
  \draw (0,0) arc (180 :0: 0.5cm and -0.2cm)node[below, pos =
  0.5] {$a$};
\end{scope}

}}
  &$\displaystyle{\frac{1}{2}\sum_{i\in I}\sum_{j\notin I}
    h_n(X_i,X_j) + (N-a)\parthree{n} \newtoni[n](\underline{X}_I)} - a\parthree{n}\newtoni[n](\underline{X}_{\widehat{I}})$
  \\ [0.5cm]\hline
\end{tabular}
\caption{In this table, the blue hashed surface always has thickness $a$,
  the red dotted ones always has thickness $b$. We denote by $I\subseteq
  \pigments$, the color of facets with thickness $a$ and by $J \subseteq
  \pigments$, the color of facets with thickness $b$ for the coloring
  $c$. In particular, $\#I = a$, $\#J = b$ and $I\cap J =
  \emptyset$. Finally $\underline{X}_I$ denotes $\{X_{i} |{i\in I}\}$
  and $\widehat{I} = \pigments \setminus I$.} \label{tab:contrib-R}
\end{table}
  Table~\ref{tab:contrib-R}  summarizes the
  contributions to $R(\foam,c)$ of basic foams (with colorings).

  Note that in Table~\ref{tab:contrib-R}, contributions to $R$ always consist of
  sums of polynomials of the form $\sum_{k+\ell =n}
  \newtoni[k](X_i)\newtoni[\ell](X_j) =X_i^k X_j^\ell= h_n(X_i, X_j)$  for some $1\leq i<j\leq \myN$
  and of the form $\newtoni[n](X_i)$ for some $1\leq i \leq \myN$.
  Let us write\footnote{If $n=1$, this decomposition is not
    unique since $h_1(X_i, X_j) = \newtoni[1](X_i) + \newtoni[1](X_j)$ . However in
    what follows, we can think of $n$ as a purely formal
    variable, so that the decomposition is well-defined.}:
  \begin{align*}
    R(\foam,c) = \sum_{i=1}^\myN r_i \newtoni[n](X_i) + \sum_{1\leq i < j\leq
    \myN} r_{ij}h_n(X_i, X_j).
  \end{align*}
The polynomials $\newtoni[n](X_i)$ and $h_n(X_i, X_j)$ appear in the
  contribution to $R(\foam,c)$ of a basic foam precisely when this basic foam contributes to
  $\Ccup_{ij}$, $\Ccup_{ji}$, $\Ccap_{ij}$, $\Ccap_{ji}$,
  $\Cdigcup_{ij}$, $\Cdigcup_{ji}$, $\Cdigcap_{ij}$,
  $\Cdigcap_{ji}$, $\Czip_{ij}$, $\Czip_{ji}$, $\Cunzip_{ij}$ or
  $\Cunzip_{ji}$. More precisely, one has:

  \begin{align}
    \begin{split}
    r_{ij} =& s\left(\Cdigcup_{ij} + \Cdigcup_{ji}-
    \Cunzip_{ij} -\Cunzip_{ji}\right) +\bar{s}\left(
\Cdigcap_{ij} + \Cdigcap_{ji}-
    \Czip_{ij} -\Czip_{ji}
    \right) \\ &+ \frac{1}{2}\left( \Ccup_{ij} + \Ccup_{ji} +\Ccap_{ij} +
      \Ccap_{ji}  \right) ,
    \end{split}
    \\
    \begin{split}
    r_i=& \sum_{j\neq i}\Big( \parone{n}\left(\Cdigcup_{ij}
    +\Czip_{ij} -\Cdigcap_{ij}-\Cunzip_{ij}\right) +
    \partwo{n}\left(\Cdigcup_{ji} +\Czip_{ji}
    -\Cdigcap_{ji}-\Cunzip_{ji}\right) \\ &
    +\parthree{n}\left( \Ccap_{ij} -\Ccap_{ji} - \Ccup_{ij}+ \Ccup_{ji} \right)
    \Big) .
    \end{split}
  \end{align}
  Using Lemmas~\ref{lem:cupcap-bi} and~\ref{lem:zip-unzip-merge-split}, we conclude that for all $1\leq i\leq \myN$,
  $r_i = 0$ and that for all $1\leq i<j \leq \myN$,
  $r_{ij} = \frac{\chi(\foam_{ij}(c))}{2}$.
  In summary, this means that:
  \begin{align}
  P(\Ln(\foam),c) = \LLn\cdot P(\foam,c) +\sum_{1\leq i<j\leq\myN} \frac{\chi(F_{ij}(c))}{2} h_n(X_i, X_j)P(\foam,c). 
  \end{align}
  Finally, we conclude
  \begin{align}
\bracketN{\Ln(\foam), c}
    &= \frac{\displaystyle{\LLn \cdot P(\foam,c) +\sum_{1\leq i<j\leq\myN}
  \frac{\chi(F_{ij}(c))}{2} h_n(X_i,X_j)P(\foam,c) }}{(-1)^{s(\foam,c)}Q(F,c)} \\
    &= \LLn \cdot \bracketN{\foam,c}.       \qedhere
\end{align}

\end{proof}

\begin{thm}\label{thm:wittaction-spherical}
  For any web $\web$, the operators $(\Ln)_{n\in \NNN}$ induce an
  action of $\ourwitt$ on the equivariant state space $\estatespaceN[\scalars, s]{\web}$. \end{thm}

\begin{proof}
  We need to prove that if a $\scalars[X_1, \dots, X_N]$-linear
  combination $\sum_i \gamma_i \foam_i$ of \spherical{} $\web$-foams in
  good position is equal to $0$ in $\bracketNs{\web}$, then for all $n
  \in \NNN$,
  \begin{align}
    \Ln\cdot \left(\sum_i \gamma_i \foam_i\right) :=
     \sum_i (\LLn\cdot\gamma_i) \foam_i  + \sum_i \gamma_i \Ln(\foam_i)      
  \end{align}
   is equal to $0$ in
  $\bracketNs{\web}$.  In other words, we need to prove that for any
  \spherical{} foam $G\co \web \to \emptyset$, that
  $\sum_i
  (\LLn\cdot \gamma_i) \bracketN{G\circ \foam_i} + \sum_i \gamma_i \bracketN{G\circ \Ln(\foam_i)}=0$. This is a direct
  consequence of Proposition~\ref{prop:witt-acts-good-pos}. Indeed, since $\sum_i \gamma_i
  \foam_i= 0$, one has
  \begin{align}
    0=\LLn \cdot \left(\sum_i \gamma_i \bracketN{G \circ \foam_i} \right)=&
    \sum_i \left( \LLn\gamma_i\right) \bracketN{G \circ \foam_i} 
+    \sum_i \gamma_i \bracketN{\Ln\left(G\right) \circ \foam_i}\\ &
+    \sum_i \gamma_i \bracketN{G \circ \Ln\left(\foam_i\right)} 
  \end{align}
   and $\sum_i \gamma_i \bracketN{\Ln(G) \circ \foam_i} =0$. From this
   we deduce that
   \begin{align}
 \sum_i
  (\LLn \cdot \gamma_i) \bracketN{G\circ \foam_i} + \sum_i \gamma_i \bracketN{G\circ \Ln(\foam_i)}=0
   \end{align}
   as desired. 
\end{proof}

\begin{exa} Let us explain how to compute the action of $\Ln$ on the
  following piece of foam. \[
\NB
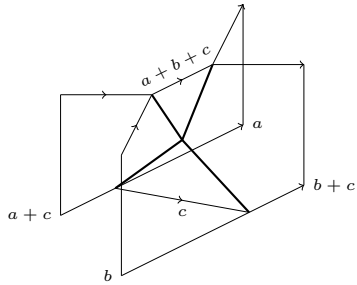
{\tikz[scale=0.8]{\begin{scope}[scale =1]
  \begin{scope}
    \coordinate (BL1) at (-1, 1);
    \coordinate (BL2) at ( 0, 0);
    \coordinate (BR2) at ( 3, 1.5);
    \coordinate (BR1) at ( 2, 2.5);
    \coordinate (BM1) at ($0.3*(BR1) + 0.7*(BL1)$);
    \coordinate (BM2) at ($0.3*(BL2) + 0.7*(BR2)$);   
    \coordinate (TL1) at (-1,3);
    \coordinate (TL2) at (0,2);
    \coordinate (TR2) at (3,3.5);
    \coordinate (TR1) at (2,4.5);
    \coordinate (TML) at (0.5,3);
    \coordinate (TMR) at (1.5,3.5);
    \coordinate (c) at   ($0.25*(BM1) + 0.25*(BM2) + 0.25*(TML) + 0.25*(TMR)$);
    
  \end{scope}

  \begin{scope}[very thin, font= \tiny]
    \draw[->] (BL1) -- (BR1) node[pos =0, left] {$a+c$} node[pos =1, right] {$a$};
    \draw[->] (BL2) -- (BR2) node[pos =0, left] {$b$} node[pos =1, right] {$b+c$};
    \draw[->-] (BM1) -- (BM2) node [pos =0.5, below] {$c$};
    \draw[->-] (TL1) -- (TML);
    \draw[->-] (TL2) -- (TML);
    \draw[->-] (TML) -- (TMR) node [pos =0.5, above, sloped ] {$a+b+c$};
    \draw[->] (TMR) -- (TR1);
    \draw[->] (TMR) -- (TR2);
    \draw (BL1) -- (TL1);
    \draw (BL2) -- (TL2);
    \draw (BR1) -- (TR1);
    \draw (BR2) -- (TR2);
  \end{scope}
  \foreach \i in {(BM1), (BM2), (TML), (TMR)}
  \draw[thick] (c) -- \i;
\end{scope}}}
\]
This foam is \emph{not} in good position, so we first isotope it slightly. As such,
it appears as a composition, that we represent as a movie.
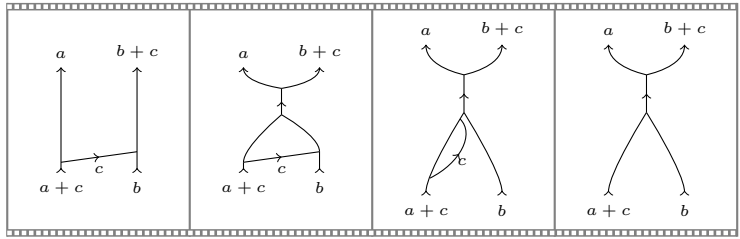
\[
  \mymoviefour[xscale =0.6]{\NB{\tikz[]{\begin{scope}[font =\tiny, yscale =0.7, rotate = 90]
  \coordinate (BL) at (-1, -0.5);
  \coordinate (BR) at ( 1, -0.5);
  \coordinate (TL) at (-1,  0.5);
  \coordinate (TR) at ( 1,  0.5);
  \coordinate (BM) at ($0.8*(BL) + 0.2*(BR)$);
  \coordinate (TM) at ($0.9*(TL) + 0.1*(TR)$);   

  \draw[>->] (BL) -- (BR) node [pos = 0, below] {$b$} node [pos =1, above]
  {$b+c$};
  \draw[>->] (TL) -- (TR) node [pos = 0, below] {$a+c$} node [pos =1, above] {$a$};
  \draw[->-] (TM) -- (BM) node [pos =0.5, below]{$c$};
\end{scope}}}}{\NB{\tikz[]{\begin{scope}[font =\tiny, yscale =0.7, rotate =90]
  \coordinate (BL) at (-1, -0.5);
  \coordinate (BR) at ( 1, -0.5);
  \coordinate (ML) at ( 0.1, -0);
  \coordinate (MR) at ( 0.6, -0);
  \coordinate (TL) at (-1,  0.5);
  \coordinate (TR) at ( 1,  0.5);
  \coordinate (BM) at ($0.8*(BL) + 0.2*(BR)$);
  \coordinate (TM) at ($0.9*(TL) + 0.1*(TR)$);   
  
  \draw[->-] (ML) -- (MR);
  \draw[->] (MR) .. controls +(0,0) and +(-0.3, 0).. (TR) node [pos =1, above]
  {$a$};
  \draw[->] (MR) .. controls +(0,0) and +(-0.3, 0) .. (BR) node [pos =1, above]
  {$b+c$};
  \draw[>-] (BL) -- (BM) node [pos = 0, below] {$b$};
  \draw[>-] (TL) -- (TM) node [pos = 0, below] {$a+c$};
  \draw (ML) .. controls +(0,0) and +(0.3, 0).. (TM);
  \draw (ML) .. controls +(0,0) and +(0.3, 0).. (BM);  
  \draw[->-] (TM) -- (BM) node [pos =0.5, below]{$c$};
\end{scope}}}}{\NB{\tikz[]{\begin{scope}[font =\tiny, yscale =1, rotate =90]
  \coordinate (BL) at (-1, -0.5);
  \coordinate (BR) at ( 1, -0.5);
  \coordinate (ML) at ( 0.1, -0);
  \coordinate (MR) at ( 0.6, -0);
  \coordinate (TL) at (-1,  0.5);
  \coordinate (TR) at ( 1,  0.5);
  
  \draw[->-] (ML) -- (MR);
  \draw[->] (MR) .. controls +(0,0) and +(-0.3, 0).. (TR) node [pos =1, above]
  {$a$};
  \draw[->] (MR) .. controls +(0,0) and +(-0.3, 0) .. (BR) node [pos =1, above]
  {$b+c$};
  \draw[-<] (ML) .. controls +(0,0) and +(0.3, 0).. (BL) node [pos = 1, below] {$b$};
  \draw[-<] (ML) .. controls +(0,0) and +(0.3, 0).. (TL) node [pos =
  1, below] {$a+c$} coordinate[pos =0.8] (T1)  coordinate[pos =0.2]
  (T2);
  \draw[->-] (T1) .. controls +(0.2, -0.4) and +(-0.2, -0.2) .. (T2)
  node[pos=0.5, below] {$c$};
\end{scope}
}}}{\NB{\tikz[]{\begin{scope}[font =\tiny, yscale =1, rotate =90]
  \coordinate (BL) at (-1, -0.5);
  \coordinate (BR) at ( 1, -0.5);
  \coordinate (ML) at ( 0.1, -0);
  \coordinate (MR) at ( 0.6, -0);
  \coordinate (TL) at (-1,  0.5);
  \coordinate (TR) at ( 1,  0.5);
  
  \draw[->-] (ML) -- (MR);
  \draw[->] (MR) .. controls +(0,0) and +(-0.3, 0).. (TR) node [pos =1, above]
  {$a$};
  \draw[->] (MR) .. controls +(0,0) and +(-0.3, 0) .. (BR) node [pos =1, above]
  {$b+c$};
  \draw[-<] (ML) .. controls +(0,0) and +(0.3, 0).. (BL) node [pos = 1, below] {$b$};
  \draw[-<] (ML) .. controls +(0,0) and +(0.3, 0).. (TL) node [pos =
  1, below] {$a+c$} coordinate[pos =0.8] (T1)  coordinate[pos =0.2]
  (T2);
\end{scope}
}}}
  \]
  Now we can use equations \eqref{eq:Ln-on-zip}, \eqref{eq:Ln-on-MP} and
\eqref{eq:Ln-on-digcap} to compute the action of $\Ln$. Isotoping
back and using the dot migration (see ~\eqref{eq:18}) we obtain:
\begin{align}
  &\Ln\left(~\NB{\tikz[scale=0.8]{}}~\right)  =
    \parone{n}\NB{\tikz[scale=0.8]{\begin{scope}[scale =1, very thin, font= \tiny]
  \begin{scope}
    \coordinate (BL1) at (-1, 1);
    \coordinate (BL2) at ( 0, 0);
    \coordinate (BR2) at ( 3, 1.5);
    \coordinate (BR1) at ( 2, 2.5);
    \coordinate (BM1) at ($0.3*(BR1) + 0.7*(BL1)$);
    \coordinate (BM2) at ($0.3*(BL2) + 0.7*(BR2)$);   
    \coordinate (TL1) at (-1,3);
    \coordinate (TL2) at (0,2);
    \coordinate (TR2) at (3,3.5);
    \coordinate (TR1) at (2,4.5);
    \coordinate (n1) at (1.9, 3.8);
    \coordinate (n2) at (0.4, 0.4);
    \coordinate (TML) at (0.5,3);
    \coordinate (TMR) at (1.5,3.5);
    \coordinate (c) at   ($0.25*(BM1) + 0.25*(BM2) + 0.25*(TML) + 0.25*(TMR)$);
    
  \end{scope}

  \begin{scope}
    \draw[->] (BL1) -- (BR1) node[pos =0, left] {$a+c$} node[pos =1, right] {$a$};
    \draw[->] (BL2) -- (BR2) node[pos =0, left] {$b$} node[pos =1, right] {$b+c$};
    \draw[->-] (BM1) -- (BM2) node [pos =0.5, below] {$c$};
    \draw[->-] (TL1) -- (TML);
    \draw[->-] (TL2) -- (TML);
    \draw[->-] (TML) -- (TMR) node [pos =0.5, above, sloped ] {$a+b+c$};
    \draw[->] (TMR) -- (TR1);
    \draw[->] (TMR) -- (TR2);
    \draw (BL1) -- (TL1);
    \draw (BL2) -- (TL2);
    \draw (BR1) -- (TR1);
    \draw (BR2) -- (TR2);
  \end{scope}
  \foreach \i in {(BM1), (BM2), (TML), (TMR)}
  \draw[thick] (c) -- \i;
  \node at (n1) {$\dotnewtoni[n]$};
  \node at (n2) {$\dotnewtoni[0]$};
\end{scope}}} \\ & \quad + \partwo{n}
  \NB{\tikz[scale=0.8]{\begin{scope}[scale =1, very thin, font= \tiny]
  \begin{scope}
    \coordinate (BL1) at (-1, 1);
    \coordinate (BL2) at ( 0, 0);
    \coordinate (BR2) at ( 3, 1.5);
    \coordinate (BR1) at ( 2, 2.5);
    \coordinate (BM1) at ($0.3*(BR1) + 0.7*(BL1)$);
    \coordinate (BM2) at ($0.3*(BL2) + 0.7*(BR2)$);   
    \coordinate (TL1) at (-1,3);
    \coordinate (TL2) at (0,2);
    \coordinate (TR2) at (3,3.5);
    \coordinate (TR1) at (2,4.5);
    \coordinate (n1) at (1.9, 3.8);
    \coordinate (n2) at (0.4, 0.4);
    \coordinate (TML) at (0.5,3);
    \coordinate (TMR) at (1.5,3.5);
    \coordinate (c) at   ($0.25*(BM1) + 0.25*(BM2) + 0.25*(TML) + 0.25*(TMR)$);
    
  \end{scope}

  \begin{scope}
    \draw[->] (BL1) -- (BR1) node[pos =0, left] {$a+c$} node[pos =1, right] {$a$};
    \draw[->] (BL2) -- (BR2) node[pos =0, left] {$b$} node[pos =1, right] {$b+c$};
    \draw[->-] (BM1) -- (BM2) node [pos =0.5, below] {$c$};
    \draw[->-] (TL1) -- (TML);
    \draw[->-] (TL2) -- (TML);
    \draw[->-] (TML) -- (TMR) node [pos =0.5, above, sloped ] {$a+b+c$};
    \draw[->] (TMR) -- (TR1);
    \draw[->] (TMR) -- (TR2);
    \draw (BL1) -- (TL1);
    \draw (BL2) -- (TL2);
    \draw (BR1) -- (TR1);
    \draw (BR2) -- (TR2);
  \end{scope}
  \foreach \i in {(BM1), (BM2), (TML), (TMR)}
  \draw[thick] (c) -- \i;
  \node at (n1) {$\dotnewtoni[0]$};
  \node at (n2) {$\dotnewtoni[n]$};
\end{scope}}} - \bar{s}\sum_{k+\ell =n} \NB{\tikz[scale=0.8]{\begin{scope}[scale =1, very thin, font= \tiny]
  \begin{scope}
    \coordinate (BL1) at (-1, 1);
    \coordinate (BL2) at ( 0, 0);
    \coordinate (BR2) at ( 3, 1.5);
    \coordinate (BR1) at ( 2, 2.5);
    \coordinate (BM1) at ($0.3*(BR1) + 0.7*(BL1)$);
    \coordinate (BM2) at ($0.3*(BL2) + 0.7*(BR2)$);   
    \coordinate (TL1) at (-1,3);
    \coordinate (TL2) at (0,2);
    \coordinate (TR2) at (3,3.5);
    \coordinate (TR1) at (2,4.5);
    \coordinate (n1) at (1.9, 3.8);
    \coordinate (n2) at (0.4, 0.4);
    \coordinate (TML) at (0.5,3);
    \coordinate (TMR) at (1.5,3.5);
    \coordinate (c) at   ($0.25*(BM1) + 0.25*(BM2) + 0.25*(TML) + 0.25*(TMR)$);
    
  \end{scope}

  \begin{scope}
    \draw[->] (BL1) -- (BR1) node[pos =0, left] {$a+c$} node[pos =1, right] {$a$};
    \draw[->] (BL2) -- (BR2) node[pos =0, left] {$b$} node[pos =1, right] {$b+c$};
    \draw[->-] (BM1) -- (BM2) node [pos =0.5, below] {$c$};
    \draw[->-] (TL1) -- (TML);
    \draw[->-] (TL2) -- (TML);
    \draw[->-] (TML) -- (TMR) node [pos =0.5, above, sloped ] {$a+b+c$};
    \draw[->] (TMR) -- (TR1);
    \draw[->] (TMR) -- (TR2);
    \draw (BL1) -- (TL1);
    \draw (BL2) -- (TL2);
    \draw (BR1) -- (TR1);
    \draw (BR2) -- (TR2);
  \end{scope}
  \foreach \i in {(BM1), (BM2), (TML), (TMR)}
  \draw[thick] (c) -- \i;
  \node at (n1) {$\dotnewtoni[k]$};
  \node at (n2) {$\dotnewtoni[\ell]$};
\end{scope}}}.
\end{align}

\end{exa}

\subsection{Action of \texorpdfstring{$\sll_2$}{sl(2)}}
\label{sec:action-from-sll_2}

Lemma~\ref{lem:sl2-to-witt} gives an embedding of $\sll_2$ in
$\ourwitt$. Hence the action of $\ourwitt$ defined above induces an
action of $\sll_2$ on $\gll_N$-foams. Some of the parameters used
before become redundant and the necessity of inverting $2$ vanishes.
We include formula for the action of $\Le$, $\Lf$ and $\Lh$ (via
some operators denoted by $\de$, $\df$ and $\dh$). They can of course be
deduced from the ones in Section~\ref{sec:action-from-witt} via the
injection of Lemma~\ref{lem:sl2-to-witt}. For the rest of the section, we fix three parameters $\tone, \ttwo, \tthree \in \scalars$. 

As before, $\de$, $\df$ and $\dh$ satisfy
the Leibniz rule with respect to composition of foams and map
traces of isotopies to $0$.  The operator $\de$ acts via $\Ln[-1]$ on
polynomials and by $0$ on any other basic foam. The operator $\dh$ is
defined as follows.\begin{gather}
\label{eq:h-act-pol} \dh\left(\NB{\tikz[scale=1.5, font=\small]{}}\right) 
  =-\deg{R}\cdot\  \NB{\tikz[scale=1.5, font=\small]{}} \\
\label{eq:h-act-assoc}  \dh\left(\NB{\tikz[scale=0.6, font=\tiny]{}}\right)=
  \dh\left(\NB{\tikz[scale=0.6,font=\tiny]{}} \right) =0 \\
\label{eq:h-act-dig-cup}  \dh\left( \NB{\tikz[font=\tiny]{}}\right)\ =\
  ab(\tone+\ttwo)\cdot\ \NB{\tikz[font=\tiny]{}} \\
\label{eq:h-act-dig-cap}  \dh\left( \NB{\tikz[font=\tiny]{}}\right)\ =\
 ab(\overline{\tone}+\overline{\ttwo})\cdot\  \NB{\tikz[font=\tiny]{}} \\
\label{eq:h-act-dig-zip}  \dh\left( \NB{\tikz[font=\tiny]{}}\right)\ =\
-ab(\overline{\tone} + \overline{\ttwo}) \cdot\ \NB{\tikz[font=\tiny]{}} 
 \\
\label{eq:h-act-dig-unzip} \dh\left( \NB{\tikz[font=\tiny]{}}\right)\ =\
 -ab({\tone} + {\ttwo}) \cdot\ \NB{\tikz[font=\tiny]{}} 
  \\
\label{eq:h-act-cup}  \dh\left( \NB{\tikz[font=\tiny, scale=1.2]{}}\right)\ =\
      2a(N-a)\tthree \cdot\ \NB{\tikz[font=\tiny, scale=1.2]{}}
  \\[3pt]
\label{eq:h-act-cap}  \dh\left( \NB{\tikz[font=\tiny, scale=1.2]{}}\right)\ =\
    2a(N-a)\overline{\tthree} \cdot\ \NB{\tikz[font=\tiny,
    scale=1.2]{}} 
\end{gather}
Finally $\df$ is given as
follows.

\begin{gather}
\label{eq:e-act-pol}  \df\left(\NB{\tikz[scale=1.5, font=\small]{}}\right)
  =-\ \NB{\tikz[scale=1.5, font=\small]{\begin{scope}
  \draw (0,0) rectangle (1,1) coordinate [midway] (A);
  \fill (A) circle (0.5mm) node[below] {$\Ln[1](R)$};
\end{scope}}} \\
\label{eq:e-act-assoc}   \df\left(\NB{\tikz[scale=0.6, font=\tiny]{}}\right)=
  \df\left(\NB{\tikz[scale=0.6,font=\tiny]{}} \right) =0 \\
\label{eq:e-act-dig-cup}  \df\left( \NB{\tikz[font=\tiny]{}}\right)\ =\
 - \tone\cdot\ \NB{\tikz[font=\tiny]{\begin{scope}[font=\tiny]
  \begin{scope}
    \coordinate (L) at (0,0);
    \coordinate (R) at (2,0);
    \coordinate (ML) at (0.5, 0);
    \coordinate (MR) at (1.5, 0);
    \draw[->-] (L) -- (ML);
    \draw[->-] (MR) -- (R) node[right] {$a+b$};
    \draw[->-] (ML).. controls + (0.4, 0.4) and +(-0.2, 0.4) .. (MR)
    node[above, pos=0.7 ] {$a$} node[below, pos =0.3] {$\dotnewtoni[1]$};
    \draw[->-] (ML).. controls +(0.2, -0.4) and +(-0.4, -0.4) .. (MR)
    node[below, pos =0.3] {$b$} node[below, pos =0.75] {$\dotnewtoni[0]$};
  \end{scope}  
 \begin{scope}[yshift = -1.3cm]
    \coordinate (LB) at (0,0);
    \coordinate (RB) at (2,0);
    \draw[->-] (LB) -- (RB);
  \end{scope}  
  \draw (R) -- (RB);
  \draw (L) -- (LB);
  \draw[thick] (ML) .. controls +(0, -1) and +(0, -1) .. (MR);
\end{scope}

}} 
  - \ttwo\cdot \ \NB{\tikz[font=\tiny]{\begin{scope}[font=\tiny]
  \begin{scope}
    \coordinate (L) at (0,0);
    \coordinate (R) at (2,0);
    \coordinate (ML) at (0.5, 0);
    \coordinate (MR) at (1.5, 0);
    \draw[->-] (L) -- (ML);
    \draw[->-] (MR) -- (R) node[right] {$a+b$};
    \draw[->-] (ML).. controls + (0.4, 0.4) and +(-0.2, 0.4) .. (MR)
    node[above, pos=0.7 ] {$a$} node[below, pos =0.3] {$\dotnewtoni[0]$};
    \draw[->-] (ML).. controls +(0.2, -0.4) and +(-0.4, -0.4) .. (MR)
    node[below, pos =0.3] {$b$} node[below, pos =0.75] {$\dotnewtoni[1]$};
  \end{scope}  
 \begin{scope}[yshift = -1.3cm]
    \coordinate (LB) at (0,0);
    \coordinate (RB) at (2,0);
    \draw[->-] (LB) -- (RB);
  \end{scope}  
  \draw (R) -- (RB);
  \draw (L) -- (LB);
  \draw[thick] (ML) .. controls +(0, -1) and +(0, -1) .. (MR);
\end{scope}

}}  \\
\label{eq:e-act-dig-cap}   \df\left( \NB{\tikz[font=\tiny]{}}\right)\ =\
   - \overline{\tone}
  \cdot\ \NB{\tikz[font=\tiny]{\begin{scope}[font =\tiny]
  \begin{scope}
    \coordinate (L) at (0,0);
    \coordinate (R) at (2,0);
    \coordinate (ML) at (0.5, 0);
    \coordinate (MR) at (1.5, 0);
    \draw[->-] (L) -- (ML);
    \draw[->-] (MR) -- (R) node[right] {$a+b$};
    \draw[->-] (ML).. controls + (0.4, 0.4) and +(-0.2, 0.4) .. (MR)
    node[above, pos =0.7] {$a$}     node[above, pos =0.3] {$\dotnewtoni[1]$};
    \draw[->-] (ML).. controls + (0.2, -0.4) and +(-0.4, -0.4) .. (MR)
    node[left, pos =0.3] {$b$} node[above, pos =0.7, yshift = -0.5mm] {$\dotnewtoni[0]$};
  \end{scope}  
 \begin{scope}[yshift = 1.3cm]
    \coordinate (LB) at (0,0);
    \coordinate (RB) at (2,0);
    \draw (LB) -- (RB);
  \end{scope}  
  \draw (R) -- (RB);
  \draw (L) -- (LB);
  \draw[thick] (ML) .. controls +(0, 1) and +(0, 1) .. (MR);
\end{scope}

}}  
  \!\!\!\! - \overline{\ttwo}
  \cdot \ \NB{\tikz[font=\tiny]{\begin{scope}[font =\tiny]
  \begin{scope}
    \coordinate (L) at (0,0);
    \coordinate (R) at (2,0);
    \coordinate (ML) at (0.5, 0);
    \coordinate (MR) at (1.5, 0);
    \draw[->-] (L) -- (ML);
    \draw[->-] (MR) -- (R) node[right] {$a+b$};
    \draw[->-] (ML).. controls + (0.4, 0.4) and +(-0.2, 0.4) .. (MR)
    node[above, pos =0.7] {$a$}     node[above, pos =0.3] {$\dotnewtoni[0]$};
    \draw[->-] (ML).. controls + (0.2, -0.4) and +(-0.4, -0.4) .. (MR)
    node[left, pos =0.3] {$b$} node[above, pos =0.7, yshift = -0.5mm] {$\dotnewtoni[1]$};
  \end{scope}  
 \begin{scope}[yshift = 1.3cm]
    \coordinate (LB) at (0,0);
    \coordinate (RB) at (2,0);
    \draw (LB) -- (RB);
  \end{scope}  
  \draw (R) -- (RB);
  \draw (L) -- (LB);
  \draw[thick] (ML) .. controls +(0, 1) and +(0, 1) .. (MR);
\end{scope}

}}  \\
\label{eq:e-act-zip}   \df\left( \NB{\tikz[font=\tiny]{}}\right)\ =\
    \overline{\tone} \cdot\ \NB{\tikz[font=\tiny]{\begin{scope}
  \begin{scope}
    \coordinate (L1) at (0.2,0.4);
    \coordinate (L2) at (0,0);
    \coordinate (R1) at (2.2,0.4);
    \coordinate (R2) at (2,0);
    \coordinate (ML) at (0.6, 0.2);
    \coordinate (MR) at (1.6, 0.2);
    \draw[->-] (ML) -- (MR) node[above, midway] {$a+b$};
    \draw (MR) .. controls +(0, 0) and +(-0.3,0) .. (R1) ;
    \draw (MR) .. controls +(0, 0) and +(-0.3,0) .. (R2);
    \draw (L1) .. controls +( 0.3, 0) and +(0,0) .. (ML);
    \draw (L2) .. controls +( 0.3, 0) and +(0,0) .. (ML);
  \end{scope}  
 \begin{scope}[yshift = -1cm]
    \coordinate (L1B) at (0.2,0.4);
    \coordinate (L2B) at (0,0);
    \coordinate (R1B) at (2.2,0.4);
    \coordinate (R2B) at (2,0);
    \draw[->-] (L1B) .. controls +( 0, 0) and +(0,0) .. (R1B) node [right, pos
    = 1] {$a$};
    \draw[->-] (L2B) .. controls +( 0, 0) and +(0,0) .. (R2B) node [right, pos
    = 1] {$b$}   node [pos = 0.2, above] {$\dotnewtoni[0]$};
 \end{scope}  
  \draw (R1) -- (R1B) node [pos = 0.2, left] {$\dotnewtoni[1]$};
  \draw (R2) -- (R2B);
  \draw (L1) -- (L1B);
  \draw (L2) -- (L2B);
  \draw[thick] (ML) .. controls +(0, -0.6) and +(0, -0.6) .. (MR);
\end{scope}

}} 
  +  \overline{\ttwo}\cdot \ \NB{\tikz[font=\tiny]{\begin{scope}
  \begin{scope}
    \coordinate (L1) at (0.2,0.4);
    \coordinate (L2) at (0,0);
    \coordinate (R1) at (2.2,0.4);
    \coordinate (R2) at (2,0);
    \coordinate (ML) at (0.6, 0.2);
    \coordinate (MR) at (1.6, 0.2);
    \draw[->-] (ML) -- (MR) node[above, midway] {$a+b$};
    \draw (MR) .. controls +(0, 0) and +(-0.3,0) .. (R1) ;
    \draw (MR) .. controls +(0, 0) and +(-0.3,0) .. (R2);
    \draw (L1) .. controls +( 0.3, 0) and +(0,0) .. (ML);
    \draw (L2) .. controls +( 0.3, 0) and +(0,0) .. (ML);
  \end{scope}  
 \begin{scope}[yshift = -1cm]
    \coordinate (L1B) at (0.2,0.4);
    \coordinate (L2B) at (0,0);
    \coordinate (R1B) at (2.2,0.4);
    \coordinate (R2B) at (2,0);
    \draw[->-] (L1B) .. controls +( 0, 0) and +(0,0) .. (R1B) node [right, pos
    = 1] {$a$};
    \draw[->-] (L2B) .. controls +( 0, 0) and +(0,0) .. (R2B) node [right, pos
    = 1] {$b$}   node [pos = 0.2, above] {$\dotnewtoni[1]$};
 \end{scope}  
  \draw (R1) -- (R1B) node [pos = 0.2, left] {$\dotnewtoni[0]$};
  \draw (R2) -- (R2B);
  \draw (L1) -- (L1B);
  \draw (L2) -- (L2B);
  \draw[thick] (ML) .. controls +(0, -0.6) and +(0, -0.6) .. (MR);
\end{scope}

}}  \\
\label{eq:e-act-unzip}   \df\left( \NB{\tikz[font=\tiny]{}}\right)\ =\
    {\tone}\cdot\ \NB{\tikz[font=\tiny]{\begin{scope}
  \begin{scope}
    \coordinate (L1) at (0.2,0.4);
    \coordinate (L2) at (0,0);
    \coordinate (R1) at (2.2,0.4);
    \coordinate (R2) at (2,0);
    \coordinate (ML) at (0.6, 0.2);
    \coordinate (MR) at (1.6, 0.2);
    \draw[->-] (ML) -- (MR) node[below, midway] {$a+b$};
    \draw (MR) .. controls +(0, 0) and +(-0.3,0) .. (R1) ;
    \draw (MR) .. controls +(0, 0) and +(-0.3,0) .. (R2);
    \draw (L1) .. controls +( 0.3, 0) and +(0,0) .. (ML);
    \draw (L2) .. controls +( 0.3, 0) and +(0,0) .. (ML) node [above,
    pos=0.2] {$\dotnewtoni[0]$};
  \end{scope}  
 \begin{scope}[yshift = 1cm]
    \coordinate (L1B) at (0.2,0.4);
    \coordinate (L2B) at (0,0);
    \coordinate (R1B) at (2.2,0.4);
    \coordinate (R2B) at (2,0);
    \draw[->-] (L1B) .. controls +( 0, 0) and +(0,0) .. (R1B) node [left, pos
    = 0] {$a$} node
  [pos=0.8, below] {$\dotnewtoni[1]$};
    \draw[->-] (L2B) .. controls +( 0, 0) and +(0,0) .. (R2B) node [left, pos
    = 0] {$b$};
 \end{scope}  
  \draw (R1) -- (R1B);
  \draw (R2) -- (R2B);
  \draw (L1) -- (L1B);
  \draw (L2) -- (L2B);
  \draw[thick] (ML) .. controls +(0, 0.6) and +(0, 0.6) .. (MR);
\end{scope}

}} 
   +  {\ttwo}\cdot \ \NB{\tikz[font=\tiny]{\begin{scope}
  \begin{scope}
    \coordinate (L1) at (0.2,0.4);
    \coordinate (L2) at (0,0);
    \coordinate (R1) at (2.2,0.4);
    \coordinate (R2) at (2,0);
    \coordinate (ML) at (0.6, 0.2);
    \coordinate (MR) at (1.6, 0.2);
    \draw[->-] (ML) -- (MR) node[below, midway] {$a+b$};
    \draw (MR) .. controls +(0, 0) and +(-0.3,0) .. (R1) ;
    \draw (MR) .. controls +(0, 0) and +(-0.3,0) .. (R2);
    \draw (L1) .. controls +( 0.3, 0) and +(0,0) .. (ML);
    \draw (L2) .. controls +( 0.3, 0) and +(0,0) .. (ML) node [above,
    pos=0.2] {$\dotnewtoni[1]$};
  \end{scope}  
 \begin{scope}[yshift = 1cm]
    \coordinate (L1B) at (0.2,0.4);
    \coordinate (L2B) at (0,0);
    \coordinate (R1B) at (2.2,0.4);
    \coordinate (R2B) at (2,0);
    \draw[->-] (L1B) .. controls +( 0, 0) and +(0,0) .. (R1B) node [left, pos
    = 0] {$a$} node
  [pos=0.8, below] {$\dotnewtoni[0]$};
    \draw[->-] (L2B) .. controls +( 0, 0) and +(0,0) .. (R2B) node [left, pos
    = 0] {$b$};
 \end{scope}  
  \draw (R1) -- (R1B);
  \draw (R2) -- (R2B);
  \draw (L1) -- (L1B);
  \draw (L2) -- (L2B);
  \draw[thick] (ML) .. controls +(0, 0.6) and +(0, 0.6) .. (MR);
\end{scope}

}}  
  \\
\label{eq:e-act-cup}   \df\left( \NB{\tikz[font=\tiny, scale=1.2]{}}\right)\ =\
     - \tthree \cdot\ \NB{\tikz[font=\tiny, scale=1.2]{\begin{scope}
  \draw (0,0) arc (180 :0: 0.5cm and 0.2cm) node[above, pos =
  0.5] {$a$};
  \draw[very thin] (0,0) arc (180 :0: 0.5cm and -0.6cm) node[pos=0.5,
  above] {$\dotnewtoni[0] \wdotnewtoni[1]$};
  \draw (0,0) arc (180 :0: 0.5cm and -0.2cm);
\end{scope}

}}
    \ 
   -\   \overline{\tthree} \cdot\ \NB{\tikz[font=\tiny, scale=1.2]{\begin{scope}
  \draw (0,0) arc (180 :0: 0.5cm and 0.2cm) node[above, pos =
  0.5] {$a$};
  \draw[very thin] (0,0) arc (180 :0: 0.5cm and -0.6cm) node[pos=0.5,
  above] {$\dotnewtoni[1] \wdotnewtoni[0]$};
  \draw (0,0) arc (180 :0: 0.5cm and -0.2cm);
\end{scope}

}}  \\[3pt]
\label{eq:e-act-cap}   \df\left( \NB{\tikz[font=\tiny, scale=1.2]{}}\right)\ =\
    -\overline{\tthree} \cdot\ \NB{\tikz[font=\tiny,
    scale=1.2]{\begin{scope}
  \draw (0,0) arc (180 :0: 0.5cm and 0.2cm);
  \draw[very thin] (0,0) arc (180 :0: 0.5cm and 0.6cm) node[pos=0.5,
  below] {$\dotnewtoni[0] \wdotnewtoni[1]$};
  \draw (0,0) arc (180 :0: 0.5cm and -0.2cm)node[ below, pos =
  0.5] {$a$};
\end{scope}

}} \ 
  - \  \tthree \cdot\ \NB{\tikz[font=\tiny, scale=1.2]{\begin{scope}
  \draw (0,0) arc (180 :0: 0.5cm and 0.2cm);
  \draw[very thin] (0,0) arc (180 :0: 0.5cm and 0.6cm) node[pos=0.5,
  below] {$\dotnewtoni[1] \wdotnewtoni[0]$};
  \draw (0,0) arc (180 :0: 0.5cm and -0.2cm)node[ below, pos =
  0.5] {$a$};
\end{scope}

}} 
\end{gather}

Note that using the embedding of $\sll_2$ in $\ourwitt$ of
Lemma~\ref{lem:sl2-to-witt}, the relations
between $s$, $\parone{}$, $\partwo{}$ and $\parthree{}$ are given by:
\begin{equation}
  \label{eq:19}
  t_1=\parone{1}+s, \qquad t_2=\partwo{1}+s,\qquad \textrm{and} \qquad
  t_3
  =\parthree{1} +\frac{1}{2}.
\end{equation}

The same proof as that of Lemma~\ref{lem:witt-acts-good-position}
gives the following.
\begin{lem}\label{lem:sl2-acts-good-position}
  Mapping $\Le$ to $\de$, $\Lh$ to $\dh$ and $\Lf$ to $\df$ defines an
  action of $\sll_2$ on the $\scalars$-module generated by
  \spherical{} foams in good position.
\end{lem}

The same proofs as those of Proposition \ref{prop:witt-acts-good-pos} and Theorem~\ref{thm:wittaction-spherical}
give the following.

\begin{prop}\label{prop:sl2action-spherical}
  For any web $\web$, the operators $\{\de, \dh, \df\}$ induce an
  action of $\sll_2$ on $\estatespaceN[\scalars, s]{\web}$.\end{prop}

\begin{rmk}\label{rmk:coefficients}
  In contrast with the operators $\Ln$, the definition of the action of  $\de$, $\dh$ and
  $\df$, do not require $2$ to be invertible in
  $\scalars$. Lemma~\ref{lem:sl2-acts-good-position} and
  Proposition~\ref{prop:sl2action-spherical} remain valid
  without this assumption. 
\end{rmk}

\subsection{$p$-DG structure}
\label{sec:pdg}

In this section, we fix $p$ a prime number and we assume that
$\scalars=\Fp$. We aim to endow $\gll_N$-state spaces with a
$p$-DG-structure, that is an $H$-module structure (see Section~\ref{sec:p-dg-structure}).
Namely, we will establish the following proposition.

\begin{prop}\label{prop:pDG-structure}
  For any web $\web$, mapping $\dif$ to $\df$ endows the state space
  $\estatespaceN[\Fp, s]{\web}$ with an $H$-module structure. 
\end{prop}

In what follows, we will denote $\df$ by $\dif$ in order to facilitate
the reading. 

\begin{proof}
  We already know that the action of $\dif$ on $\statespaceN[s]{\web}$
  is well-defined. We only need to show that $\dif^p$ acts trivially on
  $\estatespaceN[\Fp, s]{\web}$. By the Leibniz rule,
in characteristic $p$, it
  suffices to show that $\dif^p(\foam) = 0$ for any basic foam.

  For traces of isotopies, associativity and coassociativity, this is
  obvious since $\dif$ acts trivially on these basic foams. For
  polynomials, this follows from the computaton
  $\dif^p(z)=p! z^{p+1}=0 $.
  We now consider the remaining basic foams. The computations for all
  of these are similar and
  can be done as in the proof \cite[Lemma 3.10]{QRSW1}. We give here an alternative approach and focus
  first on the the digon-cup. 
  Denote $\Fp[x_1,\dots,x_a, y_1,\dots, y_b]^{S_a\times S_b}$ by $A_{(a,b)}$,
  and consider the set $M$ of linear combinations of foams which are decorated, (we
  restrict here to ``classical'' decorations) 
  \[
    \NB
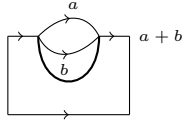
{\tikz[scale=0.8, font= \tiny]{}}
  \]
  up to dot migration (see Example~\ref{exa:dot-migration}). Since dot
  migration is a local relation in state space, it is enough to show
  that $\dif^p$ acts trivially on $M$.  The
  $\Fp$-module $M$ is naturally endowed with an
  $A_{(a,b)}$-module structure where symmetric polynomials in the
  first $a$ variables act on the $a$-thick facet and polynomials in the
  last $b$ variables act on the $b$-thick facet. The module $M$ is
  free of rank one and has a generator $1_M$ which is the foam
    \[
    \NB
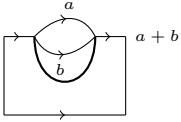
{\tikz[scale=0.8, font= \tiny]{}}.
  \]
  For $P\in A_{(a,b)}$, the element $P\cdot 1_M$ of $M$ is denoted
  $P_M$. Recall that $A_{(a,b)}$ is endowed with an $H$-module
  structure by setting
  \begin{equation}\dif(P) = \df(P) = \sum_{i=1}^{a}
    x_i^2\frac{\partial}{\partial x_i} P + \sum_{i=1}^{b}
    y_i^2\frac{\partial}{\partial y_i}P. \label{eq:dif-acts-on-Aab}
  \end{equation}
  Recall that $\dotnewtoni[0]$ on a facet of thickness $a$ equal the integer $a$.
  Following \eqref{eq:e-act-dig-cup}, we  endow $M$ with an $H'$-module structure by imposing 
\begin{equation}
  \dif(P_M) = \dif(P) - \left(b\tone  p_1(x_1, \dots x_a) + a\ttwo
    p_1(y_1, \dots y_b) \right)P . \label{eq:1}
\end{equation}
In other words, the action of $\dif$ on $M$ is that on
$A_{(a,b)}$ twisted by
\[-\left( b\tone p_1(x_1, \dots x_a) + a\ttwo p_1(y_1, \dots y_b)
\right).\]
Thus, the 
twist we imposed on $M$ precisely matches the definition of $\dif =\df$ on
this basic foam. Hence proving that $\dif^p(1_M) =0$ is enough for our
purposes. The fact that $\dif^p(1_M) =0$ follows directly from
Lemma~\ref{lem:twist-pDG-alg} and Remark~\ref{rmk:twist-works-on-invariants}.

The analysis for the other basic foams is similar. However note that for
cups and caps, one considers the algebra $A_{(a, N-a)}$. \end{proof}

In contrast with other actions we have discussed so far, this
extends to the non-equivariant setting.
See \cite[Proposition 3.8]{Wang} for conditions where the action of $\de $ extends in the non-equivariant setting.
\begin{cor}\label{cor:pDG-structure}
  For any web $\web$, mapping $\dif$ to $\df$ endows the state space
  $\zstatespaceN[\Fp,s]{\web}$ with an $H$-module structure. \end{cor}

\begin{proof}
In order to prove that the
  action of $\dif$ is well-defined on $\zstatespaceN[\Fp,s]{\web}$,
  one should prove that if a (linear combination of) \spherical{}
  foam(s) $\foam$ is equal to $0$ in $\zstatespaceN[\Fp,s]{\web}$, then
  $\dif(\foam) =0$ in $\zstatespaceN[\Fp,s]{\web}$. We will prove the
  contrapositive. Suppose that $\dif \foam \neq 0$ in
  $\zstatespaceN[\Fp,s]{\Gamma}$. This means that there exists a
  \spherical{} foam $G\co\Gamma \to \emptyset$ such that
  $\varphi_0(\kup{G\circ \dif(\foam)})\neq 0 \in \Fp$, where
  $\varphi_0\co \RN \to \Fp$ is the unique ring morphism mapping
  $E_i$ to $0$ for all $1\leq i \leq \myN$. We can suppose that both
  $\foam$ and $G$ are homogeneous. Since
  $\varphi_0(\kup{G\circ \dif\foam})\neq 0$ and $\dif$ is of degree
  $2$, one necessarily has $\deg \foam + \deg G = -2$ and in particular
  $\kup{G\circ \foam} =0$. Since the evaluation commutes with $\partial$, this implies that
  $\kup{G\circ \dif(\foam)} = - \kup{\dif(G) \circ \foam}$ and
  therefore that $\varphi_0(\kup{\dif(G) \circ \foam})\neq 0$ and
  finally, that $\foam$ is not zero in $\zstatespaceN[\Fp,s]{\web}$.
\end{proof}

\subsection{Saddles}
\label{sec:saddles}

When dropping the \spherical{} conditions on foams,
Lemma~\ref{lem:cupcap-mono} is not valid anymore and this forces us to
set $\parthree{n} = 0$ for all $n\in \NNN$ (or $\tthree=\frac{1}{2}$ for the $\sll_2$ and $p$-DG
cases). Therefore there is no option but to invert $2$ in order to
have an $\sll_2$-action or a $p$-DG structure on state spaces. We give
details below.

For $n\in \NNN,$ define 
\begin{align}
  \Ln\left( \NB{\tikz[font = \tiny]{\begin{scope}[scale=0.6]
\tdplotsetmaincoords{70}{25}
\begin{scope}[scale = 1.5, tdplot_main_coords]
  \tikzset{yxplane/.style={canvas is xy plane at z=#1}}
  \begin{scope}[yxplane=1]
    \coordinate (AT) at ({cos(  45)}, {sin(  45)});
    \coordinate (BT) at ({cos(135)}, {sin(135)});
    \coordinate (CT) at ({cos(225)}, {sin(225)});
    \coordinate (DT) at ({cos(315)}, {sin(315)});
    \coordinate (aT) at (  0.3, 0);
    \coordinate (bT) at ( -0.3, 0);
    \draw (AT) .. controls (aT) and (bT) .. (BT)
    coordinate[pos=0.5] (eT) node [pos =0.5, above] {$a$};
    \draw (DT) .. controls (aT) and (bT) .. (CT) coordinate[pos=0.5] (fT);
  \end{scope}

  \begin{scope}[yxplane=0]
    \coordinate (AM) at ({cos(  45)}, {sin(  45)});
    \coordinate (BM) at ({cos(135)}, {sin(135)});
    \coordinate (CM) at ({cos(225)}, {sin(225)});
    \coordinate (DM) at ({cos(315)}, {sin(315)});
    \coordinate (aM) at (0, 0.3);
    \coordinate (bM) at (0, -0.3);
    \draw (AM) .. controls (aM) and (bM) .. (DM) coordinate[pos=0.5] (eM);
    \draw (BM) .. controls (aM) and (bM) .. (CM) coordinate[pos=0.5] (fM);
\end{scope}
  \coordinate (OT) at (0,0, 0.5);
\draw (AM) -- (AT);
  \draw (BM) -- (BT);
  \draw (CM) -- (CT);
  \draw (DM) -- (DT);
  \draw[densely dotted] (eM) ..controls +(0,0,0.2) and + (0.1,0,0).. (OT);
  \draw[densely dotted] (fM) ..controls +(0,0,0.2) and + (-0.1,0,0).. (OT);
  \draw[densely dotted] (eT) ..controls +(0,0,-0.2) and + (0,0.1,0).. (OT);
  \draw[densely dotted] (fT) ..controls +(0,0,-0.2) and + (0,-0.1,0).. (OT);
\end{scope}  
\end{scope}

}}\right)
  &= - \frac{1}{2}\sum_{k+\ell=n} \NB{\tikz[font = \tiny]{\begin{scope}[scale=0.6]
\tdplotsetmaincoords{70}{25}
\begin{scope}[scale = 1.5, tdplot_main_coords]
  \tikzset{yxplane/.style={canvas is xy plane at z=#1}}
  \begin{scope}[yxplane=1]
    \coordinate (AT) at ({cos(  45)}, {sin(  45)});
    \coordinate (BT) at ({cos(135)}, {sin(135)});
    \coordinate (CT) at ({cos(225)}, {sin(225)});
    \coordinate (DT) at ({cos(315)}, {sin(315)});
    \coordinate (aT) at (  0.3, 0);
    \coordinate (bT) at ( -0.3, 0);
    \draw (AT) .. controls (aT) and (bT) .. (BT)
    coordinate[pos=0.5] (eT) node [pos =0.5, above] {$a$};
    \draw (DT) .. controls (aT) and (bT) .. (CT) coordinate[pos=0.5] (fT);
  \end{scope}
  \node at (-0.75, -0.1, 0.3) {$\dotnewtoni[k]$};
    \node at (0.5, 0.5, 0.3) {$\wdotnewtoni[\ell]$};
  \begin{scope}[yxplane=0]
    \coordinate (AM) at ({cos(  45)}, {sin(  45)});
    \coordinate (BM) at ({cos(135)}, {sin(135)});
    \coordinate (CM) at ({cos(225)}, {sin(225)});
    \coordinate (DM) at ({cos(315)}, {sin(315)});
    \coordinate (aM) at (0, 0.3);
    \coordinate (bM) at (0, -0.3);
    \draw (AM) .. controls (aM) and (bM) .. (DM) coordinate[pos=0.5] (eM);
    \draw (BM) .. controls (aM) and (bM) .. (CM) coordinate[pos=0.5] (fM);
\end{scope}
  \coordinate (OT) at (0,0, 0.5);
\draw (AM) -- (AT);
  \draw (BM) -- (BT);
  \draw (CM) -- (CT);
  \draw (DM) -- (DT);
  \draw[densely dotted] (eM) ..controls +(0,0,0.2) and + (0.1,0,0).. (OT);
  \draw[densely dotted] (fM) ..controls +(0,0,0.2) and + (-0.1,0,0).. (OT);
  \draw[densely dotted] (eT) ..controls +(0,0,-0.2) and + (0,0.1,0).. (OT);
  \draw[densely dotted] (fT) ..controls +(0,0,-0.2) and + (0,-0.1,0).. (OT);
\end{scope}  
\end{scope}

}}
  \\[3pt]
\Ln\left( \NB{\tikz[font=\tiny, scale=1.2]{}}\right)\ =\
  &  \frac{1}{2} \sum_{k+\ell=n} \NB{\tikz[font=\tiny, scale=1.2]{}}
  \\[3pt]
\Ln\left( \NB{\tikz[font=\tiny, scale=1.2]{}}\right)\ =\
 &\frac{1}{2} \sum_{k+\ell=n}\NB{\tikz[font=\tiny,
 scale=1.2]{}}
\end{align}
and recycle the definition of $\Ln$ on other basic foams given by \eqref{eq:Ln-on-pol}--\eqref{eq:Ln-on-unzip}.

The same proofs as those of Lemma~\ref{lem:witt-acts-good-position}
and Theorem~\ref{thm:wittaction-spherical} give the following proposition.

\begin{thm}\label{thm:witt-acts-nons-spherical}
  For any web \web, mapping $\LLn$ to $\Ln$ defines an action of $\ourwitt$ on the
  (not necessarily \spherical{}) state space $\estatespaceN[\scalars]{\web}$.
\end{thm}

Define 
\begin{align}
  \de\left( \NB{\tikz[font = \tiny]{}}\right)\ =\ 
    \de\left( \NB{\tikz[font=\tiny, scale=1.2]{}}\right)\ =\
  \de\left( \NB{\tikz[font=\tiny, scale=1.2]{}}\right)\ =\ 0,
\end{align}
\begin{align}
  \df\left( \NB{\tikz[font = \tiny]{}}\right)
  &= \frac{1}{2}\cdot\ \NB{\tikz[font = \tiny]{\begin{scope}[scale=0.6]
\tdplotsetmaincoords{70}{25}
\begin{scope}[scale = 1.5, tdplot_main_coords]
  \tikzset{yxplane/.style={canvas is xy plane at z=#1}}
  \begin{scope}[yxplane=1]
    \coordinate (AT) at ({cos(  45)}, {sin(  45)});
    \coordinate (BT) at ({cos(135)}, {sin(135)});
    \coordinate (CT) at ({cos(225)}, {sin(225)});
    \coordinate (DT) at ({cos(315)}, {sin(315)});
    \coordinate (aT) at (  0.3, 0);
    \coordinate (bT) at ( -0.3, 0);
    \draw (AT) .. controls (aT) and (bT) .. (BT)
    coordinate[pos=0.5] (eT) node [pos =0.5, above] {$a$};
    \draw (DT) .. controls (aT) and (bT) .. (CT) coordinate[pos=0.5] (fT);
  \end{scope}
  \node at (-0.75, -0.1, 0.3) {$\dotnewtoni[0]$};
    \node at (0.5, 0.5, 0.3) {$\wdotnewtoni[1]$};
  \begin{scope}[yxplane=0]
    \coordinate (AM) at ({cos(  45)}, {sin(  45)});
    \coordinate (BM) at ({cos(135)}, {sin(135)});
    \coordinate (CM) at ({cos(225)}, {sin(225)});
    \coordinate (DM) at ({cos(315)}, {sin(315)});
    \coordinate (aM) at (0, 0.3);
    \coordinate (bM) at (0, -0.3);
    \draw (AM) .. controls (aM) and (bM) .. (DM) coordinate[pos=0.5] (eM);
    \draw (BM) .. controls (aM) and (bM) .. (CM) coordinate[pos=0.5] (fM);
\end{scope}
  \coordinate (OT) at (0,0, 0.5);
\draw (AM) -- (AT);
  \draw (BM) -- (BT);
  \draw (CM) -- (CT);
  \draw (DM) -- (DT);
  \draw[densely dotted] (eM) ..controls +(0,0,0.2) and + (0.1,0,0).. (OT);
  \draw[densely dotted] (fM) ..controls +(0,0,0.2) and + (-0.1,0,0).. (OT);
  \draw[densely dotted] (eT) ..controls +(0,0,-0.2) and + (0,0.1,0).. (OT);
  \draw[densely dotted] (fT) ..controls +(0,0,-0.2) and + (0,-0.1,0).. (OT);
\end{scope}  
\end{scope}

}}\ + \
    \frac{1}{2}\cdot\ \NB{\tikz[font = \tiny]{\begin{scope}[scale=0.6]
\tdplotsetmaincoords{70}{25}
\begin{scope}[scale = 1.5, tdplot_main_coords]
  \tikzset{yxplane/.style={canvas is xy plane at z=#1}}
  \begin{scope}[yxplane=1]
    \coordinate (AT) at ({cos(  45)}, {sin(  45)});
    \coordinate (BT) at ({cos(135)}, {sin(135)});
    \coordinate (CT) at ({cos(225)}, {sin(225)});
    \coordinate (DT) at ({cos(315)}, {sin(315)});
    \coordinate (aT) at (  0.3, 0);
    \coordinate (bT) at ( -0.3, 0);
    \draw (AT) .. controls (aT) and (bT) .. (BT)
    coordinate[pos=0.5] (eT) node [pos =0.5, above] {$a$};
    \draw (DT) .. controls (aT) and (bT) .. (CT) coordinate[pos=0.5] (fT);
  \end{scope}
  \node at (-0.75, -0.1, 0.3) {$\dotnewtoni[1]$};
    \node at (0.5, 0.5, 0.3) {$\wdotnewtoni[0]$};
  \begin{scope}[yxplane=0]
    \coordinate (AM) at ({cos(  45)}, {sin(  45)});
    \coordinate (BM) at ({cos(135)}, {sin(135)});
    \coordinate (CM) at ({cos(225)}, {sin(225)});
    \coordinate (DM) at ({cos(315)}, {sin(315)});
    \coordinate (aM) at (0, 0.3);
    \coordinate (bM) at (0, -0.3);
    \draw (AM) .. controls (aM) and (bM) .. (DM) coordinate[pos=0.5] (eM);
    \draw (BM) .. controls (aM) and (bM) .. (CM) coordinate[pos=0.5] (fM);
\end{scope}
  \coordinate (OT) at (0,0, 0.5);
\draw (AM) -- (AT);
  \draw (BM) -- (BT);
  \draw (CM) -- (CT);
  \draw (DM) -- (DT);
  \draw[densely dotted] (eM) ..controls +(0,0,0.2) and + (0.1,0,0).. (OT);
  \draw[densely dotted] (fM) ..controls +(0,0,0.2) and + (-0.1,0,0).. (OT);
  \draw[densely dotted] (eT) ..controls +(0,0,-0.2) and + (0,0.1,0).. (OT);
  \draw[densely dotted] (fT) ..controls +(0,0,-0.2) and + (0,-0.1,0).. (OT);
\end{scope}  
\end{scope}

}},
  \\[3pt]
\df\left( \NB{\tikz[font=\tiny, scale=1.2]{}}\right)\ =\
  &  -\frac{1}{2}\cdot\  \NB{\tikz[font=\tiny, scale=1.2]{}}\
    - \ \frac{1}{2}\cdot\  \NB{\tikz[font=\tiny, scale=1.2]{}},
  \\[3pt]
\df\left( \NB{\tikz[font=\tiny, scale=1.2]{}}\right)\ =\
  & - \frac{1}{2}\cdot\   \NB{\tikz[font=\tiny, scale=1.2]{}}\
    - \ \frac{1}{2}\cdot\  \NB{\tikz[font=\tiny, scale=1.2]{}},
\end{align}
\begin{align}
  \dh\left( \NB{\tikz[font = \tiny]{}}\right)
  &= -a(N-a)\cdot\ \NB{\tikz[font = \tiny]{}},
  \\[3pt]
\dh\left( \NB{\tikz[font=\tiny, scale=1.2]{}}\right)\ =\
  &  a(N-a)\cdot\  \NB{\tikz[font=\tiny, scale=1.2]{}},\
  \\[3pt]
\dh\left( \NB{\tikz[font=\tiny, scale=1.2]{}}\right)\ =\
  &   a(N-a)\cdot\ \NB{\tikz[font=\tiny, scale=1.2]{}},
\end{align}
and recycle the definition of $\de, \df$ and $\dh$ on other basic foams given by \eqref{eq:e-act-pol}--\eqref{eq:e-act-unzip}.

Again, the same proofs as those of
Lemma~\ref{lem:witt-acts-good-position},
Theorem~\ref{thm:wittaction-spherical} and
Corollary~\ref{cor:pDG-structure} give the following.

\begin{prop}\label{prop:sl2-action-non-spherical}
  For any web $\web$, mapping $\Le, \Lf$ and $\Lh$ to $\de, \df$ and
  $\dh$ respectively defines an action of $\sll_2$ on the
  not necessarily \spherical{} state space $\estatespaceN[\scalars]{\web}$.
\end{prop}

\begin{prop}\label{prop:pDG-structure-nons-perical}
  Let $p>2$ be a prime number. For any web $\web$, mapping $\dif$ to $\df$ endows the state space
  $\estatespaceN[\Fp]{\web}$ with an $H$-module structure. 
\end{prop}

\begin{cor}\label{cor:pDG-structure-non-spherical}
  Let $p>2$ be a prime number. For any web $\web$, mapping $\dif$ to
  $\df$ endows the state space $\zstatespaceN[\Fp]{\web}$ with an
  $H$-module structure.
\end{cor}

\begin{rmk}
  If $\web_1$ and $\web_2$ are two webs, then since $\etqftfunc[\scalars]$ is
  monoidal, $\estatespaceN[\scalars]{\web_1 \sqcup \web_2}$ is naturally
  isomorphic to $\estatespaceN[\scalars]{\web_1} \otimes
  \estatespaceN[\scalars]{\web_2}$. Then
  by its very construction, the action of $\ourwitt$ on
  $\estatespaceN[\scalars]{\web_1 \sqcup \web_2}$ satisfies the Leibniz rule in
  the sense that for any $n \in \NNN$ and any $\foam_1 \otimes
  \foam_2$ in $\estatespaceN[\scalars]{\web_1} \otimes \estatespaceN[\scalars, s]{\web_2}$,
  \[
    \LLn \cdot (\foam_1 \otimes \foam_2) = (\LLn\cdot \foam_1) \otimes
    \foam_2 + \foam_1 \otimes (\LLn \cdot \foam_2).
  \]
  The same is true for the action of $\sll_2$ and for that of $H$.
\end{rmk}

\subsection{Related work}
\label{sec:comp-lit}

The aim of this section is to state how to tune parameters to recover
structures already present in the literature in the context of
Soergel bimodules and their Hochschild homologies. Since the equivalence
of bi-categories between Soergel bimodules (in type $\mathsf{A}$) and foams subject to
ad-hoc relations is not formally established yet (details will appear in
\cite{KRW1}), we do not aim to be very precise here. In the context of
Soergel bimodules, there are no saddles nor caps nor cups. So namely we
have to deal with parameters $s$, $\parone{}$, $\partwo{}$ for
$\ourwitt$ and $t_1$ and $t_2$ for $\sll_2$. In the context of Soergel
bimodules, the zip and unzip foams correspond to bimodule
homomorphisms, that are typically denoted by $\rb$ and $\br$ respectively.

We spotted three places with various conventions \cite{KRWitt},
\cite{EQ3, EliasQisl2}  and \cite{QiSussanLink, QRSW1} which
fit into this context. 

As already mentioned, Khovanov and Rozansky \cite{KRWitt} exhibited a
(half) Witt action on triply-graded link homology. They use a slightly
different presentation of the Witt algebra
($\LLn^{\text{\cite{KRWitt}}} \leftrightarrow -\LLn$). The maps $\rb$
and $\br$ are denoted $\chi_+$ and $\chi_-$ respectively (see
\cite[Equation (3.8)]{KRWitt}). The fact that the map $\chi_-$ has no
twist in its definition, corresponds to the choice $\parone{n}
=\partwo{n} = s =0$ for all $n \in \NNN$.

The structures considered in \cite{EQ3, EliasQisl2} deal with more
general Soergel bimodules than the type $\mathsf{A}$ Soergel bimodules
to which we can relate to with foams. Elias and Qi endow these general Soergel bimodules 
with an $\sll_2$ action (in their language, $\Le = \mathbf{d}$ and
$\Lf = \mathbf{z}$). Restricting to the type $\mathsf{A}$ case, the
formulas \cite[Equations (4.1a--i)]{EliasQisl2} correspond to the setting $t_1=1$ and $t_2=0$.

It would be interesting to relate the $\sll_2$-actions on foams to the $\sll_2$-actions on quantum groups \cite{EliasQisl2} via categorical skew Howe duality \cite{LQR, QueffelecRoseFoam}.

The papers \cite{QiSussanLink, QRSW1} are only concerned with $p$-DG
structure. The absence of twist in formula of \cite[Lemma 3.3 (ii)]{QiSussanLink}
corresponds to setting $t_1=t_2 =0$.

\bibliographystyle{alphaurl}
\bibliography{biblio}

\end{document}